\documentclass{article}
\usepackage{amsmath,amsthm,amssymb}
\usepackage{enumerate}
\usepackage{graphicx}
\usepackage{cite}
\usepackage{comment}
\usepackage{oands}
\usepackage{tikz} 
\usepackage{changepage}
\usepackage{tabularx}
\usepackage{multirow}
\usepackage{bbold}
\usepackage[margin=1in]{geometry}
\usepackage{hyperref}
\usepackage[pagewise,mathlines]{lineno}  
\usepackage[normalem]{ulem} 
\usepackage{subfigure}  
\usepackage[all]{xy}

\usepackage{array}

\newcommand{\specialcell}[1]{\begin{tabular}{c}#1\end{tabular}}

\usepackage[protrusion=true,expansion=true]{microtype}

\numberwithin{equation}{section}

\theoremstyle{plain}
\newtheorem{thm}{Theorem}[section]
\newtheorem{cor}[thm]{Corollary}
\newtheorem{lem}[thm]{Lemma}
\newtheorem{prop}[thm]{Proposition}

\newtheorem{notation}[thm]{Notation}
\newtheorem{example}[thm]{Example}
 
\theoremstyle{definition}
\newtheorem{defn}[thm]{Definition}
\newtheorem{remark}[thm]{Remark}

\def\@rst #1 #2other{#1}
\newcommand\MR[1]{\relax\ifhmode\unskip\spacefactor3000 \space\fi
  \MRhref{\expandafter\@rst #1 other}{#1}}
\newcommand{\MRhref}[2]{\href{http://www.ams.org/mathscinet-getitem?mr=#1}{MR#2}}

\hypersetup{ 
    colorlinks=false,
    linktocpage, 
    }

\newcommand{\dsb}{\begin{adjustwidth}{2.5em}{0pt}
\begin{footnotesize}}
\newcommand{\dse}{\end{footnotesize}
\end{adjustwidth}}

\newcommand{\ssb}{\begin{adjustwidth}{2.5em}{0pt}}
\newcommand{\sse}{\end{adjustwidth}}

\newcommand{\aryb}{\begin{eqnarray*}}
\newcommand{\arye}{\end{eqnarray*}}
\def\alb#1\ale{\begin{align*}#1\end{align*}}
\newcommand{\eqb}{\begin{equation}}
\newcommand{\eqe}{\end{equation}}
\newcommand{\eqbn}{\begin{equation*}}
\newcommand{\eqen}{\end{equation*}}

\newcommand{\BB}{\mathbb}
\newcommand{\ol}{\overline}
\newcommand{\ul}{\underline}
\newcommand{\op}{\operatorname}

\newcommand{\frk}{\mathfrak}
\newcommand{\eqD}{\overset{d}{=}}
\newcommand{\ep}{\epsilon}
\newcommand{\rta}{\rightarrow}

\newcommand{\wt}{\widetilde}
\newcommand{\wh}{\widehat} 
\newcommand{\mcl}{\mathcal}

\newcommand{\rng}{\mathring}
\newcommand{\bdy}{\partial}

\newcommand*\patchAmsMathEnvironmentForLineno[1]{  \expandafter\let\csname old#1\expandafter\endcsname\csname #1\endcsname
  \expandafter\let\csname oldend#1\expandafter\endcsname\csname end#1\endcsname
  \renewenvironment{#1}     {\linenomath\csname old#1\endcsname}     {\csname oldend#1\endcsname\endlinenomath}}\newcommand*\patchBothAmsMathEnvironmentsForLineno[1]{  \patchAmsMathEnvironmentForLineno{#1}  \patchAmsMathEnvironmentForLineno{#1*}}\AtBeginDocument{\patchBothAmsMathEnvironmentsForLineno{equation}\patchBothAmsMathEnvironmentsForLineno{align}\patchBothAmsMathEnvironmentsForLineno{flalign}\patchBothAmsMathEnvironmentsForLineno{alignat}\patchBothAmsMathEnvironmentsForLineno{gather}\patchBothAmsMathEnvironmentsForLineno{multline}}

\newcommand{\SLE}{{\rm SLE}}
\newcommand{\CLE}{{\rm CLE}}
\newcommand{\Z}{{\mathbb Z}}
\newcommand{\D}{{\mathbb D}}
\newcommand{\C}{{\mathbb C}}
\newcommand{\R}{{\mathbb R}}
\newcommand{\N}{{\mathbb N}}

\begin{document}

\author{
\begin{tabular}{c}Ewain Gwynne\\[-5pt]\small Cambridge\end{tabular}\;
\begin{tabular}{c}Nina Holden\\[-5pt]\small ETH Z\"urich\end{tabular}\;
\begin{tabular}{c}Jason Miller\\[-5pt]\small Cambridge\end{tabular}}

\title{An almost sure KPZ relation for SLE and Brownian motion}
\date{   }

\maketitle 

\begin{abstract} 
The peanosphere construction of Duplantier, Miller, and Sheffield provides a means of representing a $\gamma$-Liouville quantum gravity (LQG) surface, $\gamma \in (0,2)$, decorated with a space-filling form of Schramm's $\SLE_\kappa$, $\kappa = 16/\gamma^2 \in (4,\infty)$, $\eta$ as a gluing of a pair of trees which are encoded by a correlated two-dimensional Brownian motion $Z$.  We prove a KPZ-type formula which relates the Hausdorff dimension of any Borel subset $A$ of the range of $\eta$ which can be defined as a function of $\eta$ (modulo time parameterization) to the Hausdorff dimension of the corresponding time set $\eta^{-1}(A)$.  This result serves to reduce the problem of computing the Hausdorff dimension of any set associated with an $\SLE$, $\CLE$, or related processes in the interior of a domain to the problem of computing the Hausdorff dimension of a certain set associated with a Brownian motion.  For many natural examples, the associated Brownian motion set is well-known.  As corollaries, we obtain new proofs of the Hausdorff dimensions of the SLE$_\kappa$ curve for $\kappa \not=4$; the double points and cut points of SLE$_\kappa$ for $\kappa  >4$; and the intersection of two flow lines of a Gaussian free field. We obtain the Hausdorff dimension of the set of $m$-tuple points of space-filling SLE$_\kappa$ for $\kappa>4$ and $m \geq 3$ by computing the Hausdorff dimension of the so-called $(m-2)$-tuple $\pi/2$-cone times of a correlated planar Brownian motion. 
\end{abstract}
 
\tableofcontents

\parindent 0 pt
\setlength{\parskip}{0.15cm plus1mm minus1mm}

\section{Introduction}
\label{sec-intro}

\subsection{Overview}

The Schramm-Loewner evolution ($\SLE_\kappa$) \cite{schramm0} and related processes such as $\SLE_\kappa(\ul\rho)$ \cite{lsw-restriction,sw-coord,ig1} and the conformal loop ensembles ($\CLE_\kappa$) \cite{shef-cle,shef-werner-cle} have been an active area of research for the past two decades.  One line of research in this area has been the confirmation of exponents computed non-rigorously by physicists in the context of discrete models from statistical physics.  Many of these exponents were derived using the so-called \emph{KPZ relation} \cite{kpz-scaling}, which is a non-rigorous formula which relates exponents for statistical physics models on random planar maps to the corresponding exponents for the model on a Euclidean lattice, such as $\Z^2$.  Exponents derived in this way are said to be obtained from ``quantum gravity methods.''  This method of deriving exponents has been very successful because the computation of an exponent in many cases boils down to a counting problem which turns out to be much easier when the underlying lattice is random (i.e., one considers a random planar map).  Perhaps the most famous example of this type are the so-called \emph{Brownian intersection exponents}, which give the exponent of the probability that $k$ Brownian motions started on $\partial B_\ep(0)$ at distance proportional to $\epsilon$ from each other make it to $\partial \D$ without any of their traces intersecting.  These exponents were derived using quantum gravity methods by Duplantier\footnote{The Brownian intersection exponents were also derived earlier using a different method by Duplantier and Kwon in \cite{duplantier-kwon-brownian}.} in~\cite{duplantier-brownian}.  The values of the Brownian intersection exponents were then verified mathematically in one of the early successes of $\SLE$ by Lawler, Schramm, and Werner in \cite{lsw-brownian1,lsw-brownian2,lsw-brownian3}.  Following these works, a number of other exponents (hence also Hausdorff dimensions) have been calculated using $\SLE$ techniques, many of which were previously predicted in the physics literature \cite{abv-bdy-spec,alberts-shef-bdy-dim,beffara-dim,gms-mf-spec,lawler-viklund-tip,msw-gasket,miller-wu-dim,mww-extremes,nacu-werner11,schramm-sle,ssw09,schoug-boundary-spec,wang-wu-sle-bdy,wu-zhan-arm-exponent,wu-poly-arm-exponent,wu-alt-arm-exponent}.

Our main result is a rigorous version of the KPZ formula that relates the a.s.\ Hausdorff dimension of a set associated with space-filling\footnote{In order to be consistent with the notation of \cite{ig1,ig2,ig3,ig4}, unless explicitly stated otherwise we will assume that $\kappa \in (0,4)$ and $\kappa' = 16/\kappa \in (4,\infty)$.} $\SLE_{\kappa'}$~\cite{ig4}, $\kappa' \in (4,\infty)$, to the a.s.\ Hausdorff dimension of a certain Brownian motion set in the context of the so-called peanosphere construction of \cite{wedges}, which we review below.  This serves to reduce the problem of calculating the Hausdorff dimension of any set associated with $\SLE_\kappa$, $\SLE_\kappa(\ul\rho)$, or $\CLE_\kappa$ for $\kappa\not=4$ in the interior of a domain to the problem of calculating the Hausdorff dimension of a certain (explicitly described) set associated with a correlated two-dimensional Brownian motion. There are numerous formulations of the KPZ formula in the literature, see e.g.\ the original physics paper~\cite{kpz-scaling}, in addition to more recent and rigorous formulations in e.g.\ \cite{aru-kpz,grv-kpz,bjrv-gmt-duality,benjamini-schramm-cascades,wedges,shef-renormalization,shef-kpz,rhodes-vargas-log-kpz,gp-kpz}.  As explained just above, the KPZ formula is typically applied to compute the Euclidean dimension of fractal sets, after deriving the quantum dimension heuristically or rigorously by quantum gravity techniques. In our formulation the quantum dimension is explicitly given by the dimension of some Brownian motion set, hence our formula is directly useful for computations.

To illustrate the application of our main theorem, we will obtain new proofs of the a.s.\ Hausdorff dimensions of several sets, including the SLE curve for $\kappa\neq 4$, the double points of SLE, the cut points of SLE, and the intersection of two flow lines of a Gaussian free field \cite{shef-zipper,ig1,ig2,ig3,ig4}.
We will also use our theorem to calculate the a.s.\ Hausdorff dimension of the $m$-tuple points of space-filling $\SLE_{\kappa'}$, $\kappa'\in(4,8)$, by calculating the Hausdorff dimension of the so-called $(m-2)$-tuple cone times for correlated two-dimensional Brownian motion. The statement and the proof of the dimension result for $(m-2)$-tuple cone times do not rely on quantum gravity techniques or results in the remainder of the paper.

The main motivation of our theorem is to convert SLE dimension questions into Brownian motion dimension questions, since these are often much easier to solve. However, the result also works in the reverse direction. One example of a Brownian motion set whose dimension has not yet been directly computed to our knowledge is the set of times not contained in any left cone interval.  Under the peanosphere correspondence, this set corresponds to the $\CLE_{\kappa'}$ gasket for $\kappa'\in(4,8)$, the Hausdorff dimension of which is computed in \cite{msw-gasket,ssw09}.

In the subsequent work~\cite{ghm-conformal-dim}, the main theorem of the present paper (along with a theorem of Rhodes and Vargas~\cite{rhodes-vargas-log-kpz}) will be used to prove the following additional dimension formula for SLE$_\kappa$. If $\eta$ is an SLE$_\kappa$ curve for $\kappa \in (0,4)$, and $Y$ is a deterministic subset of $\BB R$ with Hausdorff dimension $d \in [0,1]$, then it is a.s.\ the case that 
\[
\dim_{\mcl H}  f(Y)  =  \frac{1}{32\kappa} \left(4 + \kappa - \sqrt{(4 + \kappa)^2 - 16 \kappa d} \right) \left(12 + 3 \kappa + \sqrt{(4 + \kappa)^2 - 16 \kappa d } \right)
\]
for almost every choice of conformal map $f$ from $\BB H$ to a complementary connected component of $\eta$ which satisfies $f(Y)\subset\eta$.

\subsection{Review of Liouville quantum gravity and the peanosphere}
\label{sec-review}

We will now provide a brief review of Liouville quantum gravity and the peanosphere construction which will be necessary to understand our main result below.  Suppose that $h$ is an instance of the Gaussian free field (GFF) on a planar domain $D$ and $\gamma \in (0,2)$.  The $\gamma$-Liouville quantum gravity (LQG) surface associated with $h$ formally corresponds to the surface with Riemannian metric
\begin{equation}
\label{eqn::lqg_metric}
e^{\gamma h(z)} (dx^2 + dy^2),
\end{equation}
where $z = x + iy  = (x,y)$ and $dx^2 + dy^2$ denotes the Euclidean metric on $D$.  This expression does not make literal sense because $h$ takes values in the space of distributions and does not take values at points.  The area measure $\mu_h$ associated with~\eqref{eqn::lqg_metric} has been made sense of using a regularization procedure (see e.g., \cite{shef-kpz}), namely by taking $e^{\gamma h(z)} dz$ to be the weak limit as $\epsilon \to 0$ of $\epsilon^{\gamma^2/2} e^{\gamma h_\epsilon(z)} dz$, where $h_\epsilon(z)$ is the average of $h$ on the circle $\partial B_\ep(z)$.  One can similarly define  a length measure $\nu_h$ by taking it to be the weak limit as $\epsilon \to 0$ of $\epsilon^{\gamma^2/4} e^{\gamma h_\epsilon(z)/2} dz$.  We refer to $\mu_h$ (resp.\ $\nu_h$) as the quantum area (resp.\ boundary length) measure associated with $h$.  Quantum boundary lengths are well-defined for piecewise linear segments \cite{shef-kpz}, their conformal images, and $\SLE$ type curves for $\kappa = \gamma^2$ \cite{shef-zipper}.  The metric space structure associated with~\eqref{eqn::lqg_metric} was constructed in \cite{tbm-characterization,lqg-tbm1,lqg-tbm2,lqg-tbm3,sphere-constructions} in the special case that $\gamma = \sqrt{8/3}$, in which case it is isometric to the Brownian map \cite{miermont-brownian-map,legall-uniqueness}.  
The metric space structure for general $\gamma \in (0,2)$ was recently constructed in~\cite{gm-uniqueness} (several years after this paper first appeared on the arXiv).

One of the main sources of significance of LQG is that it has been conjectured that certain forms of LQG decorated with $\SLE$ or $\CLE$ describe the scaling limit of random planar maps decorated with a statistical physics model after performing a conformal embedding, where different $\gamma$ values arise by considering different discrete models \cite{shef-kpz}. So far, this conjecture has been proven only in the case of the Tutte embedding of the $\gamma$-mated-CRT map (a discretized version of the peanosphere) for $\gamma \in (0,2)$~\cite{gms-tutte}, the case of the Cardy embedding of uniform triangulations~\cite{hs-cardy-embedding}, and the case of the Poisson-Voronoi approximation of the Brownian map~\cite{gms-poisson-voronoi}.  However, the convergence of other random planar maps decorated with a statistical physics model to LQG decorated with $\SLE$/$\CLE$ has been proved with respect to the peanosphere topology, which we will describe below \cite{shef-burger,wedges,gms-burger-cone,gms-burger-local,gms-burger-finite,sphere-constructions,gkmw-burger,kmsw-bipolar,ghs-bipolar,lsw-schnyder-wood}.  See also~\cite{gwynne-miller-saw,gwynne-miller-perc} for scaling limit results for self-avoiding walk and percolation, respectively, on random planar maps toward SLE-decorated $\sqrt{8/3}$-LQG with respect to the Gromov-Hausdorff-Prokhorov-uniform topology, a variant of the Gromov-Hausdorff topology for curve-decorated metric measure spaces~\cite{gwynne-miller-uihpq}.

If $D,\wt{D}$ are planar domains, $\varphi \colon D \to \wt D$ is a conformal map, and
\begin{equation}
\label{eqn::change_of_coordinates}
\wt{h} = h \circ \varphi^{-1} + Q \log| (\varphi^{-1})'| \quad\text{where}\quad Q = \frac{2}{\gamma} + \frac{\gamma}{2} ,
\end{equation}
then $\mu_{h}(A) = \mu_{\wt h}(\varphi(A))$ for all Borel sets $A \subseteq D$. The boundary length measure is similarly preserved under such a change of coordinates. A \emph{quantum surface} is an equivalence class of pairs $(D,h)$ where two such pairs are said to be equivalent if they are related as in~\eqref{eqn::change_of_coordinates}.  We refer to a representative $(D,h)$ of a quantum surface as an \emph{embedding} of the quantum surface.

One particular type of quantum surface which will be important in this article is the so-called $\gamma$-quantum cone.  This is an infinite volume surface which is naturally parameterized by $\C$ and is marked by two points, called $0$ and $\infty$, neighborhoods of which respectively have finite and infinite $\mu_h$-mass.  We will keep track of the extra marked points by indicating a $\gamma$-quantum cone with the notation $(\C,h,0,\infty)$.  In Section~\ref{sec-cone-prelim} below, we will describe a precise method for sampling from the law of $h$ for a particular embedding of a $\gamma$-quantum cone into $\C$.  This surface naturally arises, however, in the context of any $\gamma$-LQG surface $(D,h)$ with finite volume as follows.  Suppose that $z \in D$ is sampled from $\mu_h$.  Then the surface one obtains by adding $C$ to $h$, translating $z$ to $0$, and then rescaling so that $\mu_h$ assigns unit mass to $\D$ converges as $C \to \infty$ to a $\gamma$-quantum cone.  That is, a $\gamma$-quantum cone describes the local behavior of a $\gamma$-LQG surface near a typical point chosen from~$\mu_h$.

As explained in \cite{shef-zipper,wedges}, it is very natural to decorate a $\gamma$-LQG surface with either an $\SLE_\kappa$, $\kappa = \gamma^2$, or an $\SLE_{\kappa'}$, $\kappa' = 16/\gamma^2$.  In the case of a $\gamma$-quantum cone $(\C,h,0,\infty)$, it is particularly natural to decorate it with the space-filling $\SLE_{\kappa'}$ process $\eta'$ \cite{ig4} where $\eta'$ is first sampled independently of $h$ (as a curve modulo time parameterization), then reparameterized by quantum area so that $\mu_h(\eta'([s,t])) = t-s$ for all $s < t$, and then normalized so that $\eta'(0) = 0$.  In this setting, it is shown in \cite{wedges} that the pair $Z = (L,R)$ which, for a given time $t$, is equal to the quantum length of the left and right boundaries of $\eta'$, evolves as a correlated two-dimensional Brownian motion. Since these quantum boundary lengths are in fact always infinite, it is natural to normalize $Z$ so that $L_0 = R_0 = 0$. By \cite[Theorem~1.9]{wedges} (in the case $\kappa' \in (4,8]$) and \cite[Theorem~1.1]{kappa8-cov} (in the case $\kappa' > 8$), the variances and covariances of $L$ and $R$ are given by
\eqb
\op{Var}(L_t ) =  a |t|, \quad \op{Var}(R_t) =  a |t|, \quad \op{Cov}(L_t, R_t ) = -a \cos\theta |t|,\quad \theta=\frac{4\pi}{\kappa'} 
\label{eq-multiplept-10} 
\eqe
with $a$ a constant depending only on $\kappa'$.

One of the main results of \cite{wedges} is that the pair $(L,R)$ almost surely determines the pair consisting of the $\gamma$-quantum cone $(\C,h,0,\infty)$ and the space-filling $\SLE_{\kappa'}$ process $\eta'$.  That is, the latter is a measurable function of the former (and it is immediate from the construction that the former is a measurable function of the latter).  This is natural in the context of discrete models \cite{shef-burger} which can also be encoded in terms of an analogous such pair and, in fact, the main result of \cite{shef-burger} combined with \cite{wedges} gives the convergence of FK-decorated random planar maps to $\CLE$ decorated LQG with respect to the topology in which two surfaces are close if the aforementioned encoding functions are close.  This is the so-called \emph{peanosphere} topology.

As explained in Figure~\ref{fig-intro1}, $L$ and $R$ have the interpretation of being the contour functions associated with a pair of infinite trees, and $(\C,h,0,\infty)$ and $\eta'$ have the interpretation of being the embedding of a certain path-decorated surface into $\C$ which is generated by gluing together the pair of trees encoded by $L$, $R$ \cite{wedges}.

We remark that the construction in \cite{wedges} deals with the setting of infinite volume surfaces.  The setting of finite volume surfaces is the focus of \cite{sphere-constructions} and the corresponding convergence result in the finite volume setting is established in \cite{gms-burger-cone,gms-burger-local,gms-burger-finite}.

\subsection{Main result}
\label{sec-intro-main}

Our main result is a KPZ formula which allows one to use the representation $(L,R)$ of an $\SLE$ decorated quantum cone to compute Hausdorff dimensions for $\SLE$ and related processes.  

\begin{thm}
\label{thm-dim-relation}
Let $\kappa' > 4$ and $\gamma = 4/\sqrt{\kappa'}$. Let $(\BB C ,h, 0,\infty)$ be a $\gamma$-quantum cone and let $\eta' $ be an independent space-filling $\SLE_{\kappa'}$, parameterized by $\gamma$-quantum mass with respect to $h$ and satisfying $\eta' (0) = 0$.  Assume that $h$ has the circle average embedding (see Definition~\ref{def-circle-embedding} below). Let $X$ be a random Borel subset of $\BB C$ such that $X$ is independent from $h$ (e.g.\ $X$ could be a set which is determined by the curve $\eta' $ viewed modulo monotone reparameterization).
Almost surely, for each Borel set $\wh X \subset \BB R$ such that $\eta'(\wh X) = X$, we have  
\eqb \label{eqn-dim-relation}
\dim_{\mcl H}(X) = \left(2+\frac{\gamma^2}{2} \right) \dim_{\mcl H}(\wh X)   - \frac{\gamma^2}{2}  \dim_{\mcl H}(\wh X)^2  .
\eqe    
\end{thm}
 
A space-filling $\SLE_{\kappa'}$ encodes an entire imaginary geometry of flow lines, which in turn encodes both $\SLE_\kappa$ and $\op{SLE}_{\kappa'}$-type paths and a $\CLE_{\kappa'}$ (see~\cite{ig4} and Section~\ref{sec-sle-prelim} below). The work~\cite{cle-percolations} shows that the imaginary geometry framework also encodes a $\op{CLE}_\kappa$. Therefore, Theorem~\ref{thm-dim-relation} reduces the computation of the dimension of any set in the interior of a domain associated with $\SLE_\kappa$ or $\CLE_\kappa$ for $\kappa\not=4$ to computing the dimension of the corresponding set associated with the correlated Brownian motion $Z$. As we will discuss in Section~\ref{sec-intro-background} it is often possible to characterize special sets associated with $\SLE_{\kappa}$ and $\CLE_{\kappa}$ as time sets of $Z$ with particular properties. Examples of such sets are
\begin{enumerate}
\item $\SLE_\kappa$ curves for $\kappa\not= 4$ \cite{beffara-dim,schramm-sle}, 
\item double points of $\SLE_{\kappa'}$ for $\kappa'>4$ \cite{miller-wu-dim}, 
\item cut points of $\SLE_{\kappa'}$ for $\kappa'>4$ \cite{miller-wu-dim}, 
\item intersection of two GFF flow lines \cite{miller-wu-dim},
\item $m$-tuple points of space-filling $\SLE_{\kappa'}$ for $m\geq 3$ and $\kappa' \in (4,8)$, and
\item $\CLE_{\kappa'}$ gasket for $\kappa'\in(4,8)$ \cite{msw-gasket,ssw09}.
\end{enumerate} 
See Table~\ref{table-dimensions} for a summary of these sets and their dimensions.
We will use Theorem~\ref{thm-dim-relation} to calculate the Hausdorff dimension of the first five of these sets in Sections~\ref{sec-dim-calc} and~\ref{sec-multiple-pt}. The original proofs relied on rather technical two-point estimates for correlations, while our formula provides alternative proofs. In cases where the Hausdorff dimension of the SLE set is known, Theorem~\ref{thm-dim-relation} also gives the dimension of the corresponding Brownian motion set. We will use this direction of Theorem~\ref{thm-dim-relation} to calculate the dimension of the Brownian motion time set corresponding to the $\CLE_{\kappa'}$ gasket in Section~\ref{sec-dim-calc}.

\begin{table}[ht!]
\tabcolsep=0.15cm
\setlength\extrarowheight{4pt}
\begin{center}
\begin{tabular}{|c|c|c|c|} 
\hline
\textbf{SLE set}& \textbf{SLE $\op{dim}_{\mcl H}$} & \textbf{BM set} & \textbf{BM $\op{dim}_{\mcl H}$}\\ \hline
\specialcell{SLE$_{\kappa}$ trace,\\[-1ex] $\kappa\in (0,4)$} & $\displaystyle{1+\frac{\kappa}{8}}$ & Running infima of $L$ or $R$ & $\displaystyle{\frac12}$ \\    \hline
\specialcell{SLE$_{\kappa'}$ trace,\\[-1ex] $\kappa'\in (4,8)$} & $ \displaystyle{1+\frac{\kappa'}{8}} $ & Ancestor free times & $\displaystyle{\frac{\kappa'}{8}}$ \\     \hline
\specialcell{Cut points of\\[-1ex] SLE$_{\kappa'}$, $\kappa'\in (4,8)$} & $\displaystyle{3-\frac{3}{8} \kappa'}$ & \specialcell{Simultaneous running\\[-1ex] infima of $L$ and $R$} & $\displaystyle{1-\frac{\kappa'}{8}}$ \\    \hline
\specialcell{Double points of\\[-1ex] SLE$_{\kappa'}$, $\kappa'\in (4,8)$} & $\displaystyle{2 - \frac{(12-\kappa')(4+\kappa')}{8\kappa'}}$ & \specialcell{Composition of\\[-1ex] subordinators} & $\displaystyle{\frac{\kappa'}{8} - \frac12}$ \\ \hline  
\specialcell{Double points of\\[-1ex] SLE$_{\kappa'}$, $\kappa'>8$} & $\displaystyle{1 + \frac{2}{\kappa'}}$ &  Running infima of $L$ or $R$ & $\displaystyle{\frac12}$ \\  \hline  
\specialcell{GFF flow line\\[-1ex] intersection\\[-1ex] with angle gap $\theta$\\[1.2ex] } & \specialcell{$\displaystyle{2-\frac{1}{2\kappa}\left(\rho+\frac{\kappa}{2}+2\right)\left(\rho-\frac{\kappa}{2}+6\right)}$,\\[1.2ex] $\displaystyle{\rho  = \frac{\theta}{\pi} \left(2-\frac{\kappa}{2}\right) -2}$} & \specialcell{Composition of\\[-1ex] subordinators} & $\displaystyle{\frac12 - \frac{\rho+2}{\kappa}}$ \\ \hline  
\specialcell{CLE$_{\kappa'}$ gasket,\\[-1ex] $\kappa'\in (4,8)$} & $\displaystyle{2-\frac{(8-\kappa')(3\kappa'-8)}{32\kappa'}}$ &  \specialcell{Times not contained in any\\[-1ex] left $\pi/2$-cone interval} & $\displaystyle{\frac 12+\frac{\kappa'}{16}}$  \\ \hline  
\specialcell{$m$-tuple points of\\[-1ex] space-filling SLE$_{\kappa'}$} & $\displaystyle{\frac{(4m - 4 - \kappa'(m-2))(12 +(\kappa'-4)m)}{8 \kappa'}}$ & $(m-2)$-tuple $\tfrac{\pi}{2}$-cone times &  $d(\kappa',m)$ \\ \hline  
\end{tabular} 
\end{center}
\caption{The sets whose dimension we compute in this paper. Each row shows a set $X$ associated with SLE and a corresponding Brownian motion set $\wh X$ which is contained in $(\eta')^{-1}(X)$. The dimensions of these sets are related as in Theorem~\ref{thm-dim-relation}. See Sections~\ref{sec-dim-calc} and~\ref{sec-multiple-pt} for proofs that the dimensions in the table are as claimed. The quantity in the bottom right cell is $d(\kappa',m) = 1/2-(m-2)(\kappa'/8-1/2)$. }\label{table-dimensions}
\end{table}

The random set $X$ does not have to be measurable with respect to the space-filling $\SLE_{\kappa'}$. However, Theorem~\ref{thm-dim-relation} does not hold without the hypothesis that $X$ is independent from $h$. For example, suppose we take $X$ to be the $\gamma$-thick points of $h$ (see \cite{hmp-thick-pts} as well as Section~\ref{sec-thick-pt-dim}). The $\gamma$-quantum measure $\mu_h$ is supported on $X$~\cite[Proposition~3.4]{shef-kpz}, so since $\eta'$ is parameterized by quantum mass, we have $\dim_{\mcl H} (\eta')^{-1}(X) = 1$. On the other hand, it is shown in \cite{hmp-thick-pts} that $\dim_{\mcl H} X = 2-\gamma^2/2$, so~\eqref{eqn-dim-relation} does not hold for this choice of $X$.

\begin{figure}[ht!]
\begin{center}
\includegraphics[scale=0.92]{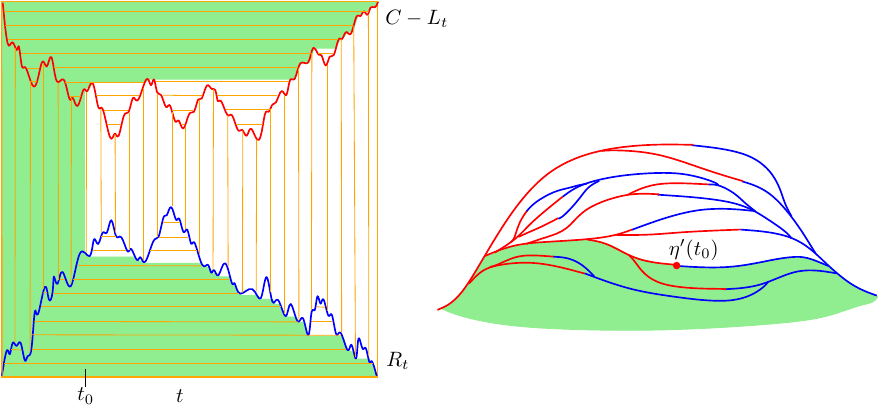}
\end{center}
\caption{\label{fig-intro1} 
The peanosphere construction of \cite{wedges} shows how to obtain a topological sphere by gluing together two correlated Brownian excursions $L,R \colon [0,1] \to [0,\infty)$ (A similar construction works when $L,R$ are two-sided Brownian motions, see \cite[Section~8.2]{wedges}). We choose $C>0$ so large that the graphs of $C-L$ and $R$ do not intersect. We then define an equivalence relation on the square $[0,1]\times [0,C]$ by identifying points which lie on the same horizontal line segment above the graph of $C-L$ or below the graph of $R$; or the same vertical line segment between the two graphs.  As explained in \cite{wedges}, it is possible to check using Moore's theorem \cite{moo25} that the resulting object is a topological sphere decorated with a space-filling path $\eta'$ where $\eta'(t)$ for $t \in [0,1]$ is the equivalence class of $(t , R_t) $.  The pushforward of Lebesgue measure on $[0,1]$ induces a so-called good measure $\mu$ on the sphere (i.e., a non-atomic measure which assigns positive mass to each open set) and $\eta'$ is parameterized according $\mu$-area, i.e., $\mu(\eta'([s,t])) = t-s$ for all $0 \leq s < t \leq 1$.  In \cite{wedges}, the resulting structure is referred to as a \emph{peanosphere} because the space-filling path $\eta'$ is the peano curve between the continuum trees encoded by $L$ and $R$.  It is shown in \cite{wedges} (infinite volume case) and in \cite{sphere-constructions} (finite volume case) that there is a measurable map which associates a pair $(L,R)$ (or equivalently the aforementioned good-measure endowed sphere together with a space-filling path) with an LQG surface decorated with an independent space-filling $\SLE$.  That is, the peanosphere comes equipped with a canonical embedding into the Euclidean sphere where the pushforward of $\mu$ encodes an LQG surface and $\eta'$ is a space-filling $\SLE$.  The embedding of the trees coded by $L$ and $R$ correspond to trees of flow lines with a common angle in an imaginary geometry \cite{ig4}.  The right side of the illustration shows a subset of the SLE-decorated LQG surface, where the green region corresponds to points that are visited by $\eta'$ before some time $t_0$, and where the two trees are embeddings of the trees with contour functions $L$ and $R$. The branches of these continuum random trees correspond to the frontier of the space-filling curve at different times, and the Brownian motions $L$ and $R$ encode the lengths of the left and right, respectively, frontier of the $\SLE$ curve. The right figure illustrates the embedding of the peanosphere for the regime when $\kappa'\in [8,\infty)$; for $\kappa'\in(4,8)$ the green region on the right figure is not homeomorphic to $\BB H$ since for this range of $\kappa'$ values, space-filling $\SLE_{\kappa'}$ is loop-forming.}
\end{figure} 

Theorem~\ref{thm-dim-relation} is in agreement with the KPZ formula~\cite{kpz-scaling} if the ``quantum dimension" of $X$ is defined to be twice the Hausdorff dimension of the Brownian motion set $\wh X$.

\begin{remark} \label{remark-bdy-dim}
Theorem~\ref{thm-dim-relation} only applies to sets associated with SLE and CLE in the interior of a domain. In order to calculate the Hausdorff dimension of sets associated with chordal or radial SLE which intersect the boundary one may apply \cite[Theorem~4.1]{rhodes-vargas-log-kpz}, which implies the following one-dimensional KPZ formula. Let $h$ be a free boundary GFF in the upper half-plane and let $\nu_h$ be the associated boundary measure. Define the quantum dimension of a set $X\subset [0,\infty)$ to be 
\eqbn
\dim_{\mcl H}(\wh X),\qquad \wh X:= \{\nu_h([0,x])\,:\, x\in X \}.
\eqen
Then it holds a.s.\ that
\eqb
\dim_{\mcl H}(X) = \left(1+\frac{\gamma^2}{4} \right) \dim_{\mcl H}(\wh X)   - \frac{\gamma^2}{4}  \dim_{\mcl H}(\wh X)^2.
\label{eqn-dimcalc-1}
\eqe
(The function $\zeta$ in the statement of \cite[Theorem~4.1]{rhodes-vargas-log-kpz} can be obtained using the scaling properties of the free boundary GFF and \cite[Proposition~2.5]{rhodes-vargas-log-kpz}.)  At the end of Section~\ref{sec-dim-calc} we will include a short example showing how one may use this theorem, combined with results of \cite{wedges}, to obtain the Hausdorff dimension of the points where a chordal SLE$_{\kappa}(\rho)$, $\kappa\in(0,8)$, in $\BB H$ intersects the real line. 
\end{remark}

We will now discuss how our version of the KPZ formula relates to other KPZ-type formulas in the literature. The results of \cite{shef-kpz} relate the expected Euclidean mass of the Euclidean $\delta$-neighborhood of a set $X$ to the expected quantum mass of the so-called quantum $\delta$-neighborhood of $X$, which is defined in terms of Euclidean balls of quantum mass $\delta>0$. The scaling exponents for these dimensions are proven to satisfy the KPZ equation. In \cite{rhodes-vargas-log-kpz, bjrv-gmt-duality, shef-renormalization}, the authors consider a notion of quantum Hausdorff dimension in terms of the quantum mass of Euclidean balls covering a set $X$ and obtain KPZ formulas relating the a.s.\ quantum Hausdorff dimension to the a.s.\ Euclidean Hausdorff dimension. The KPZ relations in the works \cite{shef-kpz,rhodes-vargas-log-kpz, bjrv-gmt-duality, shef-renormalization} all rely on the Euclidean geometry, since Euclidean balls or squares are used to cover $X$. In our formulation, by contrast, we obtain a cover of $X$ by pushing forward a cover of the time set $\wh X$ via the curve $\eta'$, hence the quantum dimension does not rely on the Euclidean geometry.  
The versions given in \cite{grv-kpz,gp-kpz} also do not rely on Euclidean geometry because the notions of quantum dimension in these papers are defined in terms of the heat kernel for the Liouville Brownian motion (see \cite{grv-lbm,berestycki-lbm,grv-heat-kernel,rhodes-vargas-criticality, jackson-lbm}) and the LQG metric, respectively, which are intrinsic to the LQG surface.  As we shall see in Section~\ref{sec-dim-calc}, however, the formulation considered here appears to be more amenable to explicit calculation.

\subsection{Background and notation}
\label{sec-intro-background}

We will now describe the basic notation and objects that we will use throughout the paper, including quantum cones, quantum wedges and space-filling $\SLE_{\kappa'}$. We refer the reader to \cite{wedges,ig4,shef-zipper} for more details. 

\subsubsection{Notation}
\label{sec-notation}

We adopt the convention of~\cite{ig1,ig2,ig3,ig4} that $\kappa$ denotes a parameter in $(0,4]$ and $\kappa' := 16/\kappa \in (4,\infty)$. If the $\kappa$-value is not restricted to either of the two intervals $(0,4]$ or $(4,\infty)$, we will simply write $\kappa$. We also use the following notation.

\begin{notation}\label{def-asymp}
If $a$ and $b$ are two quantities, we write $a\preceq b$ (resp.\ $a \succeq b$) if there is a constant $C$ (independent of the parameters of interest) such that $a \leq C b$ (resp.\ $a \geq C b$). We write $a \asymp b$ if $a\preceq b$ and $a \succeq b$. 
\end{notation}

\begin{notation} \label{def-o-notation}
If $a$ and $b$ are two quantities which depend on a parameter $x$, we write $a = o_x(b)$ (resp.\ $a = O_x(b)$) if $a/b \rta 0$ (resp.\ $a/b$ remains bounded) as $x \rta 0$ (or as $x\rta\infty$, depending on context). We write $a = o_x^\infty(b)$ if $a = o_x(b^s)$ for each $s \in\BB R$. 
\end{notation}

Unless otherwise stated, all implicit constants in $\asymp, \preceq$, and $\succeq$ and $O_x(\cdot)$ and $o_x(\cdot)$ errors involved in the proof of a result are required to satisfy the same dependencies as described in the statement of the result.

\subsubsection{Quantum cones and quantum wedges}
\label{sec-cone-prelim}

Fix $\gamma \in (0,2)$.  As explained above, a $\gamma$-LQG surface is an equivalence class of pairs $(D,h)$, where $D$ is a planar domain and $h$ is an instance of some form of the GFF on $D$, see e.g \cite{shef-kpz,shef-zipper,wedges}, where two such pairs are said to be equivalent if they are related as in~\eqref{eqn::change_of_coordinates}.  The surface $(D,h)$ is equipped with a quantum area measure that can be formally represented as $\mu_h =e^{\gamma h(z)}\,dz $(where $dz$ is Lebesgue measure) as well as a quantum length measure $\nu_h = e^{\gamma h(z)/2} dz$ (where $dz$ is Lebesgue length measure in the case when the boundary is a straight line),  which is defined on $\bdy D$ as well as on certain curves in the interior of $D$, including $\SLE_\kappa$-type curves for $\kappa  = \gamma^2$~\cite{shef-zipper}.  We refer to a particular choice of equivalence class representative $(D ,h)$ as an \emph{embedding} of the quantum surface. If $(D,h)$ and $(\wt D,\wt h)$ are equivalent and $\varphi:D\to \wt D$ is the conformal map of~\eqref{eqn::change_of_coordinates}, $\mu_{\wt h}$ is almost surely the push-forward of $\mu_{h}$ under $\varphi$, i.e., $\mu_{h}(A)=\mu_{\wt h}(\varphi(A))$ for any $A\subseteq D$ \cite[Proposition~2.1]{shef-kpz}.  A similar statement is true for~$\nu_h$.

Several types of quantum surfaces of the form $(D,h,x_1,...,x_k)$, where $x_1,...,x_k\in\overline{D}$ are additional marked points, are introduced in \cite{shef-zipper,wedges}.  Two such surfaces $(D,h,x_1,\dots ,x_k)$ and $(\wt D,\wt h, \wt x_1,\dots , \wt x_k)$ are defined to be equivalent if there exists a conformal map $\varphi \colon D \to \wt{D}$ where $\wt{h},h$ are related as in~\eqref{eqn::change_of_coordinates} and $\varphi(x_j) = \wt{x}_j$ for each $1 \leq j \leq k$.  In this paper we will mainly consider $\alpha$-quantum cones, $\alpha<Q$, which are a one-parameter family of doubly marked quantum surfaces homeomorphic to $\BB C$. In Section~\ref{sec-dim-calc} we will also need some theory of quantum wedges.  

Let $\mcl H(\BB C)$ be the Hilbert space closure modulo a global additive constant of the subspace of functions $f\in C^\infty(\BB C)$ satisfying $\|f\|^2_{\nabla}:=(f,f)_{\nabla}<\infty$, where $(f,g)_\nabla:= (2\pi)^{-1} \int_{\BB C}\nabla f\cdot \nabla g\,dz$ for $g\in C^\infty(\BB C)$ for which the integral is well-defined and finite.  Let $\mcl H_1(\BB C)\subset\mcl H(\BB C)$ be the subspace of functions that are radially symmetric about the origin, and let $\mcl H_2(\BB C)\subset\mcl H(\BB C)$ be the subspace of functions (modulo of global additive constant) which have mean zero about all circles centered at the origin. By \cite[Lemma~4.9]{wedges} we have $\mcl H(\BB C)=\mcl H_1(\BB C)\oplus \mcl H_2(\BB C)$.  Recall that a whole-plane GFF $h$ is a modulo additive constant distribution on the complex plane (i.e., a continuous linear functional defined on the subspace of functions $f\in C_0^\infty(\C)$ with $\int_{\C} f(x)dx=0$) which can be represented as $h=\sum_{n\in\N} \alpha_n f_n$, where $(\alpha_n)_{n\in\N}$ is a series of i.i.d.\ standard normal random variables, and $(f_n)_{n\in\N}$ is an orthonormal basis for $\mcl H(\C)$.

\begin{defn} \label{def-quantum-cone}
Let $\alpha\in (0,Q]$. \emph{An $\alpha$-quantum cone} is the doubly marked quantum surface $(\BB C , h , 0 , \infty)$, where $h=h^\dagger+h^0$ is a random distribution sampled as follows. The radially symmetric function $h^\dagger$ takes the value $A_s$ on $\partial B_{e^{-s}}(0)$, where $A_s=\wt B_s+\alpha s$ for $\wt B$ a standard two-sided Brownian motion, conditioned such that $A_s\geq Qs$ for $s<0$. The distribution $h^0$ is independent of $h^\dagger$, and is given by the projection of a whole-plane GFF onto $\mcl H_2(\BB C)$. 
\end{defn}

There is a two-parameter family of embeddings of a quantum cone into $(\BB C , 0 ,\infty)$ (i.e., choices of the distribution~$h$), corresponding to multiplication of $\BB C$ by a complex number. The distribution $h$ in Definition~\ref{def-quantum-cone} is one such choice.

\begin{defn}\label{def-circle-embedding}
For $\alpha \in (0,Q]$, the \emph{circle average embedding} of an $\alpha$-quantum cone is the distribution $h$ of Definition~\ref{def-quantum-cone}. 
\end{defn}
  
\begin{remark} \label{remark-cone-gff}
One of the main reasons why we are interested in the embedding of Definition~\ref{def-circle-embedding} is that under this embedding, $h |_{\BB D}$ agrees in law with the restriction to $\BB D$ of a whole-plane GFF plus $-\alpha \log |\cdot|$, with additive constant chosen so that its circle average over $\bdy \BB D$ vanishes. Indeed, this is a straightforward consequence of Definition~\ref{def-quantum-cone}.
\end{remark}

We refer to~\cite[Section~4.3]{wedges} for further details regarding quantum cones.

For $\alpha \leq Q$ (with $Q$ as in~\eqref{eqn::change_of_coordinates}), an $\alpha$-quantum wedge is a doubly marked quantum surface, which is homeomorphic $\BB H$. For $\alpha \in (Q,Q+\frac{\gamma}{2})$ an $\alpha$-quantum wedge is a Poissonian chain of finite volume doubly-marked quantum surfaces, each of which is homeomorphic to $\BB D$ and has two marked points. In the first case we say that the quantum wedge is thick, and in the second case it is thin. A thick wedge can be represented as a quantum surface $(\BB H,h,0,\infty)$, where $h=h^0+h^\dagger$ is a decomposition of $h$ into a distribution $h^0$ of mean zero on all half-circles around the origin, and a radially symmetric function $h^\dagger$.  For an $\alpha$-quantum wedge with $\alpha\leq Q$ which is parameterized by $\BB H$ one possible embedding is such that the law of $h^\dagger$ on $\BB H\cap \partial B(0,e^{-s})$ is identical to the function $A$ described above in Definition~\ref{def-circle-embedding}, except that $\wt B_s$ is replaced by $\wt B_{2s}$.

If we conformally weld the two boundaries of a quantum wedges according to quantum boundary length we obtain a quantum cone.  Conversely, we obtain a quantum wedge if we cut out a surface by considering an independent whole-plane $\SLE_{\kappa}(\rho)$ curve, $\kappa=\gamma^2$, on top of a quantum cone \cite[Theorem~1.5]{wedges}, with $\rho$ depending on the parameter $\alpha$ of the cone. We also obtain a wedge by conformally welding together multiple wedges according to quantum boundary length, and we obtain two independent wedges if we cut a wedge into two or several components by a collection of $\SLE_\kappa(\ul\rho)$ curves for certain values of $\kappa$ and $\ul\rho$. Quantum wedges and cones can be described by their weight $W=W(\alpha) $ (which is defined to be $W = \gamma (\gamma/2 +Q-\alpha)$ for a wedge and $W = 2\gamma(Q-\alpha)$ for a cone) rather than $\alpha$. The weight of the surfaces is additive under the operations of gluing/welding and cutting as described above.

\subsubsection{Space-filling SLE}
\label{sec-sle-prelim}

Here we give a moderately detailed overview of the construction and basic properties of whole-plane space-filling SLE$_{\kappa'}$ from $\infty$ to $\infty$ for $\kappa' > 4$, which was originally defined in~\cite[Section~1.4.1]{wedges}, building on the chordal definition in~\cite[Sections 1.2.3 and 4.3]{ig4}. For $\kappa' \geq 8$, whole-plane space-filling SLE$_{\kappa'}$ from $\infty$ to $\infty$ is just a certain curve from $\infty$ to $\infty$ which locally looks like an SLE$_{\kappa'}$. For $\kappa' \in (4,8)$, space-filling SLE$_{\kappa'}$ from $\infty$ to $\infty$ traces points in the same order as a curve which locally looks like SLE$_{\kappa'}$, but fills in the ``bubbles" which it disconnects from its target point with a continuous space-filling loop. 

To construct whole-plane space-filling SLE$_{\kappa'}$ from $\infty$ to $\infty$, fix a deterministic countable dense subset $\mcl C\subset \BB C$ and let $\wh h$ be a whole-plane GFF, viewed modulo a global additive multiple of $2\pi \chi$ where $\chi =  \sqrt{\kappa'}/2 - 2/\sqrt{\kappa'}$. It is shown in~\cite{ig4} that for each $z\in\mcl C$, one can make sense of the flow lines $\eta_z^L$ and $\eta_z^R$ of angles $\pi/2$ and $-\pi/2$, respectively, started from $z$. These flow lines are SLE$_\kappa(2-\kappa)$ curves for $\kappa = 16/\kappa'$~\cite[Theorem~1.1]{ig4}. The flow lines $\eta_z^L$ and $\eta_w^L$ (resp.\ $\eta_z^R$ and $\eta_w^R$) started at distinct points in $\mcl C$ eventually merge together, such that the collection of flow lines $\eta_z^L$ (resp.\ $\eta_z^R$) for $z\in\mcl C$ form the branches of a tree rooted at $\infty$. 

We define a total order on $\mcl C$ by declaring that $w$ comes after $z$ if and only if $w$ lies in a connected component of $\BB C\setminus (\eta_z^L \cup \eta_z^R)$ whose boundary is traced by the right side of $\eta_z^L$ and the left side of $\eta_z^R$. It can be shown using the same argument as in the chordal case~\cite[Section 4.3]{ig4} (or alternatively deduced from the chordal case; see~\cite[Footnote 4]{wedges}) that there is a unique space-filling curve $\eta' : \BB R\rta\BB C$ which traces the points in $\mcl C$ in order, is continuous when parameterized by Lebesgue measure, and is such that $(\eta')^{-1}(\mcl C)$ is a dense set of times. The curve $\eta'$ does not depend on the choice of $\mcl C$ and is defined to be whole-plane space-filling SLE$_{\kappa'}$ from $\infty$ to $\infty$. 

For each fixed $z\in\BB C$, it is a.s.\ the case that the left and right outer boundaries of $\eta'$ stopped at the first (and a.s.\ only) time $\tau_z$ that it hits $z$ are given by the flow lines $\eta_z^L$ and $\eta_z^R$. For $\kappa' \geq 8$, these two flow lines do not intersect so $\BB C\setminus \eta'((-\infty,t])$ for each time $t$ has the topology of the half-plane. For $\kappa' \in (4,8)$, the curves $\eta_z^L$ and $\eta_z^R$ intersect each other so $\BB C\setminus \eta'((-\infty,t])$ instead consists of a string of domains with the topology of the disk, separated by the intersection points. By~\cite[Footnote 4]{wedges}, if we condition on $\eta_z^L $ and $\eta_z^R$ (equivalently, on $\eta'((-\infty,\tau_z])$ or $\eta'([\tau_z,\infty))$, then the conditional law of $\eta'|_{[\tau_z,\infty)}$ is that of a chordal SLE$_{\kappa'}$ from 0 to $\infty$ in $\eta'([\tau_z,\infty))$ if $\kappa' \geq 8$; or a concatenation of independent chordal space-filling SLE$_{\kappa'}$ curves in the connected components of the interior of $\eta'([\tau_z,\infty))$ if $\kappa' \in (4,8)$. The conditional law of the time reversal of $\eta'|_{(-\infty,\tau_z]}$ admits a similar description. 

The curve $\eta'$ is also closely related to the SLE$_{\kappa'}(\kappa'-6)$ \emph{counterflow lines} of $\wh h$ from $\infty$ to $z$ for any given $z\in\BB C$. In particular, if we parameterize $\eta'$ by capacity as seen from $z$, so we skip all of the bubbles filled in by $\eta'$ before it hits $z$, then we a.s.\ recover the counterflow line from $\infty$ targeted at $z$. The collection of all of these counterflow lines, targeted at a countable dense set of points, forms a whole-plane branching SLE$_{\kappa'}(\kappa'-6)$ process, which can be used to construct a whole-plane CLE$_{\kappa'}$ via a whole-plane analog of the construction in~\cite{shef-cle}. Hence a whole-plane space-filling SLE$_{\kappa'}$ from $\infty$ to $\infty$ encodes a whole-plane CLE$_{\kappa'}$.

\subsection{Outline}
\label{sec-intro-outline}
Section~\ref{sec-dim-calc} gives a number of examples of $\SLE$ sets of known Hausdorff dimension, for which Theorem~\ref{thm-dim-relation} provides an alternative derivation.

Section~\ref{sec-gff-esimates} contains various SLE and GFF estimates which we will need for the proof of Theorem~\ref{thm-dim-relation}, both for the upper bound and for the lower bound. In Section~\ref{sec-sle-big} we will prove that, with high probability, any segment of a space-filling $\SLE_{\kappa'}$ curve of diameter $\ep \in (0,1)$, contains a Euclidean ball of radius $\epsilon^{1+o_\ep(1)}$. An interesting corollary of this result is that space-filling $\SLE_{\kappa'}$ is a.s.\ locally $\alpha$-H\"older continuous for any $\alpha<1/2$ when parameterized by Lebesgue measure. In Sections~\ref{sec-mu_h-big} and~\ref{sec-gff-cont} we will prove some estimates which give that the quantum mass of a small Euclidean ball is unlikely to be much smaller than what is predicted from the GFF circle average process at a nearby point. 

In Section~\ref{sec-thick-pt-dim} we prove that the dimension of the intersection of a general Borel set $A$ with the $\alpha$-thick points of a GFF $h$ is a.s.\ equal to $\dim_{\mcl H} A - \alpha^2/2$. The proof is a generalization of the argument in \cite{hmp-thick-pts}, and is used in the proof of the upper bound of Theorem~\ref{thm-dim-relation}.

Section~\ref{sec-dimH} contains the proof of Theorem~\ref{thm-dim-relation}. Unlike for most Hausdorff dimension calculations, the upper bound for $\dim_{\mcl H}(X)$ is more challenging to prove than the lower bound. Using the results of Section~\ref{sec-gff-esimates}, we obtain an estimate for the diameter of $\eta'(I)$ for $I\subset\BB R$, in terms of $\op{diam}(I)$ and the thickness of the field at a point near $\eta'(I)$. This leads to an upper bound for the dimension of the intersection of $X$ with the $\alpha$-thick points of $h$ in terms of $\dim_{\mcl H} (\wh X)$. Using the result of Section~\ref{sec-thick-pt-dim} and optimizing over $\alpha$, we obtain the upper bound in~\eqref{eqn-dim-relation}. The lower bound in~\eqref{eqn-dim-relation} is proven via a more direct argument based on moment estimates for the quantum measure along with the estimates of Sections~\ref{sec-sle-big} and~\ref{sec-mu_h-big}.  

In Section~\ref{sec-multiple-pt} we use Theorem~\ref{thm-dim-relation} to give a proof for the Hausdorff dimension of the set of $m$-tuple points of space-filling $\SLE_{\kappa'}$, which in the setting of Theorem~\ref{thm-dim-relation} correspond to the image under $\eta'$ of the so-called $(m-2)$-tuple cone times of the two-dimensional correlated Brownian motion. An $m$-tuple cone time can be described by a ``cone vector" in $\BB R^{m+1}$ consisting of $m$ cone times~\cite{shimura-cone,evans-cone} for $Z$ and the time-reversal of $Z$, where the end of one cone excursion marks the beginning of another cone excursion in the opposite direction. We obtain the Hausdorff dimension of the set of cone vectors by standard techniques, including a two-point estimate for correlations; then obtain the dimension of the set of $(m-2)$-tuple cone times (which is a subset of $\BB R$) by projection. Our result generalizes the result in \cite{evans-cone}, where the dimension of the set of single cone times is calculated. 

Finally, in Section~\ref{sec-problems}, we list some open problems related to the results of this paper.

\subsubsection*{Acknowledgements}
E.G.\ was supported by the U.S. Department of Defense via an NDSEG fellowship. N.H.\ was supported by a fellowship from the Norwegian Research Council. J.M.\ was partially supported by DMS-1204894. Part of this work was carried out during the Random Geometry semester at the Isaac Newton Institute, Cambridge University, and the authors would like to thank the institute and the organizers of the program for their hospitality. The authors would also like to thank Peter M\"orters, Scott Sheffield, and Xin Sun for helpful discussions and an anonymous referee for helpful comments on an earlier version of this paper.

\section{Examples}
\label{sec-dim-calc}

In this section we will use Theorem~\ref{thm-dim-relation} to give alternative proofs of several Hausdorff dimensions already known in the literature, in addition to a calculation of a new Brownian motion dimension for which the corresponding SLE dimension is already known. Throughout, we let $\eta'$ be a space-filling $\SLE_{\kappa'}$ on top of an independent $\gamma$-quantum cone $(\BB C,h,0,\infty)$, $\kappa'>4$. We parameterize $\eta'$ by quantum area, and let $L$ and $R$ denote the left and right boundary length process, respectively, relative to time $0$. Define $Z=(L,R)$, and recall that $Z$ has the law of a two-dimensional Brownian motion with covariances as in~\eqref{eq-multiplept-10}. Before presenting the examples we recall the definitions of cone times, cone intervals and cone excursions.
\begin{defn}
Let $\alpha\in(0,2\pi)$, $t\in\BB R$, and let $\BB v(\alpha')$ denote the unit vector in direction $\alpha'$ for any $\alpha'\in(0,2\pi)$. Then $t$ is an $\alpha$-cone time of $Z$ if there is a time $s>t$ such that, for each $s'\in[t,s]$, there exists $r_{s'} \in [0,\infty)$ and an angle $\alpha_{s'}\in[0,\alpha]$, satisfying $Z_{s'}=Z_t+r_{s'}\BB v(\alpha_{s'})$. If we do not specify an angle $\alpha$, we will assume $\alpha=\frac{\pi}{2}$, i.e., a cone time is a $\frac{\pi}{2}$-cone time. 
\label{def-conetime}
\end{defn}
For a cone time $t$ of $Z$ we define the function $v$ by
\eqb \label{eqn-cone-exit-time}
v(t) = \inf\{s>t\,:\,R_s<R_t\text{ or }L_s<L_t\}.
\eqe
The interval $(t,v(t))$ is called a \emph{cone interval} and $Z|_{(t,v(t))}$ is called a \emph{cone excursion}. We say that $t$ is a right cone time for $Z$ if $R_{v(t)}> R_t$ (equivalently, $L_{v(t)}=L_t$), and we say that $t$ is a left cone time for $Z$ if $L_{v(t)}>L_t$ (equivalently, $R_{v(t)}=R_t$). 

The Hausdorff dimension of $\SLE_\kappa$ was first calculated in \cite{schramm-sle} and \cite{beffara-dim}. For $\kappa\in(0,4)$ we obtain an alternative proof by using that the boundary of $\eta'([0,\infty))$ has the law of an $\SLE_\kappa$ curve.

\begin{example} 
The Hausdorff dimension of an $\SLE_{\kappa}$ curve for $\kappa\in(0,4)$ is a.s.\ equal to $1+\frac{\kappa}{8}$.
\label{prop-dimcalc-1}
\end{example}
\begin{proof}
Let $\kappa' := 16/\kappa$. If we stop the space-filling $\SLE_{\kappa'}$ process $\eta'$ upon reaching $0$, the boundary of the already traced region is given by two flow lines of a whole-plane GFF with angle gap $\pi$, see \cite[Footnote~4]{wedges}. The marginal law of each of these flow lines is that of a whole-plane $\SLE_\kappa(2-\kappa)$, see \cite[Theorem~1.1]{ig4}.  If $\eta'$ is parameterized by quantum mass, the times at which $\eta'$ traces the left and right boundaries of $\eta'([0,\infty))$ correspond exactly to the running infima of the left and right boundary length processes $L$ and $R$, respectively, relative to time $0$. This time set has Hausdorff dimension $1/2$ \cite[Theorem~4.24]{peres-bm}.  Combining with Theorem~\ref{thm-dim-relation} we see that a whole-plane $\SLE_\kappa(2-\kappa)$ curve a.s.\ has Hausdorff dimension $1+\frac{\kappa}{8}$. In fact, the same argument shows that a.s.\ every non-trivial segment of a whole-plane $\SLE_\kappa(2-\kappa)$ curve has Hausdorff dimension $1+\frac{\kappa}{8}$. By local absolute continuity of $\SLE_\kappa(2-\kappa)$ and $\SLE_\kappa$ away from the self-intersection times of the former~\cite[equation (9)]{sw-coord}, an ordinary radial or whole-plane $\SLE_\kappa$ also has a.s.\ Hausdorff dimension $1+\frac{\kappa}{8}$, and using local absolute continuity again~\cite[Theorems 3 and 6]{sw-coord} we deduce the same result for ordinary chordal $\SLE_\kappa$.
\end{proof}

The Hausdorff dimension of $\SLE_{\kappa'}$, $\kappa'\in(4,8)$, is obtained by using that $\SLE_{\kappa'}$ corresponds to the so-called \emph{ancestor free times} $(t(s))_{s\geq 0}$ of $Z$.  A time $s\geq 0$ is ancestor free if it is not contained in any $\pi/2$-cone interval for $Z$ which is contained in $[0,\infty)$. In other words, $s$ is ancestor free if there is no $t\in[0,s)$ such that $L_u \geq L_t$ and $R_u \geq R_t$ for all $u\in (t,s]$.

\begin{example}
The Hausdorff dimension of an $\SLE_{\kappa'}$ curve for $\kappa'\in(4,8)$ in $\BB C$ or in a domain whose boundary has dimension at most $1 + \frac{\kappa'}{8}$ is a.s.\ equal to $1+\frac{\kappa'}{8}$.  
\label{prop-dimcalc-2}
\end{example}

\begin{proof}
Consider the space-filling $\SLE_{\kappa'}$ $\eta'$ described above. Let $(t(s))_{s\geq 0}$ denote the inverse of the local time at the ancestor free times of $(L,R)$ relative to $t=0$, as described in \cite[Section~1.4.2]{wedges}. Conditioned on $\eta'([0,\infty))$ the law of $(\eta'(t(s)))_{s\geq 0}$ is that of a concatenation of independent $\SLE_{\kappa'}(\kappa'/2-4 ;\kappa'/2-4)$ processes in each of the bubbles (connected components of the interior) of $\eta'([0,\infty))$, see \cite[Lemma~10.4]{wedges}. For any such bubble $D$ let $D_n\subset D$ consist of the points in $D$ at distance at least $1/n$ from $\partial D$. By local absolute continuity~\cite{sw-coord}, a.s.\ the intersection of the $\SLE_{\kappa'}(\kappa'/2-4 ; \kappa'/2-4)$ with $D_n$ for sufficiently large $n$ has the same Hausdorff dimension as the intersection of an ordinary $\SLE_{\kappa'}$ curve with a sub-domain at positive distance from the boundary of its domain. Since $\partial D$ has Hausdorff dimension $1+\frac{2}{\kappa'}<1+\frac{\kappa'}{8}$ by Example \ref{prop-dimcalc-1}, it follows that it is sufficient to prove that the image of $(\eta'(t(s)))_{s\geq 0}$ has dimension $1+\frac{\kappa'}{8}$. By \cite[Proposition~1.13]{wedges}, the processes $(L_{t(s)})_{s\geq 0}$ and $(R_{t(s)})_{s\geq 0}$ are independent $\kappa'/4$-stable processes. Hence the range of $(L_{t(s)},R_{t(s)})_{s\geq 0}$ has dimension $\kappa'/4$; see the discussion right after \cite[Theorem~1.2]{pt69}. Kaufman's theorem \cite[Theorem~9.28]{peres-bm} implies that the corresponding time set for $(L,R)$ has dimension $\kappa'/8$.  An application of Theorem~\ref{thm-dim-relation} completes the proof. 
\end{proof}

The dimension of the cut points and double points of $\SLE_{\kappa'}$ for $\kappa' \in (4,8)$ were first calculated in \cite{miller-wu-dim}. Recall that the set of cut points of a curve $\eta$ is the set $\{\eta(t)\,:t\in(0,\infty),\,\eta((0,t))\cap\eta((t,\infty))=\emptyset\}$.  The set of local cut points of a curve $\eta$ parameterized by $\R_+$ is the set $\{\eta(t) \,: t > 0, \exists s > 0,\, \eta((t-s,t)) \cap \eta((t,t+s)) = \emptyset\}$.  The set of cut points of $\eta'|_{[t,\infty)}$ for $t \in\BB R$ is contained in the set of points where the left and the right boundaries of $\eta'([t,\infty))$ meet in the manner described in Figure \ref{fig-cutdouble}, which in turn corresponds to the set of times when both $L$ and $R$ are at a simultaneous running infimum relative to time $t$.
 
\begin{example}
The Hausdorff dimension of the sets of cut points and local cut points of a chordal, radial, or whole-plane $\SLE_{\kappa'}$ for $\kappa'\in(4,8)$ are each a.s.\ equal to $3-\frac{3}{8}\kappa'$.
\label{prop-dimcalc-3}
\end{example}
\begin{proof}
Let $(t(s))_{s\geq 0}$ be defined as in the proof of Example~\ref{prop-dimcalc-2}. Let $\wh\eta'$ be the curve from $\infty$ to $0$ which is equal to the time-reversal of $(\eta'(t(s)))_{s\geq 0}$, which has the law of a whole-plane SLE$_{\kappa'}(\kappa'-6)$; see Section~\ref{sec-sle-prelim} and~\cite[Theorem 1.20]{ig4}.

We will first prove that the set of local cut points of $\wh\eta'$  a.s.\ has dimension $3-3\kappa'/8$. The set of local cut points of $\wh\eta'$ is contained in the union over all $t \in \BB Q \cap [0,\infty)$ of the set of points where the left boundary of $\eta'([t,\infty))$ hits the right boundary of $\eta'([t,\infty))$ on the left-hand side (except for the point $\eta'(t)$ itself) and contains this set for $t=0$. By countable stability of Hausdorff dimension and translation invariance, it suffices to show that the dimension of the set of points where the left and right boundaries of $\eta'([0,\infty))$ intersect in this manner a.s.\ has dimension $3-3\kappa'/8$. 

The pre-image of this set under $\eta'$ (parameterized by quantum area) is equal to the set of times when the correlated Brownian motions~$L$ and~$R$ attain a simultaneous running infimum relative to time~$0$. A simultaneous running infimum of $L$ and $R$ is the same as a $\pi/2$-cone time $t$ for the time-reversal of $(L,R)$ with the property that~$0$ is contained in the corresponding cone interval. By~\cite[Theorem~1]{evans-cone} (c.f.\ the proof of~\cite[Lemmas~8.5]{wedges}), it follows that the Hausdorff dimension of this set of times is a.s.\ $1-\kappa'/8$; in fact, the same is a.s.\ true of its intersection with $[0,s]$ for any $s>0$. By applying Theorem~\ref{thm-dim-relation}, we obtain that the dimension of this set of intersection points of the boundaries of $\eta'|_{[0,\infty)}$, and hence also the set of local cut points of $\wh\eta'$, is given by $3-3\kappa'/8$; in fact the same is a.s.\ true for the set of local cut points of every non-trivial segment of $\wh\eta'$.

We will now argue by local absolute continuity that the a.s.\ dimension of the set of local cut points of whole-plane, chordal, or radial SLE$_{\kappa'}$ is the same as the a.s.\ dimension of the set of local cut points of $\wh\eta'$. We have local absolute continuity of the curves when we do not consider points at which the curves hit their domain boundary or their past~\cite{sw-coord}, and to conclude it is sufficient to argue that the cut point dimension of the curves does not decrease if we remove points of this kind. 
Choral SLE$_{\kappa'}$ a.s.\ does not have global cut points which intersect the domain boundary, since the left boundary of the SLE$_{\kappa'}$ has the law of an SLE$_\kappa(\kappa-4;\kappa/2-2)$ \cite[Theorem~1.4]{ig4}, which implies by \cite[Lemma~15]{dubedat-duality} that the left boundary of the SLE$_{\kappa'}$ a.s.\ does not hit the right domain boundary. It follows from reversibility~\cite{ig2,ig3} and the domain Markov property that chordal SLE$_{\kappa'}$ cannot have any local cut points which are also multiple-points or which intersect the domain boundary. By local absolute continuity, the same result follows for the radial and whole-plane cases, i.e., a radial or whole-plane SLE does not have cut points which are also multiple-points or which intersect the domain boundary. The curve $\wh\eta'$ can have a local cut point which is also a multiple point, but by local absolute continuity with respect to ordinary SLE$_{\kappa'}$ away from the times it interacts with its force point, it has to wrap around the origin between the first time it hits the point and the time when it has a local cut point there,  
so since any non-trivial segment of $\wh\eta'$ has the same local cut point dimension of $3-3\kappa'/8$, the set of local cut points of this kind does not have a larger dimension than the set of local cut points which are not multiple points.

We now argue that the dimension of the set of global cut points is also a.s.\ equal to $3-3\kappa'/8$. 
By the conformal Markov property and transience of SLE$_{\kappa'}$, for any $t_0\in\R$ and $\ep > 0$, if we condition on the initial segment $\wt\eta'|_{t\leq t_0}$ of a chordal, radial or whole-plane SLE$_{\kappa'}$ curve $\wt\eta'$, the global cut points for $\wt\eta'|_{t\leq t_0}$ which lie at distance at least $\ep$ from $\wt\eta'(t_0)$ will be global cut points for $\wt\eta'$ with positive probability. This implies that for any $\zeta>0$ the global cut points of $\wt\eta'$ have dimension at least $3-\frac{3}{8}\kappa'-\zeta$ with positive probability. Here we subtract a small parameter $\zeta$ since the dimension of the cut points of $\wt\eta'([0,t_0])$ might be slightly larger than the dimension of the cut points of $\wt\eta'([0,t_0])\setminus B_\ep(\wt\eta'(t_0))$. By scale invariance, for (say) a chordal SLE$_{\kappa'}$ in $\BB H$ from 0 to $\infty$ and for any $r>0$, the probability that the intersection of the set of global cut points of the curve with $B_r(0)$ has dimension $\geq 3-\frac{3}{8}\kappa'-\zeta$, is independent of $r$. If $\wh h$ is a GFF whose imaginary geometry counterflow line~\cite{ig1} is our given SLE$_{\kappa'}$ curve, then the sigma algebra $\cap_{r>0} \sigma(\wh h |_{B_r(0)})$ is trivial (see, e.g.,~\cite[Proposition~3.2]{ig1}).

Since the chordal SLE$_{\kappa'}$ is locally determined by $\wh h$ this implies that the global cut point dimension is $\geq 3-\frac{3}{8}\kappa' -\zeta$ almost surely. Sending $\zeta\rta 0$ we get that the global cut point dimension of a chordal SLE$_{\kappa'}$ is $3-\frac{3}{8}\kappa'$. A similar argument using GFF tail triviality works in the case of whole-plane SLE, and the radial case follows from the whole-plane case.  
\end{proof}

\begin{figure}[ht!]
\begin{center}
\includegraphics[scale=0.92]{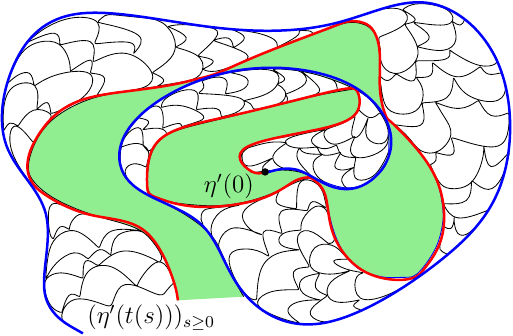}
\end{center}
\caption{\label{fig-cutdouble} Illustration of Examples~\ref{prop-dimcalc-2} and~\ref{prop-dimcalc-3}. The figure shows the west-going (resp.\ east-going) flow line from the origin in red (resp.\ blue), and the region $\eta'((-\infty,0])$ is shown in green. The counterflow line $(\eta'(t(s)))_{s\geq 0}$, whose marginal law is a whole-plane $\SLE_{\kappa'}(\kappa'-6)$, is shown in black. The cut points of the counterflow line are contained in the set of points where the west-going flow line from the origin hits the east-going flow line from the origin on the left-hand side. These points of intersection correspond to simultaneous running infima of $L$ and $R$ relative to time $0$.}
\end{figure}

\begin{example}
\label{prop-dimcalc-4}
The Hausdorff dimension of the double points of $\SLE_{\kappa'}$ is a.s.\ equal to
\eqb \label{eqn-dimcalc-4}
2-\frac{(12-\kappa')(4+\kappa')}{8\kappa'} \quad\text{for}\quad \kappa'\in(4,8) \quad\text{and}\quad
1+\frac{2}{\kappa'} \quad\text{for}\quad \kappa'\geq 8.
\eqe
\end{example}
\begin{proof}
For $\kappa'\geq 8$ the set of double points of SLE$_{\kappa'}$ has the same dimension as the boundary of $\eta'([0,\infty))$, since, conditioned on $\eta'((-\infty,0])$, $(\eta'(t))_{t\geq 0}$ has the law of a chordal SLE$_{\kappa'}$. It follows that the set of double points has dimension $1+2/\kappa'$, since the left and right boundaries of $\eta'([0,\infty))$ marginally each have the law of a whole-plane $\op{SLE}_{\kappa}(2-\kappa)$ from $0$ to $\infty$, $\kappa=16/\kappa'$.

Now assume $\kappa' \in (4,8)$. Let $U$ be a connected component of $\C\setminus\eta'((-\infty,0])$, chosen in a manner which does not depend on $h$, and let $z_1$ (resp.\ $z_2$) be the first (resp.\ last) point on $\partial U$ visited by $\eta'$ after time 0. Recall from Section~\ref{sec-sle-prelim} that the conditional law given $\eta'((-\infty,0])$ of the segment of $\eta'$ contained in $\ol U$ is that of a space-filling chordal SLE$_{\kappa'}$ from $z_1$ to $z_2$. Let $\wh\eta'$ be the curve obtained by skipping the bubbles filled in by this segment of $\eta'$, which is an ordinary SLE$_{\kappa'}$ from $z_1$ to $z_2$ in $U$ (the curve $\wh\eta'$ can be obtained by skipping the times contained in reverse $\pi/2$-cone excursions; see right before Example \ref{prop-dimcalc-gasket} for the definition). 

Since $\wh\eta'$ is obtained by skipping the bubbles filled in by $\eta'$ during a certain interval of times, we find that if $s_1 < s_2$ are such that $z_0:=\wh\eta'(s_1) = \wh\eta '(s_2)$, then there exists $t_0 \in\BB Q \cap [0,\infty)$ such that the following is true. There is a connected component $U_{t_0}$ of $\BB C\setminus \eta'((-\infty,t_0])$, such that if $\tau_{t_0}$ is the first time that $\wh\eta' $ enters $U_{t_0}$, then $s_1 < \tau_{t_0} < s_2$ and $z_0$ is a point of intersection between the chordal SLE$_{\kappa'}$ in the bead $U_{t_0}$ and the boundary of the domain (in particular, $U_{t_0}$ is the connected component of $U\setminus \wh\eta'([0,\tau_{t_0}])$ with $z_2$ on its boundary). Furthermore, there is a closed arc of $\bdy U_{t_0}$ containing the initial point $\wh\eta'(\tau_{t_0})$ such that every intersection point of the chordal SLE$_{\kappa'}$ with this arc is a double point of $\wh\eta'$; the reason this property does not hold for all intersection points between the chordal SLE$_{\kappa'}$ and $\bdy U_{t_0}$ is that some of these intersection points will be contained in $\bdy U$. 

Since $\eta'(\cdot-t_0) \eqD \eta'$ modulo rotation and scaling~\cite[Theorem~1.9]{wedges}, to show that the double point dimension of $\wh\eta'$ is a.s.\ given by~\eqref{eqn-dimcalc-4}, it suffices to show that the dimension of the intersection of $\wh\eta'$ with any non-trivial arc of $\bdy U$ containing $z_1$ is a.s.\ given by~\eqref{eqn-dimcalc-4}. 

Let $\psi : [0,\infty) \rta \BB C$ be the parameterization of the left boundary of $\eta'((-\infty,0])$ with $\psi(0) = 0$ and such that for each $0 \leq u < v$ we have that the quantum length of the segment from $\psi(u)$ to $\psi(v)$ is equal to $v-u$.  Let $a>0$ be such that $\psi(a)$ is where the left boundary of $\eta'((-\infty,0])$ first hits $\partial U$.  Let $\varphi \colon \R \times [0,\pi] \to U$ be the unique conformal transformation which takes $-\infty$ (resp.\ $+\infty$) to the initial (resp.\ terminal) point of $\wh{\eta}'$ such that the field $\wh{h} = h \circ \varphi + Q\log |\varphi'|$ on $\R \times [0,\pi]$ has the horizontal translation chosen so that the supremum of its projection onto the space of functions which are constant on vertical lines is hit at $u=0$.  For each $M\in\R$, let $A_M = \psi^{-1}(\wh{\eta}' \cap \varphi((-\infty,M]\times\{\pi\}))$.  Let $(\wt L_t,\wt R_t)_{t\in\R}$ be the time-reversal of $(R_t, L_t)_{t\in\BB R}$, so $(\wt L_t,\wt R_t)_{t\in\R}$ is the pair of Brownian motions encoding the time-reversal of $\eta'$ on top of the independent quantum cone.  
If we define 
\eqbn
\wh X = \{t\geq 0\,:\,\wt R_t=\inf_{s\in[0,t]}\wt R_s,\  \varphi^{-1}(\wh\eta'(t)) \in (-\infty,M]\times \{\pi \} \}
\eqen
then $\eta'(\wh X) = \psi(A_M)$. 

We claim that for any $M>0$ the law of the set $A_M-a$ is absolutely continuous with respect to the law of the range of a stable subordinator of index $\kappa'/4-1$ stopped at some positive time. To see this, we first describe the law of the triple $(U , h|_U , \wh\eta')$ viewed as a curve-decorated quantum surface (i.e., modulo conformal maps). Let $U_{\mcl Q}$ be the first bead of $\eta'([0,\infty))$ such that the sum of the quantum masses of the previous beads (including $U_{\mcl Q}$) is at least 1 and let $\wh\eta'_{\mcl Q}$ be the associated chordal SLE$_{\kappa'}$ curve between its marked points. Since $U_{\mcl Q}$ is chosen in a manner which does not depend on the particular embedding of the quantum surface parametrized by $\eta'([0,\infty))$ into $\BB C$, it follows that the conditional law of the curve-decorated quantum surface $(U_{\mcl Q} , h|_{U_{\mcl Q}} , \wh\eta'_{U_{\mcl Q}})$ given its quantum area and boundary length is that of a single bead of a weight-$2-\gamma^2/2$ quantum wedge decorated by an independent chordal SLE$_{\kappa'}$ curve between its marked points. Since $U$ is independent from $h$, it a.s.\ holds with positive conditional probability given $\eta'$ (viewed modulo monotonte parametrization) and $(U , h|_U , \wh\eta')$ that $ (U_{\mcl Q} ,\wh\eta'_{\mcl Q}) = (U,\wh\eta')   $. Hence the law of $(U , h|_U , \wh\eta')$ is absolutely continuous with respect to the law of $(U_{\mcl Q} , h|_{U_{\mcl Q}} , \wh\eta'_{U_{\mcl Q}})$. 

The law of the left and the right boundary length process for $\wh\eta '$ when $\wh\eta'$ is parameterized by quantum natural time and run until it exits $\varphi( (-\infty,M]\times[0,\pi] )$ is absolutely continuous with respect to a $\kappa'/4$-stable L\'evy process with only negative jumps stopped at a certain time by the preceding paragraph,~\cite[Corollary~1.19]{wedges}, and since when mapping $U$ to the strip as above, the law of the restriction of the field to $(-\infty,M]\times[0,\pi]$ is absolutely continuous with respect to the analogous restricted field for a thick quantum wedge of weight $\frac{3\gamma^2}{2}-2$ \cite[Section~4.4 and Footnote~1]{wedges}. Hence our claim follows by \cite[Lemma~VIII.1]{bertoin-book}.

The set of times $t\geq 0$ when $\wt R_t=\inf_{s\in[0,t]}\wt R_s$ has the law of the range of a stable subordinator of index $1/2$. Hence the law of $\wh X$ is absolutely continuous with respect to the law of the range of the composition of two (not necessarily independent) subordinators of index $1/2$ and $\kappa'/4-1$, respectively. By the uniform dimension transformation result for subordinators~\cite[Theorem 4.1]{hawkes-uniform}, a.s.\ $\dim_{\mcl H}(\wh X)=\kappa'/8-1/2$. By Theorem~\ref{thm-dim-relation}, a.s.\ the dimension of the intersection of $\wh\eta'$ with any non-trivial segment of $\bdy U$ is given by~\eqref{eqn-dimcalc-4}. 
Recalling the argument at the beginning of the proof, a.s.\ the set of double points of $\wh\eta'$ is a.s.\ given by~\eqref{eqn-dimcalc-4}. Since chordal SLE$_{\kappa'}$ a.s.\ does not have any boundary double points~\cite[Remark 5.3]{miller-wu-dim}, the double point dimension for other types of SLE$_{\kappa'}$ is obtained via local absolute continuity, as in the preceding examples.
\end{proof}

\begin{figure}[ht!]
\begin{center}
\includegraphics[scale=0.92]{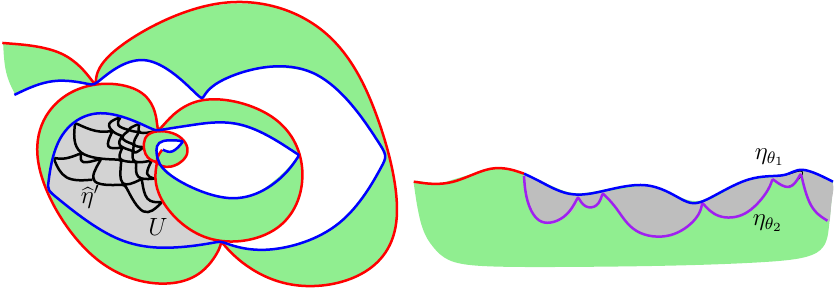}
\end{center}
\caption{The left and right figure illustrate Example~\ref{prop-dimcalc-4} and Example~\ref{prop-dimcalc-6}, respectively. Both figures illustrate $\eta'$ stopped at time zero, with the region $\eta'((-\infty,0])$ shown in green.  For $\kappa' \in (4,8)$, the $\SLE_{\kappa'}$ double points have the same Hausdorff dimension as the points of intersection between the chordal $\SLE_{\kappa'}$ $\wh\eta'$ in $U$ and the left frontier of $\eta'((-\infty,0])$. The intersection points between the two GFF flow lines on the right figure are calculated by using that the gray region between the curves is a thin quantum wedge.} 
\label{fig-examples}
\end{figure}

The dimension of the intersection points of two GFF flow lines (in the sense of \cite{ig1}) was first obtained in \cite[Theorem~1.5]{miller-wu-dim}. In our calculation below we assume the flow lines are generated from the same GFF as $\eta'$, and we assume without loss of generality that the first flow line is given by the right boundary of $\eta'([0,\infty))$.  As we will explain just below, we can sample from the law of the flow line intersection points by considering a Bessel process $Y$ which encodes the quantum wedge which lies between the two flow lines. The Bessel process $Y$ that we consider has a different dimension than the Bessel process in \cite[Section~4.4]{wedges}, since the excursions of $Y$ encode quantum boundary lengths rather than quantum areas. We derive the dimension of $Y$ in the following lemma.

\begin{lem}
Let $\mcl W$ be a (thin) quantum wedge of weight $W\in(0,\gamma^2/2)$. There is a Bessel process $Y$ of dimension $d=4W/\gamma^2$ with the following property. The ordered sequence of quantum disks of $\mcl W$ correspond to the ordered sequence of the excursions of $Y$ from $0$, and the right quantum boundary length of each quantum disk is identical to the length of the corresponding excursion of $Y$ from $0$.
\label{prop-dimcalc-7}
\end{lem}

\begin{proof}
By \cite[Definition~4.15]{wedges}, $\mcl W$ is a Poissonian chain of beads, corresponding to the ordered sequence of excursions from $0$ of a Bessel process $\wt Y$ of dimension $\wt d=1+2W/\gamma^2$, such that each surface can be parameterized as follows. Let $\mcl S=\BB R\times[0,\pi]$, and for each excursion $e$ of $\wt Y$, let $(\mcl S,h^e)$ be a parameterization of the corresponding bead of $\mcl W$. The distribution $h^e$ is given by $h^e=h^{\dagger}+h^{0}$, where $h^{\dagger}((t,u))=\wt X^e_t$ for all $(t,u)\in\mcl S$, $\wt X^e_t$ is equal to the reparameterization  of $2\gamma^{-1}\log(e)$ to have quadratic variation $2dt$, and $h^{0}$ is independent of $h^\dagger$ and equal in law to the projection of a free boundary GFF onto the space of distributions with mean $0$ on each vertical line.  Let $a=W/\gamma-\gamma/2$. If we take the horizontal translation so that $\wt X^e$ reaches its supremum at $t=0$, then by \cite[Propositions 3.4 and 3.5, Remark~3.7]{wedges}, we have $\wt X^e_t=\wt B^{e,1}_{2t}+at$ for $t\geq 0$, and $\wt X^e_t=\wt B^{e,2}_{-2t}-at$ for $t<0$, where $\wt B^{e,1}$ and $\wt B^{e,2}$ are two independent Brownian motions started from $\wt b^e:=2\gamma^{-1}\log(\sup(e))$, conditioned on $\wt X^e_t\leq \wt b^e$ for all $t\in \BB R$.  (We note that since $a < 0$, this conditioning can be made sense of as in \cite[Remark~4.4]{wedges}.)

The conditional expectation of the right quantum boundary length of $(\mcl S,h^e)$ given $(\wt X^e_t)_{t\in \BB R}$ is proportional to
\begin{equation}
\int_{\BB R} \exp(\gamma \wt X^e_t/2)\,dt 
=4\int_{\BB R} \exp(\gamma X^e_t)\,dt,
\label{eqn-dimcalc-5}
\end{equation}
where $X^e_t=B^{e,1}_{2t}+2at$ for $t\geq 0$ and $X^e_t=B^{e,2}_{-2t}-2at$ for $t<0$, and $B^{e,1}$ and $B^{e,2}$ are two independent Brownian motions started from $b^e:=\wt b^e/2$, conditioned on $ X_t^e\leq b^e$ for all $t\in \BB R$. 

We will argue that we obtain a Bessel process of dimension $d:=2\wt d-2=\frac{4W}{\gamma^2}$ if we reparameterize $\exp(\gamma X^e_t/2)$ to have quadratic variation $2dt$ for each $e$, and concatenate the resulting excursions in the order given by~$\wt Y$. By~\eqref{eqn-dimcalc-5} and \cite[Remark~4.16]{wedges} this will imply that the collection of conditional expected right boundary lengths of $\mcl W$, given the projection of each $h^e$ onto the space of functions that are constant on $\{t\}\times[0,\pi]$ for all $t\in\BB R$, has the law of the excursion lengths of a Bessel process of dimension $d$. By the same argument as in the proof of \cite[Proposition~4.18]{wedges}, this implies that there is a $d$-dimensional Bessel process $Y$ such that the length of an excursion $e$ is equal to the actual quantum boundary length of the corresponding surface, hence we can conclude the proof of the lemma.

By \cite[Proposition~3.4, Lemma~3.6, and Remark~3.7]{wedges} we would obtain a Bessel process of dimension $d$ by the above procedure, given that the collection of maxima $b^e$ of the processes $X^e_t$ has the right law, since we know that the drift $\pm 2a$ of $X_t^e$ corresponds to a Bessel process of dimension $d$; see \cite[Table~1.1]{wedges}. Given an excursion $e$ define $\wt e^*:=\sup(e)$, and note that $e^*:=(\wt e^*)^{1/2}$ is the maximum of the excursion obtained by reparameterizing $\exp(\gamma X^e_t/2)$. By \cite[Remark~3.7]{wedges} the law of $\wt e^*$ can be described by considering a Poisson point process of intensity $ds\otimes u^{\wt d-3}\,du$, where $ds$ and $du$ denote Lebesgue measure on $\R_+$. A realization of the Poisson point process is a collection of points $(s,\wt e^*)$, where the second coordinate gives the maximum value of a Bessel excursion, and the Bessel excursions are ordered chronologically by the first coordinate. The collection of points $(s,e^*)=(s,(\wt e^*)^{1/2})$ has the law of a Poisson point process of intensity proportional to $ds\otimes u^{2\wt d-5}\,du=ds\otimes u^{d-3}\,du$, hence our wanted result follows.
\end{proof}

\begin{example}
Let $\theta_1, \theta_2 \in \R$ and suppose that $\theta:=\theta_1-\theta_2\in(0,\pi]$. Consider two flow lines $\eta_{\theta_i}$, $i\in \{1,2\}$, of a whole-plane GFF $\wh h$, started from $z\in\BB C$ (in the sense of~\cite{ig4}). If $\theta\in(0,\frac{\pi\kappa}{4-\kappa}\wedge \pi ]$ the Hausdorff dimension of $\eta_{\theta_1} \cap \eta_{\theta_2}$ is a.s.\ given by
\eqbn
2-\frac{1}{2\kappa}\left( \rho+\frac{\kappa}{2}+2 \right)\left(\rho-\frac{\kappa}{2}+6 \right),
\eqen
where $\rho=\theta(2-\kappa/2)/\pi-2$. If $\kappa\leq 2$ and $\theta\in [\frac{\pi\kappa}{4-\kappa},\pi]$, the flow lines a.s.\ do not intersect.
\label{prop-dimcalc-6}
\end{example}

\begin{proof}
By \cite[Theorem~1.1]{ig4}, $\eta_{\theta_1}$ has the law of a whole-plane $\SLE_{\kappa}(2-\kappa)$. By \cite[Theorem~1.11]{ig4}, the conditional law of $\eta_{\theta_2}$ given $\eta_{\theta_1}$ is that of a chordal $\SLE_\kappa(\rho_1 ; \rho_2)$ from $0$ to $\infty$ in $\BB C\backslash \eta_{\theta_1}$, where $\rho_i=W_i-2$, $W_1=\theta(2-\gamma^2/2)/\pi\geq 0$, $W_2=W-W_1\geq 0$ and $W=4-\gamma^2$. Note that $\rho_1=\rho$, with $\rho$ as defined in the statement of the example.  Let $(\BB C,h,0,\infty)$ be a weight-$W$ quantum cone (equivalently, a $\gamma$-quantum cone) independent of $\eta_{\theta_i}$, $i\in\{1,2\}$, and assume w.l.o.g.\ that $z=0$. By \cite[Theorem~1.2]{wedges}, the quantum surface $\mcl W_1$ (resp.\ $\mcl W_2$) having $\eta_{\theta_1}$ as left (resp.\ right) boundary and $\eta_{\theta_2}$ as right (resp.\ left) boundary, is a quantum wedge of weight $W_1$ (resp.\ $W_2$). If $\kappa\leq 2$ and $\theta\in[\frac{\pi\kappa}{4-\kappa},\pi]$, $\mcl W_1$ and $\mcl W_2$ are thick wedges, implying that $\eta_{\theta_1}\cap\eta_{\theta_2}=\{0\}$. Assume $\kappa> 2$ or $\theta\not\in[\frac{\pi\kappa}{4-\kappa},\pi]$. By Lemma~\ref{prop-dimcalc-7}, there is a Bessel process $\wh B$ of dimension $d=4W_1/\gamma^2$, such that the ordered lengths of its excursions from $0$, are identical to the ordered sequence of the right boundary lengths of the bubbles. By the comment right after \cite[Proposition~2.2]{bertoin-sub}, there is a subordinator $S_1$ of index $\alpha_1=1-d/2$ such that the zero set of $\wh B$ is equal to its range.

Let $\eta'$ be a whole-plane space-filling $\op{SLE}_{\kappa'}$ from $\infty$ to $\infty$ as above. The right boundary of $\eta'([0,\infty))$ has the law of an $\SLE_\kappa(2-\kappa)$, so we can assume w.l.o.g.\ that $\eta_{\theta_1}$ is equal to the right boundary of $\eta'([0,\infty))$.

Let $\wh X\subset [0,\infty)$ be the set of times that $\eta'|_{[0,\infty)}$ visits a point in $\eta_{\theta_1}\cap\eta_{\theta_2}$ for the first time.  Note that $\eta'|_{[0,\infty)}$ visits a point in $\eta_{\theta_1}$ exactly when $R$ is equal to its running infimum since time zero, and $\eta'$ visits a point in the intersection $\eta_{\theta_1}\cap\eta_{\theta_2}$ when the additional condition $-R_t \in S_1(\BB R_+)$ holds.  Hence,
\eqbn
\dim_{\mcl H} (Z_t)_{t\in \wh X} 
= \dim_{\mcl H} \{(L_t, R_t)\,: -R_t\in S_1(\BB R^+),\, R_t = \inf_{0\leq s\leq t} R_s\}.
\eqen
The set of times $t$ where $R_t=\inf_{0\leq s\leq t} R_s$, is equal to the range of a stable subordinator $S_2$ of index $\alpha_2=1/2$, such that $S_2(x)$ is the first time that $R_t$ hits $-x$, for any $x\geq 0$.  It follows that $\wh X=S_2(S_1(\BB R^+))$.  By \cite[Theorem~4.1]{hawkes-uniform} it holds a.s.\ that $\dim_{\mcl H}(\wh X) = \alpha_1\alpha_2=1/2-W_2/\gamma^2=1/2-(\rho+2)/\kappa$.
Applying Theorem~\ref{thm-dim-relation} completes the proof.
\end{proof}

In our final application of Theorem~\ref{thm-dim-relation} we will use the theorem in the reverse direction as compared to the examples above. We use the dimension of the CLE$_{\kappa'}$ gasket determined in \cite{ssw09,msw-gasket} for $\kappa'\in(4,8)$ to calculate the dimension of times not contained in any left $\pi/2$-cone intervals for a correlated Brownian motion. Recall that a left (resp.\ right) $\pi/2$-cone interval for $Z|_{[0,\infty)}$ is a time interval $(s,t)\subset[0,\infty)$ such that $R_u\geq R_s$ and $L_u\geq L_s$ for all $u\in(s,t)$, and such that $R_s=R_t$ (resp.\ $L_s=L_t$).  Furthermore a reverse $\pi/2$-cone interval for $Z|_{[0,\infty)}$ is a time interval $(s,t)\subset[0,\infty)$ such that $(-t,-s)$ is a cone interval for the time-reversal $(Z_{-t})_{t\leq 0}$ of $Z$.
\begin{example} \label{prop-dimcalc-gasket}
Let $\kappa'\in(4,8)$.
Consider $(Z_t)_{t\geq 0}$ and let $\wh X$ be the set of times that are not contained in any left cone intervals, i.e.,
\eqb
\wh X = [0,\infty)\backslash \{u\geq 0\,:\,\exists \text{ left cone interval }(s,t), 0\leq s<t, \text{ s.t. }u\in(s,t)\}.
\eqe
Then $\op{dim}_{\mcl H}(\wh X)=\frac 12+\frac{\kappa'}{16}$.
\end{example}
\begin{proof}
It is sufficient to consider a reverse right cone interval $(t_1,t_2)$, $0<t_1<t_2<\infty$, and prove that the set
\eqb 
\wh X' = (t_1,t_2)\backslash \{u\in(t_1,t_2)\,:\,
\exists 
\text{\,a reverse left cone interval }
(s_1,s_2)\subset(t_1,t_2), s_1<s_2, \text{ s.t. }u\in(s_1,s_2)\}
\eqe
satisfies $\op{dim}_{\mcl H}(\wh X')=\frac 12+\frac{\kappa'}{16}$. This is sufficient by invariance in law under time-reversal of Brownian motion, since any compact subset of $[0,\infty)$ is a.s.\ contained in some reverse right $\pi/2$-cone interval (possibly starting before time $0$), and since the interval $[0,\infty)$ a.s.\ contains some reverse right $\pi/2$-cone interval by \cite{evans-cone}. 
 As mentioned in Section~\ref{sec-intro-background}, $\eta'$ encodes a whole-plane CLE$_{\kappa'}$. The interior $U$ of the image of the reverse right cone interval $(t_1,t_2)$ under $\eta'$ is a ``bubble" disconnected from $\infty$ by $\eta'$, with the boundary traced in the clockwise direction by $\eta'$. The restriction of the CLE$_{\kappa'}$ to $U$ has the law of a CLE$_{\kappa'}$ in $U$. It follows from e.g.\ \cite{shef-cle,msw-gasket,gwynne-miller-cle} that the interiors of the outermost CLE$_{\kappa'}$ loops in $U$ associated with the space-filling $\SLE_{\kappa'}$ $\eta'$ correspond to outermost reverse left cone excursions of $Z|_{[t_1,t_2]}$. The gasket of the $\op{CLE}_{\kappa'}$ in $U$ is the set of points in $U$ not contained in the interiors of any of these loops. The result now follows from Theorem~\ref{thm-dim-relation}, since we know by \cite{msw-gasket,ssw09} that the CLE$_{\kappa'}$ gasket a.s.\ has dimension $2-(8-\kappa')(3\kappa'-8)/(32\kappa')$.
\end{proof}

The final example in this section will be an application of \cite[Theorem~4.1]{rhodes-vargas-log-kpz} to calculate the Hausdorff dimension of the points of intersection between the real line and a chordal SLE$_{\kappa}(\rho)$, $\kappa>0$, in the upper half-plane where $\rho$ is in the range of values in which the process does not fill the boundary. This Hausdorff dimension was first obtained in \cite{alberts-shef-bdy-dim} for the special case $\rho=0$ and $\kappa>4$, and the formula was proved for general values of $\rho$ and $\kappa$ in \cite[Theorem~1.6]{miller-wu-dim}. Our main result Theorem~\ref{thm-dim-relation} cannot be used in this setting, since it only applies to SLE and CLE sets in the interior of a domain. 
\begin{example}
Let $\kappa >0$, $\kappa\not=4$, and $\rho\in(-2\vee (\frac{\kappa}{2}-4),\frac{\kappa}{2}-2)$, and consider an SLE$_{\kappa}(\rho)$ $\eta$ on $\BB H$ from $0$ to $\infty$ with force-point at $0^+$. Almost surely, 
\eqbn
\dim_{\mcl H}(\eta\cap\BB R_+)=1-\frac{1}{\kappa}(\rho+2)\left(\rho+4-\frac{\kappa}{2}\right).
\eqen
\label{ex1}
\end{example}
\begin{proof}
First we consider the case $\kappa>4$, hence we will write $\kappa' $ instead of $\kappa$ and $\eta'$ instead of $\eta$.  Let $\eta'$ be a chordal $\SLE_{\kappa'}$ curve from 0 to $\infty$ in $\BB H$ on top of an independent quantum wedge $(\BB H,h,0,\infty)$ of weight $\frac{3\gamma^2}{2}-2+\frac{\gamma^2}{4}\rho$. Let $D$ be the open subset of $\BB H$ which is between the right boundary of $\eta'$ and $[0,\infty)$. By \cite[Theorem~1.16]{wedges}, $(D,h,0,\infty)$ has the law of a thin quantum wedge of weight $W=\gamma^2-2+\frac{\gamma^2}{4}\rho$. Defining $\wh X\subset[0,\infty)$ by $\wh X:=\{\nu([0,x])\,:\,x\in \eta'\cap\BB R_+ \}$ it follows by Lemma~\ref{prop-dimcalc-7} and \cite[Proposition~2.2]{bertoin-sub} that $\wh X$ has the law of the range of a stable subordinator of index $\kappa'/4-1-\rho/2$, hence $\op{dim}_{\mcl H}(\wh X)=\kappa'/4-1-\rho/2$. An application of~\eqref{eqn-dimcalc-1} completes the proof. Note that we may apply this formula to the field $h$ since, if $h$ is given the circle-average embedding, say, then the restriction of $h$ to any sub-domain of $\BB H$ bounded away from 0, $\infty$, and $\bdy\BB D\cap \BB H$ is mutually absolutely continuous with respect to the corresponding restriction of a free-boundary GFF on $\BB H$ normalized to have average zero on $\partial\BB D\cap\BB H$.

We proceed by the exact same argument when $\kappa\in(0,4)$, except that we apply \cite[Theorem~1.2]{wedges} instead of \cite[Theorem~1.16]{wedges}. Alternatively, we may obtain this dimension by using the result for $\kappa'>4$ and SLE duality \cite{zhan-duality1,zhan-duality2,dubedat-duality,ig1,ig4}.
\end{proof}

\section{SLE and GFF estimates}
\label{sec-gff-esimates}

In this section we will prove various estimates for space-filling SLE and for GFFs which we will need in the sequel. 

\subsection{Space-filling SLE absorbs a ball with positive probability}
\label{sec-sle-ball}

Throughout this subsection, we fix $\kappa' > 4$ and let $\eta'$ be a whole-plane space-filling SLE$_{\kappa'}$ from $\infty$ to $\infty$, parameterized by Lebesgue measure and satisfying $\eta'(0) = 0$. Our goal is to prove the following lemma which (together with a multi-scale argument) will be used in the next subsection to argue that $\eta'$ is very unlikely to travel a long distance without absorbing a large Euclidean ball (see Lemma~\ref{prop-hitting-bubble}). 
Define
\eqb \label{eqn-sle-exit}
T_\rho := \inf\left\{ t\geq 0 : \eta'([0,t])\not\subset B_\rho(0) \right\} ,\quad \forall \rho > 0.
\eqe 
The main result of this subsection is the following lemma.

\begin{lem} \label{prop-sle-ball-pos}
There are constants $\delta , p \in (0,1)$ depending only on $\kappa'$ such that the following is true. For any $\ep\in(0,1)$, it a.s.\ holds with conditional probability at least $p$ given $\eta'|_{[0,T_1]}$ that $\eta'([T_1,T_{1+\ep} ])$ contains a ball of radius at least $\delta \ep$.
\end{lem}

The proof of Lemma~\ref{prop-sle-ball-pos} proceeds via a combination of elementary complex analysis and facts from imaginary geometry~\cite{ig1,ig2,ig3,ig4}. See Figure~\ref{fig-sle-ball-pos} for an illustration and outline of the proof.

\begin{figure}[ht!]
\begin{center}
\includegraphics[scale=0.92]{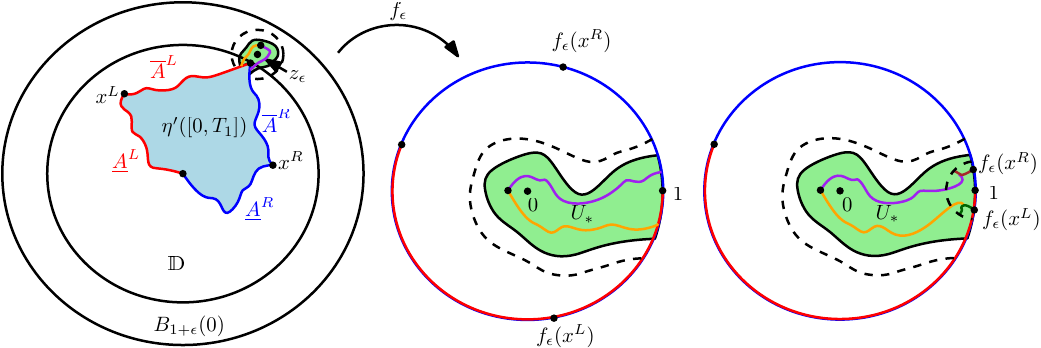}
\end{center}
\caption{\label{fig-sle-ball-pos}
Illustration of the proof of Lemma~\ref{prop-sle-ball-pos}. We seek to show that $\eta'$ absorbs the ball $B_{\delta\ep}(z_\ep)$ for some $\delta = \delta(\kappa')>0$ (not shown) before exiting $B_{1+\ep}(0)$. To do this, we study the conformal map $f_\ep : \BB C\setminus K_{T_1} \rta \BB D$ taking $z_\ep$ to 0 and $\eta'(T_1)$ to 1. The set $U_*$ is a neighborhood of a path from $0$ to 1 in $\BB D$ whose pre-image under $f_\ep$ is contained in $V_\ep \subset B_{1+\ep}(0)$ (in the figure, $f_\ep^{-1}(U_*)$ is contained in $B_{\ep/2}(z_\ep)$, but in general it might just be disconnected from $\infty$ by $K_{T_1}\cup B_{\ep/2}(z_\ep)$). Consider the flow lines started from a point near $z_\ep$ which form the left and right boundaries of $\eta'$ at the time when it hits this point. These flow lines and their images under $f_\ep$ are shown in orange and purple. In the case when $f_\ep(x^L)$ and $f_\ep(x^R)$ are at macroscopic distance from 1 (left and middle panels) we can use a local absolute continuity argument to show that with uniformly positive conditional probability given $\eta'|_{[0,T_1]}$, the orange and purple flow lines in the middle figure together with $\bdy\BB D$ form a pocket contained in $U_*$ which itself contains $f_\ep(B_{\delta \ep}(z_\ep))$. The pre-image of this pocket under $f_\ep$ is contained in $\eta'([T_1,T_{1+\ep}])$. If instead $f_\ep(x^L)$ and $f_\ep(x^R)$ are very close to $1$ (as shown in the right panel) we instead need to grow flow lines started from $f_\ep(x^L)$ and $f_\ep(x^R)$ (dark green and brown), which can equivalently be described as small segments of $f_\ep(\eta_0^L\setminus \eta'([0,T_1]) )$ and $f_\ep(\eta_0^R \setminus \eta'([0,T_1]) )$. If only one of $f_\ep(x^L)$ or $f_\ep(x^R)$ is very close to 1, we only need to grow a single extra flow line. On a uniformly positive probability event, if we map the complementary connected component containing 0 of the green and brown flow lines to $\BB D$, the images of the tips will be at uniformly positive distance from 1. This gives us a configuration which looks like the one in the middle figure, which allows us to argue that with uniformly positive probability, the union of all 4 flow lines in the right panel and $\partial\D$ forms a pocket surrounding $f_\ep(B_{\delta\ep}(z_\ep) ) $. The pre-image of this pocket under $f_\ep$ will again be contained in $\eta'([T_1,T_{1+\ep}])$. }
\end{figure}

We remark that in the case when $\kappa' \in (4,8]$, one can give a somewhat simpler argument (which does not directly use imaginary geometry). The reason is that in this case, the time reversal of $\eta'|_{[0,\infty)}$ traces points in the same order as the associated SLE$_{\kappa'}(\kappa'-6)$ curve started from $\infty$, and the time reversal of this curve is also an SLE$_{\kappa'}(\kappa'-6)$~\cite[Theorem 1.20]{ig4}. So, we can reduce our problem to proving that if $\eta_0'$ is an SLE$_{\kappa'}(\kappa'-6)$ from 0 to $\infty$, then with uniformly positive conditional probability given $\eta_0'$ up to the first time it exits $\BB D$, it holds that $\eta_0'$ forms a bubble which contains a ball of radius $\delta \ep$ before exiting $B_{1+\ep}(0)$. This statement can be proven using basic complex analysis plus the Markov property of SLE$_{\kappa'}(\kappa'-6)$ and~\cite[Lemma 2.4]{miller-wu-dim}. 
However, this alternative argument does not work in the case when $\kappa' > 8$ since in this case the whole-plane SLE$_{\kappa'}(\kappa'-6)$ is not reversible. In fact, the marginal law of $\eta'|_{[0,\infty)}$ is not that of an SLE$_{\kappa'}(\rho)$ for any choice of $\rho$. So, we instead need to control this curve using interior flow lines of a GFF (which form its left and right boundaries). 

We now proceed with the proof of Lemma~\ref{prop-sle-ball-pos}.
First we introduce some notation.
For $t\geq 0$, let $K_t$ be the hull generated by $\eta'([0,t])$, i.e.\ the union of $\eta'([0,t])$ and the set of points which it disconnects from $\infty$ (this hull is just $\eta'([0,t])$ if $\kappa' \geq 8$). 

Following~\cite{ig1}, we define the constants
\eqb \label{eqn-ig-param}
\kappa := \frac{16}{\kappa'} ,\quad \chi := \frac{2}{\sqrt\kappa} - \frac{\sqrt\kappa}{2} ,\quad \lambda := \frac{\pi}{\sqrt\kappa},\quad \lambda' := \frac{\pi}{\sqrt{\kappa'}} .
\eqe
We also let $\wh h$ be the whole-plane GFF viewed modulo a global additive multiple of $2\pi\chi$ which is used to construct $\eta'$ as in Section~\ref{sec-sle-prelim}. For $z\in\BB C$ we let $\eta_z^L$ (resp.\ $\eta_z^R$) be the flow line of $\wh h$ started from $z$ with angle $\pi/2$ (resp.\ $-\pi/2$), so that a.s.\ $\eta_z^L$ and $\eta_z^R$ are the left and right boundaries of $\eta'$ at the first time it hits $z$. 

The set $\bdy \eta'([0,T_1])$ can be divided into four distinguished arcs, which we denote as follows.
\begin{itemize}
\item $\ul A^L$ (resp.\ $\ul A^R$) is the arc of $\bdy K_{T_1}$ traced by $\eta_0^L$ (resp.\ $\eta_0^R$). 
\item $\ol A^L$ (resp.\ $\ol A^R$) is the arc of $\bdy K_{T_1} $ not traced by $\eta_0^L$ or $\eta_0^R$ which lies to the left (resp.\ right) of $\eta'(T_1)$. 
\end{itemize}  

Using the notation~\eqref{eqn-sle-exit}, we define the $\sigma$-algebra
\eqbn
\mcl F_1 := \sigma\left( \eta'|_{[0,T_1]} ,\, \wh h|_{\eta'([0,T_1])} \right) .
\eqen

\begin{lem} \label{prop-local-set}
The set $\eta'([0,T_1])$ is a local set for $\wh h$ in the sense of~\cite[Lemma 3.9]{ss-contour}. 
In particular, the boundary data for the conditional law of $h|_{\BB C\setminus K_{T_1} }$ given $\mcl F_1$ on each of the arcs $\ul A^L$, $\ul A^R$, $\ol A^L$, and $\ol A^R$ coincides with the boundary data of the corresponding flow line of $\wh h$ (i.e., it is given by flow line boundary conditions as described in~\cite[Figure 9]{ig4}). 
\end{lem}
\begin{proof}
We first check that $\eta'([0,T_1])$ is a local set for $\wh h$. By~\cite[Lemma 3.9, condition 1]{ss-contour}, it suffices to show that for each deterministic open set $U\subset \BB C$, the event $\{\eta'([0,T_1]) \cap U \not=\emptyset\}$ is a.s.\ determined by $h|_{\BB C\setminus U}$. For $z\in\BB C$, let $S_z^L$ (resp.\ $S_z^R$) be the first time that the flow line $\eta_z^L$ (resp.\ $\eta_z^R$) enters $U$. Since flow lines are local sets, each $\eta_z^q([0,S^q_z])$ for $q \in \{L,R\}$ and $z\in\BB C$ is a.s.\ determined by $h|_{\BB C\setminus U}$. Since the outer boundary of $\eta'$ at the first time it hits any given rational $z\in\BB C$ is equal to $\eta_z^L\cup \eta_z^R$, we see that a.s.\ $\eta'([0,T_1])$ intersects $U$ if and only if there is a $z\in\BB Q\setminus \BB D$ such that $\eta_z^L$ merges into $\eta_0^L([0,S_0^L])$ before time $S_z^L$; and the same is true with ``$R$" in place of ``$L$". This latter event is a.s.\ determined by $h|_{\BB C\setminus U}$. 

By~\cite[Lemma 3.11]{ss-contour} (applied to the local sets $\eta'([0,T_1])$ and $\eta_z^L , \eta_z^R$ for $z\in\BB Q$) and the known boundary data for interior flow lines of a whole-plane GFF~\cite[Theorem 1.1]{ig4}, we obtain the claimed description of the boundary data for the conditional law of $\wh h$ given $\mcl F_1$. 
\end{proof}

Let $z_\ep$ be the point of $\bdy B_{1+\ep/4}(0)$ closest to $\eta'(T_1)$ and let $f_\ep : (\BB C\cup \{\infty\}) \setminus K_{T_1} \rta \BB D$ be the conformal map which takes $z_\ep$ to $0$ and $\eta'(T_1)$ to 1. 
Let $V_\ep$ be the union of $B_{\ep/2}(z_\ep) \setminus \bdy K_{T_1}$ and the set of points which it disconnects from $\infty$ in $\BB C\setminus K_{T_1}$. 
Then $\bdy K_{T_1} \cap \bdy V_\ep$ is a connected arc of $\bdy K_{T_1}$. 
Let $I^L$ (resp.\ $I^R$) be the sub-arc of $\bdy K_{T_1} \cap \bdy V_\ep$ lying to the left (resp.\ right) of $\eta'(T_1)$ as viewed from $\eta'(T_1)$, looking out from $\BB D$. Note that $I^L$ (resp.\ $I^R$) need not be part of the left (resp.\ right) outer boundary of $K_{T_1}$ if all of this left (resp.\ right) outer boundary is part of $ \bdy V_\ep$.
 
There is a universal constant $q \in (0,1/2)$ such that conditional on $\eta'|_{[0,T_1]}$, a Brownian motion started from $z_\ep$ has probability at least $q$ to exit $B_{3\ep/8}(z_\ep)$ at a point outside of $B_{3\ep/8}(z_\ep)\cap\BB D$, then make a counterclockwise loop around $B_{\ep/4}(z_\ep)$ before re-entering $B_{\ep/4}(z_\ep)$ or leaving $B_{\ep/2}(z_\ep)$. If it does so, then such a Brownian motion first hits $K_{T_1}$ at a point of $I^L$ before exiting $V_\ep$. Symmetrically, Brownian motion started from $z_\ep$ has conditional probability at least $q$ to first hit $K_{T_1}$ at a point of $I^R$ before exiting $V_\ep$.

From the above estimates and the conformal invariance of Brownian motion, we infer that there is a universal constant $c_0 \in (0,1)$ such that each point of $f_\ep(\BB C\setminus (B_{\ep/2}(z_\ep) \cup K_{T_1}))$ lies at distance at least $c_0$ from 0 and each of the arcs $f_\ep(I^L)$ and $f_\ep(I^R)$ of $\bdy\BB D$ has Euclidean length at least $c_0$. In fact, the probability that a Brownian motion started from 0 hits any given ball of radius $c$ centered at a point of $\BB D\setminus B_{c_0}(0)$ before exiting $\BB D$ tends to 0 as $c\rta0$, uniformly over all possible choices of center for the ball. 
Hence the estimate of the first paragraph implies that we can find a universal constant $c \in (0,c_0/2]$ and a random path $\alpha$ in $\ol{\BB D}$ from 0 to $1$ such that $B_c(\alpha) \subset f_\ep(V_\ep)$.

Let $\mcl U$ be the collection of all simply connected open subsets of $\BB D$ which take the form $U = B_{c/100}(\beta)$ for $\beta$ a simple piecewise linear path from 0 to 1 in $\BB D$ whose linear segments all connect nearest neighbor points $(c/50) \BB Z^2$ (by slightly shrinking $c$, we can assume without loss of generality that $1/c$ is an integer, so that $1\in (c/50)\BB Z^2$). Then $\mcl U$ is a finite set and there a.s.\ exists $U_*  = B_{c/100}(\beta_*) \in \mcl U$ with $U_* \subset B_c(\alpha) \subset f_\ep(B_{\ep/2}(z_\ep)\setminus K_{T_1})$.  

By the Koebe quarter theorem, $|f_\ep'(z_\ep)| \asymp \ep^{-1}$, with universal implicit constant, so by the Koebe growth theorem $f_\ep^{-1}(B_{c/10^{10}}(0))$ contains $B_{\delta\ep}(z_\ep)$ for a universal choice of $\delta\in (0,1)$. Hence it suffices to show that $B_{c/10^{10}}(0)$ is contained in $f_\ep(\eta'([T_1,  T_{1+\ep}]))$ with uniformly positive conditional probability given $\eta'|_{[0,T_1]}$. 
 
Recall the imaginary geometry parameters from~\eqref{eqn-ig-param}. Let $\wh h_\ep := \wh h \circ f_\ep^{-1} - \chi \op{arg} (f_\ep^{-1})'$, so that $\wh h_\ep$ is similar to a GFF on $\BB D$ with Dirichlet boundary data determined by the images of the distinguished arcs $\ul A^L$, $\ul A^R$, $\ol A^L $, and $\ol A^R$ under $f_\ep$ (this boundary data is described in Lemma~\ref{prop-local-set}) except that it possesses a singularity at $f_\ep(\infty)$.

Let $\wh h_\ep^{U_*}$ be a GFF on $U_*$ with Dirichlet boundary data which coincides with that of $\wh h_\ep^*$ on $ \bdy U_* \cap\bdy \BB D$ and whose boundary data on $\bdy U_* \setminus\bdy \BB D$ is 0. As we will see, the laws of $\wh h_\ep$ and $\wh h_\ep^{U_*}$ are mutually absolutely continuous on subsets of $U_*$ at positive distance from $\bdy U_*\setminus \bdy\BB D$. Recalling that $ U_* = B_{c/100}(\beta_*)$ for the piecewise linear curve $\beta_*$, we define 
\eqb \label{eqn-domain-shrink}
U_*^r := B_{rc/100}(\beta_*),\quad \forall r \in (0,1] .
\eqe

\begin{lem} \label{prop-rn-restrict}
Let $\wh{\frk h}_\ep$ (resp.\ $\wh{\frk h}_\ep^{U_*}$) be the harmonic part of $\wh h_\ep|_{ U_*}$ (resp.\ $\wh h_\ep^{ U_*}$).
For $r \in (0,1)$, 
the conditional laws of $\wh h_\ep|_{U_*^r}$ and $\wh h_\ep^{U_*}|_{U_*^r}$ given $\mcl F_1 \vee \sigma(\wh{\frk h}_\ep, \wh{\frk h}_\ep^{U_*})$ are a.s.\ mutually absolutely continuous. Furthermore, if $M = M_r$ is the Radon-Nikodym derivative of the former with respect to the latter, then there is a $\xi = \xi(  \kappa' , r)  > 0$ such that $\BB E(M^{-\xi} \,|\,\mcl F_1) $ is bounded above by a deterministic constant depending only on $\kappa'$ and $r$. 
\end{lem}
\begin{proof}
Set $r ' = (1+r)/2$. Also let $\phi$ be a smooth bump function which equals 1 on $U_*^r$ and $0$ on $\BB C\setminus U_*^{r'}$, chosen in a manner which depends only on $U_*$ and $r$. 
The proof of~\cite[Proposition 3.4]{ig1} (applied to the zero-boundary parts of $\wh h_\ep|_{U_*}$ and $\wh h_\ep^{U_*}$) shows that if we condition on $\mcl F_1 \vee \sigma(\wh{\frk h}_\ep, \wh{\frk h}_\ep^{U_*})$ then the conditional law of $\wh h_\ep|_{ U_*^r}$ is absolutely continuous with respect to the conditional law of $\wh h_\ep^{U_*}|_{ U_*^r}$, with Radon-Nikodym derivative given by
\eqbn
M := \BB E\left[ \exp\left( (\wh h_\ep^{U_*} - \wh{\frk h}_\ep^{U_*}  , g)_\nabla -\frac12 (g,g)_\nabla \right)  \,|\, \wh h_\ep^{U_*} |_{U_*^r} ,\, \mcl F_1 \vee \sigma(\wh{\frk h}_\ep, \wh{\frk h}_\ep^{U_*}) \right] 
\eqen
where $g:=  \phi (\wh{\frk h}_\ep - \wh{\frk h}_\ep^{U_*} )$ and $(\cdot,\cdot)_\nabla$ denotes the Dirichlet inner product. 
By Jensen's inequality applied to the convex function $x\mapsto x^{-\xi}$ (in order to pass the exponent inside the conditional expectation given $\wh h_\ep|_{U_*^r}$) and since $(\wh h_\ep^{U_*}  - \wh{\frk h}_\ep^{U_*}  , g)_\nabla$ is Gaussian with variance $(g,g)_\nabla$ under the conditional law given $\mcl F_1 \vee \sigma(\wh{\frk h}_\ep, \wh{\frk h}_\ep^{U_*})$,
\begin{align} \label{eqn-rn-jensen}
\BB E\left[ M^{-\xi} \,|\, \mcl F_1 \vee \sigma(\wh{\frk h}_\ep, \wh{\frk h}_\ep^{U_*}) \right]
&\leq \BB E\left[ \exp\left(  - \xi (\wh h_\ep^{U_*}  - \wh{\frk h}_\ep^{U_*}  , g)_\nabla  + \frac{\xi}{2} (g,g)_\nabla \right)  \,|\,  \mcl F_1 \vee \sigma(\wh{\frk h}_\ep, \wh{\frk h}_\ep^{U_*}) \right] \notag \\
& =  \exp\left(  \frac{\xi^2 + \xi}{2} (g,g)_\nabla \right)   .
\end{align}

Since $\phi$ and all its derivatives vanish on $\bdy U_* \setminus \bdy\BB D$ and $\wh{\frk h}_\ep^{U_*} - \wh{\frk h}_\ep$ is harmonic and vanishes on $\bdy U_* \cap \bdy\BB D$, a short computation using integration by parts shows that
\begin{align} \label{eqn-parts}
(g,g)_\nabla 
&=   \int_{U_*} \left( \frac12 \Delta(\phi(w)^2)  - \phi(w) \Delta\phi(w) \right) (\wh{\frk h}_\ep^{U_*}(w) - \wh{\frk h}_\ep(w) )^2 \,dw  \notag \\
&\preceq \int_{U_*} (\wh{\frk h}_\ep^{U_*}(w) - \wh{\frk h}_\ep(w) )^2 \,dw
\end{align} 
with the implicit constant depending only on $\phi$, and hence only on $U_*$ and $r$.
By~\cite[Lemma 6.4]{ig1} (applied to the conditional field given $\mcl F_1$), for each $w \in U_*^{r'}$ the conditional law given $\mcl F_1$
of each of $\wh{\frk h}_\ep(w) $ and $\wh{\frk h}_\ep^{U_*}(w) $ is Gaussian with variance bounded above by a universal constant and mean bounded above in absolute value by a constant depending only on $\kappa'$ and $r$ (for this latter statement, we use that the boundary data for each of $\wh h_\ep$ and $\wh h_\ep^{U_*}$ is bounded).
From this and Jensen's inequality applied to the exponential function, we infer that the conditional expectation given $\mcl F_1$ of the right side of~\eqref{eqn-rn-jensen} is a.s.\ finite and bounded above by a constant depending only on $U_*$, $\kappa'$, and $r$ provided $\xi$ is chosen sufficiently small depending only on $U_*$, $\kappa'$, and $r$. Since there are only finitely many possible realizations $U \in \mcl U$ of $U_*$, we obtain the statement of the lemma by taking a minimum over all such $ U $. 
\end{proof}

\begin{proof}[Proof of Lemma~\ref{prop-sle-ball-pos}]
Let $x^L$ (resp.\ $x^R$) be the a.s.\ unique point of $\ul A^L\cap \ol A^L$ (resp.\ $\ul A^R\cap \ol A^R$) and with $U_*^{1/2}$ as in~\eqref{eqn-domain-shrink} let
\eqbn
H_\ep := \left\{ f_\ep(x^L) ,\, f_\ep(x^R) \notin \bdy U_*^{1/2} \right\} \subset \mcl F_1 .
\eqen 
On $H_\ep$, the conditional law given $\mcl F_1$ of the auxiliary GFF $\wh h_\ep^{U_*}$ depends only on the choice of $U_*$ (i.e., the boundary data of $\wh h_\ep$ does not depend on $\mcl F_1$). 

Let $\wh\eta_\ep^{L,U_*}$ (resp.\ $\wh\eta_\ep^{R,U_*}$) be the flow line of $\wh h_\ep^{U_*}$ started from $-c/500$ with angle $\pi/2$ (resp.\ $-\pi/2$), as in~\cite[Theorem 1.1]{ig4}. We stop $\wh\eta_\ep^{L,U_*}$ (resp.\ $\wh\eta_\ep^{R,U_*}$) at the first time it hits $f_\ep(\ol A_\ep^L)\cap \bdy U_*$ (resp.\ $f_\ep(\ol A_\ep^R)\cap \bdy U_*$). Let $G_\ep^{U_*}$ be the event that the region whose boundary is formed by the left side of $\wh\eta_\ep^{L,U_*}$, the right side of $\wh\eta_\ep^{R,U_*}$, and $\bdy\BB D\cap \bdy U_*$ contains $B_{c/10^{10} }(0)$ and is contained in $U_*^{1/2}$. By~\cite[Lemma~3.9]{ig4} applied to $\wh\eta_\ep^{L,U_*}$ and then~\cite[Lemma~2.5]{miller-wu-dim} applied to the conditional law of $\wh\eta_\ep^{R,U_*}$ given $\wh\eta_\ep^{L,U_*}$, we infer that a.s.\ $\BB P(G_\ep^{U_*} \,|\, \mcl F_1)  > 0$ on $H_\ep$. Since this conditional probability depends only on $U_*$ on $H_\ep$ and there are only finitely many possible choices of $U_*$, we can find a $p_0 \in (0,1)$ depending only on $\kappa'$ such that a.s.\ $\BB P(G_\ep^{U_*} \,|\, \mcl F_1)\BB 1_{H_\ep} \geq p_0 \BB 1_{H_\ep} $. 

To transfer this to an estimate for $\wh h_\ep$ (rather than for $\wh h_\ep^{U_*}$), let $\wh\eta_\ep^{L }$ (resp.\ $\wh\eta_\ep^{R }$) be the flow line of $\wh h_\ep $ started from $-c/500$ with angle $\pi/2$ (resp.\ $-\pi/2$), stopped at the first time it hits $f_\ep(\ul A_\ep^L \cup \ol A_\ep^L)$ (resp.\ $f_\ep(\ul A_\ep^R \cup \ol A_\ep^R$), if this time is finite.  Equivalently, $\wh\eta_\ep^L = f_\ep\circ \eta_{f_\ep^{-1}(-c/500)}^L $, stopped at the first time it hits merges into the left boundary of $\eta'([0,T_1])$; and similarly for $\wh\eta_\ep^R$. 
Let $G_\ep$ be the event that $ \wh\eta_\ep^L$ and $\wh\eta_\ep^R$ are contained in $U_*^{1/2}$; and the region whose boundary is formed by the left side of $\eta_\ep^{L,U_*}$, the right side of $\eta_\ep^{R,U_*}$, and $\bdy\BB D\cap \bdy U_*$ contains $B_{c/10^{10} }(0)$. 
 
We will now deduce from the previous two paragraphs and Lemma~\ref{prop-rn-restrict} that there is a $p  \in (0,1)$, depending only on $\kappa'$, such that a.s.\ $\BB P(G_\ep   \,|\, \mcl F_1) \BB 1_{H_\ep} \geq p  \BB 1_{H_\ep} $. To see this, define the harmonic parts $\wh{\frk h}_\ep, \wh{\frk h}_\ep^{U_*}$ and the Radon-Nikodym derivative $M$ as in Lemma~\ref{prop-rn-restrict} (the latter with $r = 1/2$). Since $\BB P(G_\ep^{U_*} \,|\, \mcl F_1)\BB 1_{H_\ep} \geq p_0 \BB 1_{H_\ep} $, there is a $p_1 = p_1(\kappa') \in (0,1)$ such that it a.s.\ holds with conditional probability at least $p_1$ given $\mcl F_1 $ that 
\eqbn
\BB P\left(  G_\ep^{U_*} \,|\, \mcl F_1 \vee \sigma(\wh{\frk h}_\ep, \wh{\frk h}_\ep^{U_*}) \right) \BB 1_{H_\ep} \geq p_1  \BB 1_{H_\ep} .  
\eqen
Since $\BB E(M^{-\xi} \,|\,\mcl F_1) $ is bounded above by a constant depending only on $\kappa'$, we can find $b > 0$ depending only on $\kappa'$ such that with conditional probability at least $1-p_1/2$ given $\mcl F_1$, it holds that
\eqbn
\BB P\left( M \geq b \,|\, \mcl F_1 \vee \sigma(\wh{\frk h}_\ep, \wh{\frk h}_\ep^{U_*}) \right) \BB 1_{H_\ep} \geq ( 1-p_1/2 ) \BB 1_{H_\ep} .
\eqen
Since flow lines are determined locally by the field, on $H_\ep$ it holds with conditional probability at least $p_1/2$ given $\mcl F_1$ that
\eqbn
\BB P\left( G_\ep \,|\, \mcl F_1 \vee \sigma(\wh{\frk h}_\ep, \wh{\frk h}_\ep^{U_*}) \right) 
= \BB E\left( M \BB 1_{G_\ep^{U_*}} \,|\, \mcl F_1 \vee \sigma(\wh{\frk h}_\ep, \wh{\frk h}_\ep^{U_*}) \right)
\geq \frac12 b p_1 .
\eqen
Taking expectations conditional on $\mcl F_1$ proves our claim with $p = b p_1^2/4$.

Since $U_*\subset f_\ep(B_{\ep/2}(z_\ep)\setminus\BB D) $, it follows from the definition of space-filling SLE that on $G_\ep$ the region in the definition of $G_\ep$ is contained in the interior of $f_\ep(\eta'([T_1,T_{1+\ep}])$. Since we have chosen $\delta>0$ so that $B_{\delta\ep}(z_\ep)\subset f_\ep^{-1}(B_{c/10^{10}}(0))$, we infer from the preceding paragraph that $\BB P(F_\ep   \,|\, \mcl F_1) \BB 1_{H_\ep} \geq p  \BB 1_{H_\ep} $.  

It remains to treat the case when $H_\ep$ does not occur. The idea is to extend $K_{T_1}$ by growing small segments of $\eta_0^L$ and $\eta_0^R$, beyond the ones contained in $\eta'([0,T_1])$, to get a larger hull for which an analog of $H_\ep$ occurs. See the right panel of Figure~\ref{fig-sle-ball-pos} for an illustration. For simplicity, suppose that $H_\ep$ does not occur and that both $f_\ep(x^L)$ and $ f_\ep(x^R) $ are contained in $ \bdy U_*^{1/2}$ (the case when only one of these points is contained in $\bdy U_*$ is treated similarly by extending only one flow line instead of two). 
Let $\rng\eta_\ep^L$ (resp.\ $\rng\eta_\ep^R$) be the flow line of $\wh h_\ep$ started from $x^L$ (resp.\ $x^R$) with angle $\pi/2$ (resp.\ $-\pi/2$) targeted at $-1$, say, and let $\rng S_\ep^L$ (resp.\ $\rng S_\ep^R$) be its exit time from $ B_{c/100}(1) \subset U_*$. Note that by uniqueness of flow lines (c.f.~\cite[Theorem~2.4]{ig1}) $\rng\eta_\ep^L|_{[0,S_\ep^L]}$ is an initial segment of $ f_\ep(\eta_0^L\setminus K_{T_1} )$ and similarly with ``$R$'' in place of ``$L$''.  Let $\rng D$ be the connected component containing $0$ of $\BB D\setminus ( \rng\eta_\ep^L([0,\rng S_\ep^L]) \cup \rng\eta_\ep^R([0,\rng S_\ep^R]) )$.

Fix a small constant $a> 0$, to be chosen later, and let $\rng E_\ep$ be the event that the harmonic measure from 0 in $\rng D$ of each of the left side of $\rng\eta_\ep^L([0,S_\ep^L])$ and the right side of $\rng \eta_\ep^R([0,\rng S_\ep^R])$ is at least $a$. Using Lemma~\ref{prop-rn-restrict} and two applications of~\cite[Lemma~2.4]{miller-wu-dim} (which we emphasize does not depend on the particular location of the force points) and a similar argument to the one above, we infer that there is a universal choice of $a>0$ and a universal constant $\rng q \in (0,1)$ such that $\BB P(\rng E_\ep \,|\, \mcl F_1) \geq \rng q$ a.s.\ on the event $\left\{ f_\ep(x^L) ,\, f_\ep(x^L) \in \bdy U_*^{1/2} \right\}$. 

Let $\rng f_\ep : \rng D \rta \BB D$ be the conformal map which fixes 0 and such that $\rng f_\ep^{-1}(1)$ is equal to 1 if $\kappa'\geq 8$ or to the last (in chronological order along either curve) intersection point of the right side of $ \rng\eta_\ep^L|_{[0, \rng S_\ep^L]} $ and the left side of $\rng\eta_\ep^R|_{[0, \rng S_\ep^R]} $ if $\kappa' \in (4,8)$. 
Also let $\rng h_\ep := \wh h_\ep \circ \rng f_\ep^{-1} - \chi \op{arg}(f_\ep^{-1})'$. The field $\rng h_\ep$ has the same boundary data along the image of the left side of $\rng\eta_\ep^L([0,S_\ep^L])$ as it does along $\ol A_\ep^R$, and similarly with ``$R$" in place of ``$L$". 

If we apply exactly the same argument as in the case when $H_\ep$ occurs but with the field $\rng h_\ep$ in place of the field $\wh h_\ep$, then pull back to $\rng D$, we find that that after possibly shrinking $p$ it a.s.\ holds with conditional probability at least $p$ given $\mcl F_1$ that the interior flow lines $\wh \eta_\ep^L$ and $\wh\eta_\ep^R$ defined above merge into $\rng\eta_\ep^L([0,\rng S_\ep^L])$ and $\rng\eta_\ep^R([0,\rng S_\ep^R])$, respectively, before leaving $U_*$ and the region enclosed by these 4 flow lines run up to the merging time contains $B_{c/10^{10}}(0)$. By definition of space-filling SLE, on this event this region is contained in $f_\ep(\eta'([T_1,T_\ep])$. 
Hence a.s.\ $\BB P(F_\ep   \,|\, \mcl F_1)  \geq p$, as required.
\end{proof}

\subsection{Continuity estimates for space-filling SLE}
\label{sec-sle-big}

The goal of this subsection is to prove that it is unlikely that a space-filling $\SLE_{\kappa'}$ travels a long distance without filling in a big ball. More precisely,

 \begin{prop} \label{prop-sle-bubble}
Let $\eta'$ be a space-filling $\SLE_{\kappa'}$ from $\infty$ to $\infty$ in $\BB C$. For $r \in (0,1)$, $R > 0$, and $\ep > 0$, let $\mcl E_\ep = \mcl E_\ep(R,r)$ be the event that the following is true. For each $\delta\in (0,\ep]$ and each $a < b\in\BB R$ such that $\eta'([a,b]) \subset B_R(0)$ and $\op{diam} \eta'([a,b]) \geq \delta^{1-r}$, the set $\eta'([a,b])$ contains a ball of radius at least $\delta $. Then
\eqbn
\lim_{\ep\rta 0} \BB P\!\left(\mcl E_\ep \right) = 1 .
\eqen
\end{prop}  

Proposition~\ref{prop-sle-bubble} yields as a corollary the optimal H\"older exponent for $\eta'$ when it is parameterized by Lebesgue measure.

\begin{cor} \label{prop-sle-holder}
Let $\eta'$ be a space-filling $\SLE_{\kappa'}$ from $\infty$ to $\infty$ in $\BB C$, parameterized by Lebesgue measure. Almost surely, $\eta'$ is locally H\"older continuous with any exponent $<1/2$, and is not locally H\"older continuous with any exponent $>1/2$. 
\end{cor}
\begin{proof}
Since $\eta'$ is parameterized by Lebesgue measure, we always have $\op{diam} \eta'([a,b]) \geq 2\pi^{-1/2}(b-a)^{1/2}$, so $\eta'$ cannot be H\"older continuous for any exponent $>1/2$. 

Conversely, it suffices to prove the H\"older continuity of the restriction of $\eta'$ to $(\eta')^{-1}(B_R(0))$ for some fixed $R>0$. Also fix $r>0$. By Proposition~\ref{prop-sle-bubble}, it is a.s.\ the case that for sufficiently small $\delta>0$ and each $a < b\in\BB R$ such that $\eta'([a,b]) \subset B_{2R}(0)$ and $\op{diam} \eta'([a,b]) \geq \delta^{1-r}$, the set $\eta'([a,b])$ contains a ball of radius at least $\delta $, whence $b-a \geq \pi\delta^2$. Hence for sufficiently small $\delta$, it holds that whenever $a,b\in (\eta')^{-1}(B_R(0))$ with $a<b$ and $b-a \leq \pi^{1/2} \delta$, we have $\op{diam} \eta'([a,b]) \leq \delta^{(1-r)/2}$. This proves the desired H\"older continuity. 
\end{proof}

The key input in the proof of Proposition~\ref{prop-sle-bubble} is the following estimate, which is an easy consequence of the results of Section~\ref{sec-sle-ball}.

 \begin{lem} \label{prop-hitting-bubble}
Let $\eta'$ be a whole-plane space-filling $\SLE_{\kappa'}$ from $\infty$ to $\infty$ with any choice of parameterization.
For $z\in\BB C$, let $\tau_z$ be the first time $\eta'$ hits $z$ and for $\rho \geq 0$ let $\tau_z(\rho)$ be the first time after $\tau_z$ at which $\eta$ exits $B_\rho(z)$. 
For $\ep \in (0,1)$, let $E_z^\ep(\rho)$ be the event that $\eta'([\tau_z , \tau_z(\rho)])$ contains a Euclidean ball of radius at least $\ep \rho$.  
There are constants $a_0,a_1>0$ depending only on $\kappa'$ such that for each $\rho > 0$ and each $\ep \in (0,1)$, 
\eqbn
\BB P\!\left( E_z^\ep(\rho)^c \right) \leq  a_0 e^{-a_1/\ep} .
\eqen   
\end{lem}
\begin{proof}
Fix $C >1$ to be chosen later, depending only on $\kappa'$. For $\ep \in (0,1)$, let $N_\ep := \lfloor (C \ep \rho)^{-1} \rfloor$ and for $k\in \{1,\dots,N_\ep\}$, let $\rho^\ep(k) := k C \ep \rho$. For $k\in \{2,\dots, N_\ep\}$, let $F_z^\ep(k)$ be the event that $\eta'([\tau_z(\rho^\ep(k-1))  ,\tau_z(\rho^\ep(k )) ])$ contains a ball of radius at least $\ep \rho$. Then $\bigcup_{k=2}^{N_\ep} F_z^\ep(k) \subset E_z^\ep(\rho)$. 

By Lemma~\ref{prop-sle-ball-pos} and scale and translation invariance of the law of whole-plane SLE$_{\kappa'}$, if we choose $C>1$ sufficiently large and $p \in (0,1)$ sufficiently small, depending only on $\kappa'$, then for $k\in \{1,\dots,N_\ep-1\}$,
\eqb \label{eqn-scale-bubble-pos}
\BB P\!\left( F_z^\ep(k+1) \,|\, \eta'|_{[0, \tau_z^\ep(\rho^\ep(k) )]}  \right) \geq p . 
\eqe
Multiplying this estimate over all $k\in \{1,\dots, N_\ep-1\}$ gives
\eqbn
\BB P\!\left( E_z^\ep(\rho)^c \right) \leq (1-p)^{N_\ep-1} \leq a_0 e^{-a_1/\ep }
\eqen
for an appropriate choice of $a_0 , a_1 >0$ as in the statement of the lemma. 
\end{proof}

We now want to extract Proposition~\ref{prop-sle-bubble} from Lemma~\ref{prop-hitting-bubble}. 
For $\ep > 0$ and $\rho > 0$, let 
\eqb \label{eqn-lattice-def}
\mcl S_\ep(\rho) := B_\rho(0) \cap (\ep \BB Z^2) .
\eqe 
Lemma~\ref{prop-hitting-bubble} tells us that with high probability, each segment of the space-filling $\SLE_{\kappa'}$ curve $\eta'$ which has diameter at least $\ep^{1-r}$ and which starts from the first time $\eta'$ hits a point of $\mcl S_\ep(\rho)$ contains a ball of radius at least $\ep$. However, it is still possible that there exists a segment of $\eta'$ contained in $B_R(0)$ for $R<\rho$ which has diameter $>\ep^{1-r}$, fails to contain a ball of radius $\ep$, and fails to contain any point of $\mcl S_\ep(\rho)$. In the remainder of this subsection, we will rule out this possibility. 

To this end, we will view the space-filling $\SLE_{\kappa'}$ curve $\eta'$ as being coupled with a whole-plane GFF $h$, defined modulo a global additive multiple of $2\pi\chi$, as in \cite{ig4}. For $z\in\BB C$, let $\eta_z^\pm$ be the flow lines of $h$ started from $z$ with angles $\mp\pi/2$. By (the whole-plane analog of) the construction of space-filling $\SLE_{\kappa'}$ in \cite{ig4}, the flow lines $\eta_z^\pm$ form the left and right boundaries of $\eta'$ stopped at the first time it hits $z$.

 \begin{lem} \label{prop-big-pocket}
Suppose we are in the setting described just above. Fix $R>0$. For $\rho >R$, let $\wt F (\rho) = \wt F (\rho,R)$ be the event that the following is true. For each $\ep \in (0,1)$, there exists $z_0  , w_0 \in \mcl S_\ep(\rho) \setminus B_R(0)$ such that the flow lines $\eta_{w_0}^-$ and $\eta_{z_0}^-$ and the flow lines $\eta_{z_0}^+$ and $\eta_{w_0}^+$ merge and form a pocket containing $B_R(0)$ before leaving $B_\rho(0)$. Then for each fixed $R> 0$, we have $\BB P\!\left(\wt F (\rho)\right) \rta 1$ as $\rho\rta\infty$. 
\end{lem} 
\begin{proof}
Let $t_R$ (resp.\ $t_R'$) be the first time $\eta'$ hits (resp.\ finishes filling in) $B_R(0)$. 
For $\rho_1 > \rho_0  > R$, let $\tau_{\rho_0}$ be the first time $\eta'$ hits $B_{\rho_0}(0)$ and let $\sigma_{\rho_0}^{\rho_1}$ be the first time after $\tau_{\rho_0}$ at which $\eta'$ leaves $B_{\rho_1}(0)$. Since $\eta'$ is a.s.\ continuous and a.s.\ hits every point of $\BB C$, it follows that there a.s.\ exists a random $\rho_1 >\rho_0 > R$ such that the following is true.
\begin{enumerate}
\item $\tau_{\rho_0} < t_R < t_R' < \sigma_{\rho_0}^{\rho_1}$. 
\item $\eta'([\tau_{\rho_0} , t_R])$ and $\eta'([t_R' , \sigma_{\rho_0}^{\rho_1}])$ each contain a ball of radius 1.
\end{enumerate}
For this choice of $\rho_0$, $\rho_1$ and any $\ep \in (0,1)$, there exists $z_0 \in \mcl S_\ep(\rho_1) \cap \eta'([\tau_{\rho_0} , t_R])$ and $w_0 \in \mcl S_\ep(\rho_1) \cap\eta'([t_R' , \sigma_{\rho_0}^{\rho_1}])$.  
By the construction of space-filling SLE, the pocket formed by the flow lines $\eta_{z_0}^\pm$ and $\eta_{w_0}^\pm$ stopped at the first time they merge is precisely the set of points traced by $\eta'$ between the first time it hits $z_0$ and the first time it hits $w_0$. Since $\eta'([\tau_{\tau_0} , \sigma_{\rho_0}^{\rho_1}]) \subset B_{\rho_1}(0)$ and $\eta'([t_R , t_R']) \supset B_R(0)$, it follows that this pocket contains $B_R(0)$ and is contained in $B_{\rho_1}(0)$. Since $\rho_1$ is a.s.\ finite, it follows that $\BB P\!\left(\rho_1 < \rho\right) \rta 1$ as $\rho\rta \infty$. The statement of the lemma follows. 
\end{proof}

\begin{lem} \label{prop-lattice-pockets}
Suppose we are in the setting described just above Lemma~\ref{prop-big-pocket}. Fix $R>0$. For $\ep \in (0,1)$ and $\rho > R$, let $\mcl P_\ep(\rho)$ be the set of complementary connected components of 
\eqbn
\bigcup_{z\in \mcl S_\ep(\rho)} (\eta_z^+ \cup \eta_z^-) 
\eqen
which intersect $B_R(0)$.
For $r \in (0,1)$, let $\wt{\mcl E}_\ep(\rho) = \wt{\mcl E}_\ep(\rho ;R,r)$ be the event that $\sup_{P\in\mcl P_\ep(\rho)} \op{diam} P \leq   \ep^{1-r}$. Also let $\wt F (\rho)$ be defined as in Lemma~\ref{prop-big-pocket}. Then for each fixed $\rho > R$, we have
\eqbn 
 \BB P\!\left(\wt{\mcl E}_\ep(\rho)^c \cap \wt F (\rho) \right)  \leq \rho^2 o_\ep^\infty(\ep) 
\eqen
at a rate depending only on $R$. 
\end{lem}
\begin{proof}
The proof given here is more or less implicit in the proof of continuity of space-filling SLE in \cite[Section~4.3]{ig4}, but we give a full proof for completeness. See Figure~\ref{fig-lattice-pockets} for an illustration.

For $\rho > R$ and $z\in B_\rho(0)$, let $\wt E_\ep^z(\rho)  $ be the event that the following is true. There exists $w \in \mcl S_\ep(\rho)$ such that $w\not=z$ and the curve $\eta_z^-$ hits (and subsequently merges with) $\eta_w^-$ on its left side before leaving $\partial B_{\ep^{1-r}/8}(z)$; and the same is true with $\eta_z^+$ in place of $\eta_z^-$ and/or ``right" in place of ``left". Let
\eqbn
\wt{\mcl E}_\ep'(\rho) := \bigcap_{z\in \mcl S_\ep(\rho)} \wt E_\ep^z(\rho) . 
\eqen
By \cite[Proposition~4.14]{ig4}
\footnote{The statement of \cite[Proposition~4.14]{ig4} does not specify the side at which the merging occurs, but the proof shows that we can require the merging to occur on a particular side of the curve.}  
we have $\BB P\!\left(\wt E_\ep^z(\rho)^c \right) = o_\ep^\infty(\ep)$, so by the union bound, $\BB P\!\left(\wt{\mcl E}_\ep'(\rho)\right) = 1-\rho^2 o_\ep^\infty(\ep)$. Hence to complete the proof it suffices to show that $\wt{\mcl E}_\ep'(\rho) \cap \wt F(\rho) \subset \wt{\mcl E}_\ep(\rho)$.

We first argue that on $\wt F(\rho)$, the boundary of each $P\in\mcl P_\ep(\rho)$ is entirely traced by curves $\eta_z^\pm$ for $z\in \mcl S_\ep(\rho)$. To see this, let $z_0, w_0 \in \mcl S_\ep(\rho)$ be as in the definition of $\wt F (\rho)$ and let $P_0$ be the pocket formed by $\eta_{z_0}^\pm$ and $\eta_{w_0}^\pm$ surrounding $B_R(0)$. Then for $z\notin B_\rho(0)$, the flow lines $\eta_z^\pm$ cannot cross $\partial P_0$ without merging into $\eta_{z_0}^\pm$, hence cannot enter $B_R(0)$ without merging into flow lines $\eta_z^\pm$ for $z\in \mcl S_\ep(\rho)$. 

Now suppose $\wt{\mcl E}'_\ep(\rho) \cap \wt F(\rho)$ occurs and $P\in \mcl P_\ep(\rho)$. We must show $\op{diam} P \leq \ep^{1-r}$. Since $\wt F (\rho)$ occurs, $\partial P$ consists of either two arcs traced by a pair of flow lines $\eta_z^-$ and $\eta_z^+$ for some $z\in \mcl S_\ep(\rho)$; or four arcs traced by $\eta_z^\pm$ and $\eta_w^\pm$ for some $z,w \in \mcl S_\ep(\rho)$. 

Suppose first that we are in the latter case, i.e., that there exists $z,w\in \mcl S_\ep(\rho)$ with the property that $\partial P$ contains non-trivial arcs traced by the left side of $\eta_z^-$, the right side of $\eta_z^+$, the right side of $\eta_w^-$, and the right side of $\eta_w^+$. Let $I_z^-$ be the arc of $\partial P$ traced by the left side of $\eta_z^-$. The curve $\eta_z^-$ cannot hit $\eta_v^-$ for any $v\in \mcl S_\ep(\rho)$ on its left side before $\eta_z^-$ finishes tracing $I_z^-$ (otherwise part of this arc $I_z^-$ would lie on the boundary of a pocket other than $P$). The same is true if we replace $\eta_z^-$ with one of the other three arcs of $\partial P$ in our description of $\partial P$. Since $\wt E_\ep^z(\rho) \cap \wt E_\ep^w(\rho)$ occurs, each of these four arcs has diameter at most $\frac14 \ep^{1-r}$. Therefore, $\op{diam} P \leq   \ep^{1-r}$. By a similar argument, but with only two distinguished boundary arcs instead of four, we obtain $\op{diam} P \leq \ep^{1-r}$ in the case when $\partial P$ is traced by a pair of flow lines $\eta_z^-$ and $\eta_z^+$ for $z\in  \mcl S_\ep(\rho)$. Thus $\wt{\mcl E}_\ep'(\rho)  \cap \wt{F}(\rho)  \subset \wt{\mcl E}_\ep(\rho)$, as required.
\end{proof}

\begin{figure}[ht!]
 \begin{center}
\includegraphics[scale=0.92]{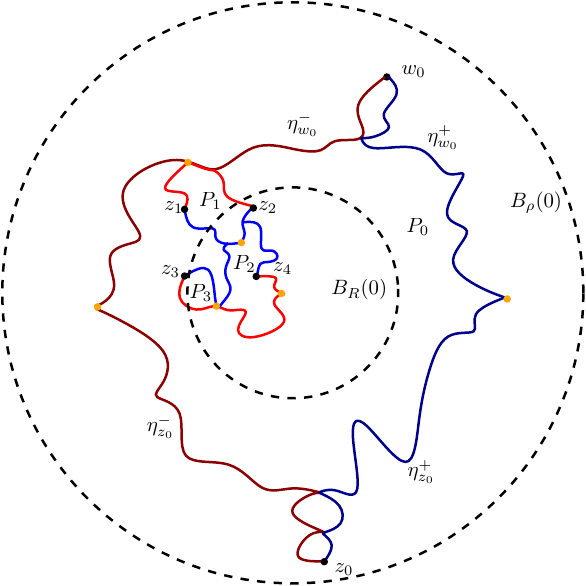} 
\caption{An illustration of the proof of Lemma~\ref{prop-lattice-pockets}. On the event $\wt F(\rho)$, there exists points $z_0$ and $w_0$ in the lattice $\mcl S_\ep(\rho)$ such that the flow lines $\eta_{z_0}^\pm$ and $\eta_{w_0}^\pm$ intersect at points shown in orange to form a pocket $P_0$ which separates $B_R(0)$ from $\partial B_\rho(0)$. Also shown are four points $z_1 , z_2, z_3,z_4 \in \mcl S_\ep(\rho)$ and three pockets $P_1,P_2,P_3 \in \mcl P_\ep(\rho)$ formed by the flow lines started at these points. Flow lines $\eta_{z_i}^-$ for $i\in\{1,\dots ,4\}$ are shown in red, and flow lines $\eta_{z_i}^+$ are shown in blue. The points where two of these flow lines merge are shown in orange. The pockets $P_1$ and $P_2$ are of the type with boundary arcs traced by four flow lines started from two points and the pocket $P_3$ is of the type with boundary arcs traced by two flow lines started from a single point. } \label{fig-lattice-pockets}
\end{center}
\end{figure}

\begin{proof}[Proof of Proposition~\ref{prop-sle-bubble}]
Fix $r' \in (0,r)$. For $\ep \in (0,1)$, let $k_\ep$ be the smallest $k\in\BB N$ such that $2^{-k} \geq \ep$. Let $ \mcl E_\ep^1$ be the event that the following is true. For each $k \geq k_\ep$ and each $z \in \mcl S_{2^{-k}}(2R)$ (defined as in~\eqref{eqn-lattice-def}), the event $E_z^{2^{-kr'} }(2^{-k( 1-r')} )$ of Lemma~\ref{prop-hitting-bubble} occurs (i.e., with $2^{-k( 1-r')}$ in place of $\rho$ and $2^{-kr}$ in place of $\ep$). By Lemma~\ref{prop-hitting-bubble} and the union bound, we have
\eqbn
\BB P\!\left( \mcl E_\ep^1 \right) = 1-o_\ep^\infty(\ep). 
\eqen

Fix $\rho > R$ to be chosen later and define the event $\wt F(\rho)$ as in Lemma~\ref{prop-big-pocket} and the events $\wt{\mcl E}_\ep(\rho)$ as in Lemma~\ref{prop-lattice-pockets}, the latter with $r'$ in place of $r$. Also let $\mcl P_\ep(\rho)$ be the set of pockets as defined in Lemma~\ref{prop-lattice-pockets}. Let
\eqbn
 \mcl E_\ep^2(\rho) := \bigcup_{k=k_\ep}^\infty \wt{\mcl E}_{2^{-k}}(\rho) .
\eqen
By Lemma~\ref{prop-lattice-pockets}, the union bound, and the argument given just above, we have
\eqbn
\BB P\!\left(  \mcl E_\ep^1 \cap  \mcl E_\ep^2(\rho) \cap \wt F(\rho) \right) = \BB P\!\left(\wt F(\rho)\right) - \rho^2 o_\ep^\infty(\ep) .
\eqen
Since $\BB P\!\left(\wt F(\rho)\right) \rta 1$ as $\rho \rta \infty$ (by Lemma~\ref{prop-big-pocket}), to complete the proof of the proposition it suffices to show that $\mcl E_\ep^1 \cap  \mcl E_\ep^2(\rho) \subset \mcl E_\ep$ for each choice of $\rho > R$ and each sufficiently small $\ep \in (0,1)$.  

To this end, suppose $\mcl E_\ep^1 \cap  \mcl E_\ep^2(\rho)$ occurs and we are given $\delta\in (0,\ep]$ and $a < b\in\BB R$ such that $\eta'([a,b]) \subset B_R(0)$ and $\op{diam} \eta'([a,b]) \geq \delta^{1-r}$. Let $t_*$ be the smallest $t\in [a,b]$ such that $\op{diam} \eta'([a,t_*]) \geq \frac12 \delta^{1-r}$. Let $k_\delta$ be the largest $k\in\BB N$ such that $2^{-k} \geq \delta$. Since $\wt{\mcl E}_{2^{-k_\delta}}(\rho)$ occurs with $r'$ in place of $r$, it follows that for sufficiently small $\ep$, the set $\eta'([a,t_*])$ is not contained in any single pocket in $\mcl P_{2^{-k_\delta}}(\rho)$. Whenever $\eta'$ exits a pocket in $\mcl P_{2^{-k_\delta}}(\rho)$, it hits a point $z\in \mcl S_{2^{-k_\delta}}(2 R)$ for the first time. Hence there exists $t_{**} \in [a,t_*]$ and $z\in \mcl S_{2^{-k_\delta}}(2 R)$ such that $\eta'$ hits $z$ for the first time at time $t_{**}$. The set $\eta'([t_{**}, b])$ has diameter at least $\frac12 \delta^{1-r}$. Since $E_{2^{-k_\delta}}(z)$ occurs, it follows that this set contains a ball of radius $\delta$. Thus $\eta'([a,b])$ contains a ball of radius $\delta$, which concludes the proof.
\end{proof}

\begin{remark} \label{remark-sle-bubble-poly}
If we use some additional facts about space-filling SLE which are proven in~\cite{hs-euclidean}, our proof yields a more quantitative version of Proposition~\ref{prop-sle-bubble}, namely that
\eqb \label{eqn-sle-bubble-poly}
\BB P\!\left( \mcl E_\ep^c \right) = o_\ep^\infty(\ep) ,\quad \text{as $\ep\rta 0$},
\eqe  
i.e., this probability decays faster than any positive power of $\ep$. 
Indeed, by~\cite[Proposition~6.2]{hs-euclidean} and scale invariance there exists some $\xi = \xi(\kappa) > 0$ such that if $\eta'$ is parameterized by Lebesgue measure with $\eta'(0) = 0$ then for $R > 0$,
\eqb \label{eqn-sle-ball-tail}
\BB P\!\left( B_R(0)\not\subset \eta'([-C , C]) \right) \preceq C^{-\xi}  \quad\forall C > 0,
\eqe 
with the implicit constant independent of $C$. We claim that this estimate implies that the conclusion of Lemma~\ref{prop-big-pocket} (for fixed choice of $R$) can be improved to
\eqb \label{eqn-big-pocket-sharper}
\BB P\!\left(\wt F(\rho)\right) \leq \rho^{- 2\xi  + o_\rho(1)} .
\eqe
It is immediate from Lemma~\ref{prop-hitting-bubble} that except on an event of probability $o_C^\infty(C)$ as $C\rta \infty$, the set $\eta'([-2C ,2C])$ is contained in $B_{C^{1/2+o_C(1)}}(0)$. By~\eqref{eqn-sle-ball-tail}, the set $\eta'([-C, C])$ contains $B_R(0)$ except on an event of probability $O_C(C^{-\xi})$. By Lemma~\ref{prop-hitting-bubble} and the translation invariance of $\eta'$ parameterized by Lebesgue measure~\cite[Lemma~2.3]{hs-euclidean}, each of $\eta'([C , 2C])$ and $\eta'([-2C , -C])$ contains a ball of radius 1 with high probability. Taking $C = \rho^{2 + o_\rho(1)}$ and the points $z_0 , w_0 \in S_\ep(\rho)$ in Lemma~\ref{prop-big-pocket} to be contained in these balls of radius 1 yields~\eqref{eqn-big-pocket-sharper}.  
We now obtain~\eqref{eqn-sle-bubble-poly} by taking $\rho = \ep^{-N}$ for large $N \in\BB N$ and using~\eqref{eqn-big-pocket-sharper} instead of Lemma~\ref{prop-big-pocket} in the above proof of Proposition~\ref{prop-sle-bubble}.
\end{remark}

\subsection{Estimate for the quantum measure}
\label{sec-mu_h-big}

The following lemma will reduce the problem of estimating the $\gamma$-quantum measure of a quantum cone (i.e., the quantum surface appearing in Theorem~\ref{thm-dim-relation}) to the problem of estimating the quantum measure induced by a whole-plane Gaussian free field. 
  
\begin{lem} \label{prop-annulus-abs-cont}
Let $\alpha < Q$ and let $h$ be the circle-average embedding of an $\alpha$-quantum cone (Definition~\ref{def-circle-embedding}). Also let $\wt h$ be a whole-plane GFF normalized so that its circle average over $\bdy\BB D$ is 0. There exists $c>0$ depending only on $\alpha$ and $\gamma$ and a coupling of $h$ and $\wt h$ such that $h - \wt h$ is a.s.\ a continuous function on $\BB C\setminus \{0\}$ and for each $R>1$ and each $M>1$, 
\eqb \label{eqn-annulus-abs-cont}
\BB P\left(\sup_{z\in B_R(0) \setminus B_{1/R}(0) } |(h - \wt h)(z)|  > M \right) \preceq e^{-c M}
\eqe 
with the implicit constant depending only on $R$, $\alpha$, and $\gamma$. 
\end{lem}

Note that if $h$ and $\wt h$ are coupled as in the statement of the lemma, the $\gamma$-quantum area measures $\mu_h$ and $\mu_{\wt h}$ are a.s.\ mutually absolutely continuous on $B_R(0) \setminus B_{1/R}(0)$ and the estimate~\eqref{eqn-annulus-abs-cont} implies that the Radon-Nikodym derivative satisfies
\eqb \label{eqn-annulus-measure}
\BB P\left(\sup_{z\in B_R(0) \setminus B_{1/R}(0) } \left| \frac 1\gamma \log \frac{d\mu_h}{d\mu_{\wt h}}(z) \right|  > M \right) \preceq e^{-c M} .
\eqe

\begin{proof}[Proof of Lemma~\ref{prop-annulus-abs-cont}]
For $r>0$ let $h_r(0)$ and $\wt h_r(0)$ be the circle averages of $h$ and $\wt h$, respectively, over $\bdy B_r(0)$. By Definition~\ref{def-quantum-cone}, the laws of $h - h_{|\cdot|}(0)$ and $\wt h  -\wt h_{|\cdot|}(0)$ agree. Furthermore, these distributions are independent from $h_{|\cdot|}(0)$ and $\wt h_{|\cdot|}(0)$, respectively. The law of $\wt h_{e^{-s}}(0)$ for $s\in\BB R$ is that of a standard linear two-sided Brownian motion~\cite[Proposition~3.2]{shef-kpz}. The law of $h_{e^{-s}}(0) - \alpha s$ for $s\geq 0$ is that of a standard linear Brownian motion and its law for $s < 0$ is that of a standard linear Brownian motion conditioned on the event that $h_{e^{-s}}(0) - \alpha s \geq Q s$. It follows that we can couple $h$ and $\wt h$ in such a way that $h - h_{|\cdot|}(0) = \wt h  -\wt h_{|\cdot|}(0)$ and a.s.\ $h_r(0)  - \wt h_r(0) = \alpha \log r^{-1}$ for $r < 1$. Henceforth assume $h$ and $\wt h$ are coupled in this way. We need to bound $\sup_{r\in [1 , R]} |h_r(0) -\wt h_r(0)|$. 
 
 By standard estimates for Brownian motion, we have $\BB P\left(\sup_{r\in [1,R]} |\wt h_r(0)| > M \right) \preceq e^{- c M }$ for every choice of $c>0$. By~\cite[Remark 4.4]{wedges}, we can express the law of $h_{e^{-s}}(0) - \alpha s$ for $s\leq 0$ as follows. Let $B$ be a standard linear Brownian motion and let $\tau$ be the largest $t>0$ for which $B_t + (Q-\alpha) t = 0$ (which is a.s.\ finite). If we set $\wh B_t := B_{t + \tau} + (Q-\alpha) (t+\tau)$, then $\{\wh B_{-s}\}_{s \leq 0} \eqD \{h_{e^{-s}}(0) - \alpha s\}_{s\leq 0}$. By the reflection principle, the Gaussian tail bound, and the union bound, for $T>0$ we have
\eqbn 
 \BB P\left( \tau > T \right)
 \leq \BB P\left( \text{$\exists t > T$ such that $|B_t| > (Q-\alpha) t$} \right) 
\preceq e^{-(Q-\alpha)^2 T/2} .
\eqen
Consequently, we have
\alb
&\BB P\left( \sup_{t\in [0, \log R  ]} |B_{t + \tau} + (Q-\alpha) (t+\tau)| > M \right)  \\
&\qquad \leq \BB P\left( \sup_{t\in [0,  \frac12 (Q-\alpha)^{-1} M  ]} |B_t| > M/2 \right) + \BB P\left( \tau >  \frac12 (Q-\alpha)^{-1} M - \log R \right)  
\preceq e^{-c M}
\ale
for $c>0$ as in the statement of the lemma. By combining this with our above description of the law of $\{h_{e^{-s}}(0) - \alpha s\}_{s\leq 0}$ and the triangle inequality, we obtain~\eqref{eqn-annulus-abs-cont}. 
\end{proof}

In the remainder of this section we will prove several estimates for the quantum measure induced by a whole-plane GFF. The estimates of this section will be used in conjunction with Proposition~\ref{prop-sle-bubble} and Lemma~\ref{prop-annulus-abs-cont} to estimate the size of the pre-image of the set $X$ under $\eta'$ in the proof of Theorem~\ref{thm-dim-relation}.

\begin{lem} \label{prop-harmonic-var}
Let $h$ be a whole plane GFF. Let $\frk h$ be the harmonic part of $h|_{\BB D}$. For each $r \in (0,1)$ and each $\rho \in (0,1)$, we have
\eqb \label{eqn-harmonic-var}
\BB P\!\left(\sup_{z\in B_\rho(0)} |\frk h(z) - \frk h(0)| \geq r\log \ep^{-1} \right) = o_\ep^\infty(\ep) .
\eqe 
\end{lem}
\begin{proof}
Let $\rho' := \frac{1+\rho}{2}$. By the mean value property of harmonic functions,  
\eqbn
\sup_{z\in B_\rho(0)} |\frk h(z) - \frk h(0)| \preceq \int_{B_{\rho'}(0)} |\frk h(w) - \frk h(0)| \, dw  ,
\eqen 
with the implicit constants depending only on $\rho$. 
By Jensen's inequality, for each $p > 0$ we have
\eqbn
\sup_{z\in B_\rho(0)} e^{p|\frk h(z) - \frk h(0)|} \leq  c_0 \int_{B_{\rho'}(0)} e^{ p c_1 |\frk h(w) - \frk h(0)|} \, dw   
\eqen
for constants $c_0$ and $c_1$ depending only on $\rho$. 
By \cite[Proposition~6.4]{ig1}, for each $w\in B_{\rho'}(0)$, $  \frk h(w) - \frk h(0)$ is a centered Gaussian with variance bounded above by a constant depending only on $\rho$. Hence for each $p>0$, 
\eqbn
\BB E\!\left( c_0 \int_{B_{\rho'}(0)} e^{p c_1 |\frk h(w) - \frk h(0)|} \, dw   \right) < \infty.  
\eqen
By applying the Chebyshev inequality and letting $p \rta \infty$, we infer~\eqref{eqn-harmonic-var}. 
\end{proof}

\begin{lem} \label{prop-mu_h-inf}
Fix $r > 0$ and $R>1$. Let $h$ be a whole-plane GFF, let $\mu_h$ be its $\gamma$-quantum are measure, and let $(h_\ep)$ be its circle average process. For each $z\in \BB C$, each $\ep \in (0,1)$, and each $\delta \in (0,1)$,  
\eqb \label{eqn-mu_h-inf}
\BB P\!\left( \mu_h(B_\ep(z)) < \delta \ep^{2+\frac{\gamma^2}{2} }  e^{\gamma h_\ep(z)} \right)   =o_\delta^\infty(\delta) , 
\eqe  
at a rate depending only on $r$ and $R$. 
\end{lem}
\begin{proof}
The estimate~\eqref{eqn-mu_h-inf} is independent of the choice of additive constant for $h$, can assume without loss of generality that $h$ is normalized so that the circle average $h_1(0)  $ is equal to 0. Fix $z\in \BB C$. For $\ep \in (0,1)$, let $\phi_z^\ep(w) := \ep w + z$ be an affine map which takes $\BB D$ to $B_\ep(z)$ and let $\wt h_z^\ep := h \circ \phi_z^\ep$. By \cite[Proposition~2.1]{shef-kpz}, we have 
\eqb \label{eqn-mu_h-ball-shift}
\mu_h(B_\ep(z)) =  \mu_{\wt h_z^\ep+ Q\log \ep }(B_1(0)) = \ep^{2+\frac{\gamma^2}{2}}  \mu_{\wt h_z^\ep}(B_1(0)) .
\eqe 
By translation and scale invariance $\wt h_z^\ep$ has the law of a whole plane GFF, modulo additive constant. The circle average of $\wt h_z^\ep$ at $\partial\BB D$ is equal to $h_\ep(z)$. To see this, we observe that the former circle average is equal to the conditional mean of $\wt h_z^\ep$ evaluated at $0$, given its values outside of $\BB D$. By conformal invariance of harmonic functions and of the zero-boundary GFF, this equals the conditional mean of $h$ at $z$ given its values outside $B_\ep(z)$, which equals $h_\ep(z)$. 

Hence the field $\wh h_z^\ep := \wt h_z^\ep - h_\ep(z)$ agrees in law with $h$ (not just modulo additive constant), and by~\eqref{eqn-mu_h-ball-shift}, we have
\eqb \label{eqn-B_ep-B_1}
\mu_h(B_\ep(z)) = e^{\gamma h_\ep(z)} \ep^{2+\frac{\gamma^2}{2}}  \mu_{\wh h_z^\ep}(B_1(0)) .
\eqe 
It remains to argue that $\mu_{\wh h_z^\ep}(B_1(0))$ is unlikely to be small; equivalently, $\mu_h(B_1(0))$ is unlikely to be small. To see this, let $\frk h$ be the conditional mean of $h$ given its values outside $\BB D$ and let $\dot h := h|_{\BB D} - \frk h$. Then $\dot h$ is a zero-boundary GFF on $\BB D$ which is conditionally independent from $\frk h$ given $h|_{\BB C\setminus \BB D}$, and we have 
\eqb \label{eqn-mu_h(D)-split}
\mu_h(\BB D) \geq   \mu_{\dot h}(B_{1/2}(0)) \exp\left( - \gamma \sup_{w\in B_{1/2}(0)} |\frk h(w)| \right) .
\eqe 
By \cite[Lemma~4.5]{shef-kpz},
\eqb \label{eqn-mu_h(D)-split1}
\BB P\!\left( \mu_{\dot h}(B_{1/2}(0))< \delta^{1/2} \right) = o_\delta^\infty(\delta) .
\eqe 
Since $\frk h(0) = h_1(0) = 0$, Lemma~\ref{prop-harmonic-var} implies
\eqb \label{eqn-mu_h(D)-split2}
\BB P\!\left( \exp\left( - \gamma \sup_{w\in B_{1/2}(0)} |\frk h(w)| \right) < \delta^{1/2} \right) = o_\delta^\infty(\delta)  .
\eqe
By~\eqref{eqn-mu_h(D)-split},~\eqref{eqn-mu_h(D)-split1}, and~\eqref{eqn-mu_h(D)-split2} we obtain $\BB P\!\left(\mu_h(\BB D) < \delta \right) =  o_\delta^\infty(\delta) $. By~\eqref{eqn-B_ep-B_1}, it follows that~\eqref{eqn-mu_h-inf} holds.
\end{proof}

\subsection{Circle average continuity}
\label{sec-gff-cont}

In this subsection we will prove various results which says that the circle average of a whole-plane GFF around a small circle or the quantum mass of a small ball is unlikely to differ too much from the value we would expect given the circle average around a larger circle centered at a nearby point. We start with a basic continuity estimate for the circle average. 
 
\begin{lem} \label{prop-circle-cont}
Fix $\rho > R> 0$. Suppose that either $h$ is a zero-boundary GFF on $B_{\rho}(0)$ or $h$ is a whole-plane GFF. Let $(h_\ep)$ be the circle average process of $h$. There a.s.\ exists a modification of $h_\ep$ (still denoted by $h_\ep$) such that the following is true. For $r > 0$ and $C>1$, let $\wt{\mcl C}(C) = \wt{\mcl C}(C, R,r)$ be the event that
\eqb \label{eqn-circle-cont}
|h_\ep(z) - h_{\ep'}(z')| \leq C \frac{|(z,\ep) - (z',\ep')|^{(1-r)/2}}{\ep^{1/2 }}
\eqe 
for each $\ep , \ep' \in (0,1]$ with $\frac12 \leq \ep/\ep' \leq 2$ and each $z,z' \in B_R(0)$. Then $\BB P\!\left(\wt{\mcl C}(C) \right) \rta 1$ as $C\rta \infty$.  
\end{lem}
\begin{proof}
The statement for the case of a zero-boundary GFF on $B_{\rho}(0)$ follows from \cite[Proposition~2.1]{hmp-thick-pts} (c.f.\ the proof of \cite[Proposition~8.4]{qle}). In particular, for such a zero-boundary GFF we have $\BB P\!\left( \bigcup_{C>1} \wt{\mcl C}(C) \right) =1$. If $h$ is a whole-plane GFF, then the law of $h|_{B_R(0)}$ is mutually absolutely continuous with respect to the law of a zero-boundary GFF on $B_{\rho}(0)$ restricted to $B_R(0)$ (see, e.g.\ \cite[Proposition~3.2]{ig1}). Hence $\BB P\!\left( \bigcup_{C>1} \wt{\mcl C}(C) \right) = 1$ in this case, so $\BB P\!\left(\wt{\mcl C}(C) \right) \rta 1$ as $C\rta \infty$.
\end{proof}

We henceforth assume that we have replaced $(h_\ep)$ with a modification as in Lemma~\ref{prop-circle-cont}.

\begin{lem} \label{prop-circle-shift}
Fix $\rho > R  > 1$ and $r \in (0,1)$. Let $h$ be either a zero-boundary GFF on $B_{\rho}(0)$ or a whole-plane GFF and let $(h_\ep)$ be the circle average process for $h$. Define the events $\wt{\mcl C}(C) = \wt{\mcl C}(C, R,r)$ as in Lemma~\ref{prop-circle-cont}. For each $a> 0$, each $c> 0$, and each $z,w\in B_R(0)$ with $|z-w| \leq c\ep$, we have 
\eqbn
\BB P\!\left( |h_{\ep }(w) - h_{\ep^{1 -   r}}(z)| \geq  a  \log \ep^{-1} ,\, \wt{\mcl C}(C) \right) \preceq \ep^{\frac{a^2 }{2 r}}  
\eqen
with the implicit constant depending only on $c$, $C$, $R$, and $r$.  
\end{lem}
\begin{proof} 
If $z,w\in B_R(0)$ with $|z-w| \leq c \ep$ and $\wt{\mcl C}(C)$ occurs, then 
\eqbn
|h_{\ep^{1- r}}(z) - h_{\ep^{1- r}}(w)| \preceq 1 .
\eqen
For $t > 0$, let $B_t :=  h_{e^{-t} \ep^{1- r} }(w) - h_{\ep^{1- r}}(w)$. By the calculations in \cite[Section~3.1]{shef-kpz}, $B$ is a standard linear Brownian motion. Therefore,
\eqbn
\BB P\!\left(| h_{\ep }(w) - h_{\ep^{1- r}}(w)| \geq a  \log \ep^{-1} - C\right) = \BB P\!\left(|B_{ r\log\ep^{-1}}| \geq a  \log \ep^{-1} - C\right) \preceq \ep^{\frac{a^2 }{2 r}} .
\eqen
We conclude by means of the triangle inequality. 
\end{proof}

\begin{lem} \label{prop-gff-cont-spread}
Let $\rho > R > 0$. Let $h$ be either a zero-boundary GFF on $B_{\rho}(0)$ or a whole-plane GFF and let $(h_\ep)$ be the circle average process for $h$. For $r \in (0,1/2)$ and $\ep \in (0,1)$, let $\mcl C_\ep = \mcl C_\ep(R,r  )$ be the event that the following is true. For each $\delta \in (0,\ep]$ and each $z,w\in B_R(0)$ with $|z-w| \leq 2\delta$, we have
\eqbn
|h_{\delta }(w) - h_{\delta^{1 -  r}}(z)| \leq 3\sqrt{10 r} \log \delta^{-1} .
\eqen
For each $r\in (0,1/2)$, we have $\BB P\!\left(\mcl C_\ep\right) \rta 1$ as $\ep\rta 0$. 
\end{lem}
\begin{proof}
Fix $C>1$ and define the event $\wt{\mcl C}(C) = \wt{\mcl C}(C,R,r)$ as in Lemma~\ref{prop-circle-cont}. By Lemma~\ref{prop-circle-shift}, for each $z,w\in B_R(0)$ with $|z-w| \leq 6\delta$,
\eqb \label{eqn-good-on-C}
\BB P\!\left( |h_{\delta }(z) - h_{\delta^{1 -   r}}(w)| \geq  \sqrt{10 r} \log \delta^{-1}  ,\, \wt{\mcl C}(C) \right) \preceq \delta^{5}  
\eqe 
with the implicit constant depending only on $C$, $R$, and $r$. Choose a finite collection $\mcl S_\delta$ of at most $O_\delta(\delta^{-\frac{2}{1-r}}  )$ points in $B_R(0)$ such that for each $z\in B_R(0)$, there exists $z' \in \mcl S_\delta$ with $|z-z'| \leq \delta^{\frac{1}{1-r}} $. On $\wt{\mcl C}(C)$, for such a $z$ and $z'$ we have
\eqbn
|h_{\delta }(z) - h_{\delta}(z') | \leq C ,\qquad |h_{\delta^{1- r}}(z) - h_{\delta^{1- r}}(z') | \leq C  .
\eqen
By~\eqref{eqn-good-on-C} and the union bound, on $\wt{\mcl C}(C)$ it holds except on an event of probability $\preceq \delta$ (implicit constant depending only on $C$, $R$, and $r$) that
\eqbn
|h_{\delta}(z) - h_{\delta^{1-r} }(w)| \leq  \sqrt{10 r} \log \delta^{-1} 
\eqen
whenever $z,w \in \mcl S_\delta$ with $|z-w| \leq  6\delta$.  
By the triangle inequality, whenever this is the case and $\delta$ is sufficiently small (depending on $r$ and $C$), we have
\eqb \label{eqn-discrete-good}
 |h_{\delta }(w) - h_{\delta^{1 - r}}(z)| \leq 2\sqrt{10 r} \log \delta^{-1}   
\eqe 
whenever $z,w\in B_R(0)$ with $|z-w| \leq 2\delta$. For $\delta>0$, let $\mcl C_\delta'$ be the event that this last statement holds, so that
\eqb \label{eqn-C'-prob}
\BB P\!\left( (\mcl C_\delta' )^c ,\, \wt{\mcl C}(C) \right) \preceq \delta .
\eqe 

Fix a sequence $(\zeta_n)$ decreasing to $0$ such that 
\eqb \label{eqn-zeta_n-choice}
\lim_{n\rta \infty} \frac{(\zeta_n - \zeta_{n+1})^{(1-r)/2}}{\zeta_{n+1}^{1/2}} = \lim_{n\rta \infty} \frac{(\zeta_n^{1-r} - \zeta_{n+1}^{1-r})^{(1-r)/2}}{\zeta_{n+1}^{(1-r)/2}} =  0 \quad \op{and} \quad \sum_{n=1}^\infty \zeta_n < \infty ,
\eqe
e.g.\ $\zeta_n =n^{-q}$ for appropriate $q > 1$, depending on $r$. For $\ep \in (0,1)$, let $n_\ep$ be the greatest integer $n$ such that $\zeta_{n } \geq \ep^{1-r}$. By~\eqref{eqn-C'-prob} and the union bound,
\eqbn
\liminf_{\zeta\rta 0} \BB P\!\left( \wt{\mcl C}(C) \cap \bigcap_{n=n_\ep}^\infty  \mcl C_{\zeta_n}'   \right) = \BB P\!\left(\wt{\mcl C}(C)\right) - o_\ep(1)  .
\eqen
Since $\BB P\!\left(\wt{\mcl C}(C)\right) \rta 1$ as $C\rta \infty$ (by Lemma~\ref{prop-circle-cont}), it suffices to show that for sufficiently small $\ep >0$,
\eqb \label{eqn-bigcup-contain}
\wt{\mcl C}(C) \cap \bigcap_{n=n_\ep}^\infty \mcl C_{\zeta_n}'  \subset \mcl C_\ep   .
\eqe 
To see this, suppose given $\delta \in (0,\ep]$ and each $z,w\in B_R(0)$ with $|z-w| \leq 2\delta$. Let $n_\delta$ be chosen so that $\delta \in [\zeta_{n_\delta+ 1} , \zeta_{n_\delta }]$. 
By our choice of $(\zeta_n)$, on $\wt{\mcl C}(C)$ we have
\eqbn
|h_\delta(w) - h_{\zeta_{n_\delta}}(w)| \leq C o_\ep(1) ,\qquad |h_{\delta^{1-r}}(z) - h_{\zeta_{ n_\delta}^{1-r}}(z)| \leq C o_\ep(1) .
\eqen
By~\eqref{eqn-discrete-good} with $ \zeta_{n_\delta}$ in place of $\delta$, along with the triangle inequality, we obtain~\eqref{eqn-bigcup-contain}. 
\end{proof}

\begin{prop} \label{prop-thick-local}
Let $\rho > R > 1$. Suppose that either $h$ is zero-boundary GFF on $B_\rho(0)$ or a whole-plane GFF. Let $(h_\ep)$ be the circle average process for $h$ and let $\mu_h$ be its $\gamma$-quantum area measure. There is a function $\psi : (0,\infty) \rta (0,\infty)$ with $\lim_{r\rta 0} \psi(r) = 0$, depending only on $\gamma$, such that the following holds. For $r \in (0,1/2)$, and $\ep \in (0,1)$, let $\mcl G_\ep = \mcl G_\ep(R,r)$ be the event that the following is true. For each $\delta \in (0,\ep]$ and each $z,w\in B_R(0)$ with $|z-w| \leq \delta$, we have
\eqbn
\mu_h(B_\delta(w) ) \geq \delta^{2+\frac{\gamma^2}{2} + \psi(r) } e^{\gamma h_{\delta^{1- r}}(z)}  .
\eqen
Then $\BB P\!\left(\mcl G_\ep\right) \rta 1$ as $\ep\rta 0$. 
\end{prop}
\begin{proof} 
For $C>1$, define the event $\wt{\mcl C}(C) = \wt{\mcl C}(C, R,r)$ as in Lemma~\ref{prop-circle-cont}. Define the event $\mcl C_\ep = \mcl C_\ep(R,r)$ as in Lemma~\ref{prop-gff-cont-spread}. Also fix a sequence $\zeta_n \rta 0$ satisfying~\eqref{eqn-zeta_n-choice}. 

Fix $p > 1$. For $n\in\BB N$, choose a finite collection $\mcl S_n$ of at most $O_n(\zeta_n^{-2p})$ points in $B_R(0)$ such that each point of $B_R(0)$ lies within distance $\zeta_n^p$ of some point in $\mcl S_n$. For $\ep \in (0,1)$, let $n_\ep$ be the greatest integer $n$ such that $\zeta_n \geq \ep^{1-r}$ and let
\[
\mcl D_\ep := \bigcap_{n=n_\ep}^\infty \bigcap_{z\in \mcl S_n} \left\{\mu_h(B_{\zeta_n}(z)) \geq \zeta_n^{2+\frac{\gamma^2}{2} + r } e^{\gamma  h_{\zeta_n}(z)} \right\} .
\]
By Lemma~\ref{prop-mu_h-inf} and the union bound, we have $\BB P\!\left(\mcl D_\ep\right)  \rta 1$ as $\ep\rta 0$. 

Since
\eqbn
\BB P\!\left(\wt{\mcl C}(C) \cap \mcl C_\ep \cap \mcl D_\ep \right) \rta 1
\eqen
as $\ep \rta 0$ and then $C \rta \infty$, it suffices to show that if $p$ is sufficiently large and $\ep$ is sufficiently small (depending on $p$), then 
\eqbn
\wt{\mcl C}(C) \cap \mcl C_\ep \cap \mcl D_\ep \subset \mcl G_\ep
\eqen
for an appropriate choice of $\psi(r)$ depending only on $r$ and $\gamma$. 
To this end, suppose that $\wt{\mcl C}(C) \cap \mcl C_\ep \cap \mcl D_\ep$ occurs and we are given $\delta   \in (0,\ep]$ and $z,w\in\BB A_R$ with $|z-w| \leq \delta$. By definition of $\wt{\mcl C}(C)$ and by~\eqref{eqn-zeta_n-choice}, if $p$ is chosen sufficiently large, depending only on $r$ and the choice of sequence $(\zeta_n)$, we can find $n_\delta   \in \BB N \cap [n_\ep , \infty)$ and $w' \in \mcl S_{n_\delta}$ such that
\eqbn
\delta/2 < \zeta_{n_\delta} < \delta , \quad |w' - w| \leq \zeta_{n_\delta}^p < \delta - \zeta_{n_\delta},   \quad \op{and} \quad |h_{\delta^{1-r}}(z) - h_{\zeta_{n_\delta }^{1-r }}(z )| \leq C o_\ep(1) .
\eqen 
By the definitions of $\mcl C_\ep $ and $ \mcl D_\ep$, we therefore have  
\alb
\mu_h(B_\delta(w) ) 
&\geq \mu_h(B_{\zeta_{n_\delta}}(w'))  
\succeq \delta^{2+\frac{\gamma^2}{2} + r  } e^{\gamma h_{\zeta_{n_\delta}}(w')  } 
 \geq \delta^{2+\frac{\gamma^2}{2} + r + 5\gamma\sqrt{10 r} } e^{\gamma h_{\zeta_{n_\delta}^{1-r }}(z)  } \\
&\succeq \delta^{2+\frac{\gamma^2}{2} + r + 5\gamma\sqrt{10 r} } e^{\gamma h_{\delta^{1-r }}(z)  } , 
\ale
with the implicit constants depending on $C$ but tending to $C$-independent constants as $\ep \rta 0$. This proves the statement of the lemma in the case of a whole-plane GFF with $\psi(r)$ slightly larger than $ r + 5\gamma\sqrt{10 r}$. 
\end{proof}

\section{Intersection of the thick points of a GFF with a Borel set}
\label{sec-thick-pt-dim}

Let $h$ be a GFF on a domain $D\subset \BB C$ and let $(h_\ep)$ be its circle average process. Recall that for $\alpha \geq 0$, a point $z\in D$ is called an \emph{$\alpha$-thick points} of $h$ provided
\eqbn
\lim_{\ep\rta 0} \frac{h_\ep(z)}{\log\ep^{-1}} = \alpha .
\eqen
Let $T_h^\alpha$ be the set of $\alpha$-thick points of $h$. Also let
\eqbn
\wh T_h^\alpha := \left\{z\in \BB C \,:\, \liminf_{\ep\rta 0} \frac{h_\ep(z)}{\log\ep^{-1}} \geq \alpha \right\} .
\eqen
Thick points are introduced and studied in~\cite{hmp-thick-pts}
\footnote{The authors of \cite{hmp-thick-pts} use a different normalization of the GFF from the one used in this paper, so our $\alpha$-thick points are the same as their $\alpha^2/2$-thick points}.
In particular, it is proven in~\cite[Theorem~1.2]{hmp-thick-pts} that a.s.\ $\dim_{\mcl H}(T_h^\alpha) = 2-\alpha^2/2$. In this section we will adapt the proof of~\cite[Theorem~1.2]{hmp-thick-pts} to obtain a generalization of this fact which gives the a.s.\ dimension of the intersection of $T_h^\alpha$ with a general Borel set. The lower bound from this result will be needed in the proof of the upper bound in Theorem~\ref{thm-dim-relation}.

\begin{thm} \label{thm-thick-intersect}
Let $D\subset \BB C$ be a simply connected domain and let $h$ be a zero-boundary GFF on $D$. Also let $A\subset D$ be a deterministic Borel set. If $0\leq \alpha^2/2 \leq \dim_{\mcl H} A$, then almost surely 
\eqbn
\dim_{\mcl H}\!\left(T_h^\alpha \cap A\right) =  \dim_{\mcl H}\!\left(\wh T_h^\alpha \cap A\right) =   \dim_{\mcl H} A  - \frac{\alpha^2}{2}  .
\eqen
If $\alpha^2/2 > \dim_{\mcl H} A$, then a.s.\ $T_h^\alpha\cap A = \wh T_h^\alpha \cap A= \emptyset$. 
\end{thm}

By \cite[Theorem~B.2.5]{peres-fractal}, for each $d < \dim_{\mcl H}(A)$, there exists a closed set $A'\subset A$ with $\dim_{\mcl H}(A') \geq d$. Hence we can assume without loss of generality that $A$ is closed. We make this assumption throughout the remainder of this section.

Before we commence with the proof of Theorem~\ref{thm-thick-intersect}, we observe that $\dim_{\mcl H}(T_h^\alpha \cap A)$ and $\dim_{\mcl H}(T_h^\alpha \cap A)$ are each a.s.\ equal to a constant.

\begin{lem} \label{prop-thick-zero-one}
Suppose we are in the setting of Theorem~\ref{thm-thick-intersect}. There are deterministic constants $a , \wh a \geq 0$ such that $\dim_{\mcl H} (T_h^\alpha \cap A) = a$ and $\dim_{\mcl H} (T_h^\alpha \cap A) = \wh a$ a.s.\ 
\end{lem}
\begin{proof}
Let $\{f_j\}$ be an orthonormal basis for the Hilbert space closure for the Dirichlet inner product on the set of compactly supported smooth functions in $D$, with each $f_j$ smooth and compactly supported. We can write $h = \sum_{j=1}^\infty a_j f_j$, where the $a_j$'s are i.i.d.\ standard Gaussians. Permuting a finite number of coefficients in this series expansion does not affect whether a given point is an $\alpha$-thick point for $h$, nor does it affect the dimension of $T_h^\alpha \cap A$ or $\wh T_h^\alpha \cap A$. By the Hewitt-Savage zero-one law, we obtain the statement of the lemma.
\end{proof}

\subsection{Upper bound}
\label{sec-thick-upper}

In this subsection we will prove the upper bound in Theorem~\ref{thm-thick-intersect}. It is clear that $T_h^\alpha\subset \wh T_h^\alpha$, so we only need to prove an upper bound on the dimension of $\wh T_h^\alpha \cap A$. To do this we will need a couple of basic lemmas.

\begin{lem} \label{prop-thick-uniform}
Suppose we are in the setting of Theorem~\ref{thm-thick-intersect}. Fix $\alpha \in (0,2]$ and $r>0$. 
Almost surely, there exists a random $\ol \ep = \ol\ep(\alpha , r) > 0$ such that the following is true. If we set
\eqbn
A^{\alpha,r} := \left\{z\in A \,:\, \frac{h_\ep(z)}{\log\ep^{-1}} \geq \alpha- r \: \forall \ep \in (0,\ol\ep] \right\}
\eqen
then $\dim_{\mcl H} A^{\alpha,r} \geq \dim_{\mcl H}(\wh T_h^\alpha \cap A) - r$.
\end{lem}
\begin{proof}
We have
\eqbn
 \wh T_h^\alpha \cap A  \subset \bigcup_{n=1}^\infty \left\{z\in A \,:\, \frac{h_\ep(z)}{\log\ep^{-1}} \geq \alpha- r \: \forall \ep \in (0, 1/n] \right\}
\eqen
so the statement of the lemma follows from countable stability of Hausdorff dimension. 
\end{proof}

\begin{lem} \label{prop-thick-upper}
Fix $r \in (0,1/2)$, $\rho > R > 0$, and $\alpha>0$. Let $h$ be a zero-boundary GFF on $B_\rho(0)$. For $\ep \in (0,1)$ and $z\in B_R(0)$, let $E_\ep^{\alpha,r}(z)$ be the event that there is a $w\in B_{\ep}(z)   $ such that $h_{\ep}(w) \geq (\alpha -r) \log \ep^{- 1}$. Also let $\mcl C_\ep = \mcl C_\ep(R,r)$ be the event of Lemma~\ref{prop-gff-cont-spread}. Then for each $\wt \ep \geq \ep$, we have
\eqb \label{eqn-thick-upper}
\BB P\!\left(E_\ep^{\alpha ,r} (z) \cap \mcl C_{\wt \ep} \right) \preceq \ep^{\frac{\alpha^2}{2} +o_r(1)} ,
\eqe 
with the implicit constant and the $o_r(1)$ independent of $\ep$ and uniform for $z \in B_R(0)$. 
\end{lem}
\begin{proof} 
By definition of $\mcl C_{\wt \ep}$, if $E_\ep^{\alpha,r}(z) \cap \mcl C_{\wt \ep}$ occurs, then 
\eqbn
h_{\ep^{1-r} }(z) \geq \left( \alpha - r -3\sqrt{10 r}    \right) \log \ep^{-1} .
\eqen
Since $h_{\ep^{1-r}}(z)$ is Gaussian with variance $\log \ep^{-(1-r)} + O_\ep(1)$, statement of the lemma follows from the Gaussian tail bound.
\end{proof}

\begin{proof}[Proof of Theorem~\ref{thm-thick-intersect}, upper bound]
By conformal invariance (see \cite[Section~4]{hmp-thick-pts}) we can assume without loss of generality that $D= B_\rho(0)$ for some $\rho > 0$. By countable stability of Hausdorff dimension we can assume without loss of generality that $A \subset B_R(0)$ for some $R \in (0,\rho)$. Fix $r\in (0,\alpha\wedge 1/2)$ and $p \in (0,1)$. Let $\ol\ep = \ol\ep(\alpha,r)  > 0$ and $A^{\alpha,r}$ be as in Lemma~\ref{prop-thick-uniform}. By the defining property of $A^{\alpha,r}$, we can find a deterministic $\ep_0 \in (0,1)$ such that with probability at least $1-p$, we have $\ol \ep \geq \ep_0$, in which case
\eqb \label{eqn-A^r-uniform}
   h_{\ep }(z) \geq (\alpha-r) \log \ep^{-1} ,\quad \forall \ep \in (0,\ep_0],\quad \forall z \in A^{\alpha,r}. 
\eqe  
Let $F^{\alpha,r}$ be the event that~\eqref{eqn-A^r-uniform} holds, so that $\BB P\!\left(F^{\alpha,r}\right) \geq 1-p$.

By the definition of Hausdorff dimension, for each $\beta > \dim_{\mcl H} A$ and each $k\in\BB N$, we can find a countable collection of balls $\{B_{\ep_k^j}(z_k^j)\}_{j\in\BB N}$ of radius $\ep_k^j \leq \ep_0 \wedge 2^{-k}$ such that 
\eqbn
A\subset \bigcup_{j=1}^\infty B_{\ep_k^j}(z_k^j) \quad \op{and} \quad \sum_{j=1}^\infty (\ep_k^j)^\beta \leq 2^{-k} .
\eqen 
By the definition of $A^{\alpha,r}$, if $B_{\ep_k^j}(z_k^j) \cap A^{\alpha,r} \neq\emptyset$ and the event $F^{\alpha,r}$ of~\eqref{eqn-A^r-uniform} occurs, then the event $E_{\ep_k^j}^{\alpha,r}(z_k^j)$ of Lemma~\ref{prop-thick-upper} occurs.

By Lemma~\ref{prop-thick-upper}, for each $\xi > 0$,
\eqb \label{eqn-hausdorff-sum}
\BB E\!\left(\BB 1_{\mcl C_{2^{-k}} \cap F^{\alpha,r}}  \sum_{j=1}^\infty (\ep_k^j)^\xi \BB 1_{\left(B_{\ep_k^j}(z_k^j) \cap A^{\alpha,r} \not=\emptyset\right)} \right) \leq  \sum_{j=1}^\infty (\ep_k^j)^{\xi + \alpha^2/2 + o_r(1)}    
\eqe  
with $\mcl C_{2^{-k}} = \mcl C_{2^{-k}}(R,r)$ as in Lemma~\ref{prop-gff-cont-spread}. If $\xi > \beta - \alpha^2/2$, then for sufficiently small $r$, this sum is $\leq 2^{-k}$. Since $\BB P\!\left(\mcl C_{2^{-k}} \right) \rta 1$ as $k\rta\infty$, if $F^{\alpha,r}$ occurs then it is a.s.\ the case that for infinitely many $k$, we have
\eqbn
\sum_{j=1}^\infty (\ep_k^j)^\xi \BB 1_{\left(B_{\ep_k^j}(z_k^j) \cap A^{\alpha,r} \not=\emptyset \right)}  \leq 2^{-k/2} .
\eqen 
Therefore $\dim_{\mcl H} A^{\alpha,r} \leq \beta - \alpha^2/2$ a.s. on $F^{\alpha,r}$. Since $p$ can be made arbitrarily close to 0, a.s.\ 
\eqbn
\dim_{\mcl H} (\wh T_h^\alpha\cap A) \leq \dim_{\mcl H}A^{\alpha,r} +r \leq \beta - \alpha^2/2 + r .
\eqen
Upon letting $\beta \rta \dim_{\mcl H} A$ and $r\rta 0$ we get $\dim_{\mcl H} (\wh T_h^\alpha \cap A) \leq \dim_{\mcl H}(A) - \alpha^2/2$. If $ \dim_{\mcl H}(A)  - \alpha^2/2 < 0$, then for $\beta$ sufficiently close to $\dim_{\mcl H}(A)$ and $r$ sufficiently small, the sum~\eqref{eqn-hausdorff-sum} is $\leq 2^{-k/2}$ when $\xi = 0$. Hence it is a.s.\ the case that on $F^{\alpha,r}$, it holds for arbitrarily large $k$ that none of the balls $B_{\ep_k^j}(z_k^j)$ intersect $A^{\alpha,r}$. Therefore $A^{\alpha,r} =\emptyset$ a.s., so by Lemma~\ref{prop-thick-uniform} we have $\wh T_h^{\alpha } \cap A =\emptyset$ a.s.  
\end{proof}

\subsection{Lower bound}
\label{sec-thick-lower}

In this subsection we will prove the lower bound in Theorem~\ref{thm-thick-intersect}. It suffices to prove the lower bound for $\dim_{\mcl H}(T_h^\alpha\cap A)$. By considering the intersection of $A$ with a dyadic square and re-scaling, we can assume without loss of generality that $A\subset [0,1]^2 \subset D$. We make this assumption in addition to the assumption that $A$ is closed throughout the remainder of this section.

Fix $\alpha > 0$. We make the following definitions (as in~\cite[Section~3.2]{hmp-thick-pts}). For $n\in\BB N$, let $\ep_n := 1/n!$ and $t_n := \log n!$. Define the following events for each $\ep > 0$, $z\in D$, and $n, j\in\BB N$. 
\begin{align}
E_{z,j} &:= \left\{\sup_{\ep \in [\ep_{j+1} , \ep_{j}]} |h_\ep(z) - h_{\ep_j}(z) - \alpha (\log \ep^{-1} - \log \ep_j^{-1}  ) | \leq \sqrt{\log \ep_{j+1}^{-1} - \log \ep_j^{-1}} \right\}  \\
F_{z,j} &:= \left\{|h_\ep(z)   - h_{\ep_j}(z)| \leq \log \ep^{-1} - \log \ep_j^{-1} + 1 \quad \forall \ep \leq \ep_j \right\} \\
E^n(z) &:= F_{z,n+1} \cap \bigcap_{j=1}^n E_{z,j}  
\end{align} 

For $n\in\BB N$, divide $[0,1]^2$ into $\ep_n^{-2}$ squares of side length $\ep_n $, which intersect only along their boundaries. Let $\wt{\mcl D}_n$ be the set of points in $[0,1]^2$ which are centers of these squares and for $z\in \wt{\mcl D}_n$, let $  S_n(z)$ be the square of side length $\ep_n$ centered at $z$. Let $\mcl D_n$ be the set of $z\in \wt{\mcl D}_n$ such that $S_n(z) \cap A \not=\emptyset$ and let $\mcl D_n^*$ be the set of those $z\in\mcl D_n$ for which $E^n(z)$ occurs. We define the \emph{$\alpha$-perfect-points} by
\eqb  \label{eqn-perfect-def}
\mcl P^\alpha = \mcl P^\alpha(h , A) := \bigcap_{k\geq 1} \ol{\bigcup_{n\geq k} \bigcup_{z\in \mcl D_n^*} S_n(z)} .
\eqe 
It is shown in \cite[Lemma~3.2]{hmp-thick-pts} that $\mcl P^\alpha\subset T_h^\alpha$ a.s. Since we have assumed that $A$ is closed, we a.s.\ have 
\eqb \label{eqn-perfect-contain}
\mcl P^\alpha \subset T_h^\alpha\cap A .
\eqe

We next need estimates for the probabilities of the events $E^n(z)$. 

\begin{lem} \label{prop-thick-1pt}
For $z\in \wt{\mcl D}_n$ and $n\in\BB N$, we have
\eqbn
\BB P\!\left(E^n(z)\right) \geq  \ep_n^{\alpha^2/2+ o_n(1)} 
\eqen
with the $o_n(1)$ uniform for $z\in \wt{\mcl D}_n$.
\end{lem}
\begin{proof}
Since $t\mapsto h_{e^{-t}}(z)$ evolves as a standard linear Brownian motion, \cite[Lemma~A.3]{hmp-thick-pts} applied with $T = \log\ep_{j+1}^{-1} - \log\ep_j^{-1}  $ implies that for each $z\in \wt{\mcl D}_n$ and $j\in\BB N$, we have
\eqbn
\BB P\!\left(E_{z,j}\right) \geq (\ep_{j+1}/\ep_j)^{\alpha^2/2 + o_j(1)} .
\eqen
Furthermore, we have $\BB P\!\left(F_j(z)\right) \succeq 1$. 
By the Markov property,
\eqbn
\BB P\!\left(E^n(z) \right)  =\BB P\!\left(F_{n+1}(z)\right) \prod_{j=1}^n \BB P\!\left(E_{z,j} \right)   \geq \ep_{n+1}^{\alpha^2/2 + o_n(1)} = \ep_n^{\alpha^2/2 + o_n(1)} .
\eqen
\end{proof}
 
The following is a restatement of \cite[Lemma~3.3]{hmp-thick-pts}.

\begin{lem} \label{prop-thick-2pt}
There is a constant $C>0$ depending only on $D$ and $\alpha$ such that the following is true. For each $l\in\BB N$, each $z,w\in [0,1]^2$ with $w\in S_l(z) \setminus S_{l+1}(z)$, and each $n \geq l$, 
\eqbn
\BB P\!\left(E^n(z) \cap E^n(w)\right) \leq C^l \beta_l^{-\alpha^2/2} \ep_l^{-\alpha^2/2} \BB P(E^n(z))\BB P(E^n(w))
\eqen
where
\eqbn
\beta_l = \prod_{k=1}^l e^{\frac12 \sqrt{\log k} }.
\eqen
\end{lem}

\begin{remark}
In the setting of Lemma~\ref{prop-thick-2pt}, we have $\ep_l = |z-w|^{1+o_{|z-w|}(1)}$ and $\beta_l = |z-w|^{o_{|z-w|}(1)}$, so the estimate of Lemma~\ref{prop-thick-2pt} can be re-stated as
\eqbn
\frac{\BB P\!\left(E^n(z) \cap E^n(w)\right)}{\BB P(E^n(z))\BB P(E^n(w))} \leq |z-w|^{-\alpha^2/2 + o_{|z-w|}(1)}  
\eqen
with the $o_{|z-w|}(1)$ depending only on $z$ and $w$. 
\end{remark}

\begin{proof}[Proof of Theorem~\ref{thm-thick-intersect}, lower bound]
Fix $d \in (0, \dim_{\mcl H}(A))$. By Frostman's lemma \cite[Theorem~4.30]{peres-bm} there exists a Borel probability measure $\mu$ on $A$ such that
\eqb  \label{eqn-frostman-int}
\int_A \int_A \frac{1}{|x-y|^d} \, d\mu(x) \, d\mu(y) < \infty. 
\eqe 
We extend $\mu$ to $\BB C$ by setting $\mu(B) = \mu(B\cap A)$ for each Borel set $B\subset \BB C$. From~\eqref{eqn-frostman-int}, we obtain that for any $a , b \geq 0$ with $a+b \leq d$, 
\alb
&\int_A \int_A \frac{1}{|x-y|^d} \, d\mu(x) \, d\mu(y) \\
&\qquad = \sum_{z\not= w\in \mcl D_n} \int_{S_n(z)} \int_{S_n(w)} \frac{1}{|x-y|^d} \, d\mu(x) \, d\mu(y)  + \sum_{z\in \mcl D_n} \int_{S_n(z)} \int_{S_n(z)} \frac{1}{|x-y|^d}     \, d\mu(x) \, d\mu(y)  \\
&\qquad \succeq \sum_{z\not= w\in \mcl D_n} |z-w|^{-a} \int_{S_n(z)} \int_{S_n(w)} \frac{1}{|x-y|^b} \, d\mu(x) \, d\mu(y)  + \sum_{z\in \mcl D_n} \ep_n^{-a}  \int_{S_n(z)} \int_{S_n(z)} \frac{1}{|x-y|^b}     \, d\mu(x) \, d\mu(y)  .
\ale
Hence for any such $a$ and $b$, 
\eqb \label{eqn-frostman-2terms}
\sum_{z\not= w\in \mcl D_n} |z-w|^{-a} \int_{S_n(z)} \int_{S_n(w)} \frac{1}{|x-y|^b} \, d\mu(x) \, d\mu(y)  + \sum_{z\in \mcl D_n} \ep_n^{-a}\int_{S_n(z)} \int_{S_n(z)} \frac{1}{|x-y|^b}     \, d\mu(x) \, d\mu(y)  \preceq 1  .
\eqe

For $n\in\BB N$ define a measure $\nu_n$ on $A$ by
\eqbn
d\nu_n(x) = \sum_{z\in \mcl D_n} \frac{  \BB 1_{E^n(z)} \BB 1_{S_n(z) }(x)}{\BB P\!\left(E^n(z)\right)} \, d\mu(x) . 
\eqen
Observe that
\eqbn
\BB E\!\left(\nu_n(A)\right) = \sum_{z\in \mcl D_n} \mu(S_n(z)) = \mu(A) = 1.
\eqen
By Lemmas~\ref{prop-thick-1pt} and~\ref{prop-thick-2pt}, 
\alb
\BB E\!\left(\nu_n(A)^2\right) 
&= \sum_{z\not=w \in \mcl D_n}  \frac{ \BB P\!\left(E^n(z) \cap E^n(w)\right)     }{\BB P\!\left(E^n(z)\right) \BB P\!\left(E^n(w)\right)} \mu\!\left(  S_n(z)   \right)    \mu\!\left(  S_n(w)   \right)   +   \sum_{z \in \mcl D_n}  \frac{    \mu\!\left(  S_n(z)   \right)^2    }{\BB P\!\left(E^n(z)\right)  }       \\
&\preceq \sum_{z\not=w \in \mcl D_n} \frac{\mu\!\left(  S_n(z)   \right)    \mu\!\left(  S_n(w)   \right) }{ |z-w|^{   \alpha^2/2 + o_{|z-w|}(1)}}        +   \sum_{z \in \mcl D_n} \frac{\mu(S_n(z))^2 }{\ep_n^{  \alpha^2/2   +o_n(1)}}  .       
\ale
By~\eqref{eqn-frostman-2terms} applied with $b = 0$, this is bounded above by an $n$-independent constant provided $d - \alpha^2/2 > 0$. 

Similarly, for $b > 0$ we have 
\alb
&\BB E\!\left(\int_A \int_A \frac{1}{|x-y|^b} \, d\nu_n(x) \, d\nu_n(y) \right) \\
&\qquad  =  \sum_{z\not=w \in \mcl D_n}  \frac{ \BB P\!\left(E^n(z) \cap E^n(w)\right)     }{\BB P\!\left(E^n(z)\right) \BB P\!\left(E^n(w)\right)} \int_{S_n(z)} \int_{S_n(w)} \frac{1}{|x-y|^b} \,d\mu(x) \,d\mu(y) \\
&\qquad\qquad  +   \sum_{z \in \mcl D_n}  \frac{1  }{\BB P\!\left(E^n(z)\right)  } \int_{S_n(z)} \int_{S_n(z)} \frac{1}{|x-y|^b} \,d\mu(x) \,d\mu(y)       \\
&\qquad  \preceq  \sum_{z\not=w \in \mcl D_n}  |z-w|^{-\alpha^2/2 + o_{|z-w|}(1)} \int_{S_n(z)} \int_{S_n(w)} \frac{1}{|x-y|^b} \,d\mu(x) \,d\mu(y) \\
&\qquad\qquad  +   \sum_{z \in \mcl D_n} \ep_n^{ -\alpha^2/2   +o_n(1)} \int_{S_n(z)} \int_{S_n(z) } \frac{1}{|x-y|^b} \,d\nu_n(x) \,d\nu_n(y)       .
\ale
By~\eqref{eqn-frostman-2terms}, this is bounded above by an $n$-independent constant provided $b < d - \alpha^2/2$. It follows from the usual argument (see the proof of \cite[Lemma~3.4]{hmp-thick-pts}) that for such a $b$, it holds with positive probability that we can find a weak subsequential limit $\nu$ of the measures $\nu_n$ such that $\nu$ is supported on $\mcl P^\alpha$, $\nu(\mcl P^\alpha) > 0$ and 
\eqbn
\int_{\mcl P^\alpha} \int_{\mcl P^\alpha} \frac{1}{|x-y|^b} \, d\nu(x) \, d\nu(y) <\infty. 
\eqen
By~\eqref{eqn-perfect-contain} and \cite[Theorem~4.27]{peres-bm}, it holds with positive probability that $\dim_{\mcl H}(A \cap T_h^\alpha) \geq b$.  By Lemma~\ref{prop-thick-zero-one}, this probability is in fact equal to 1 for each $b < d-\alpha^2/2$. Therefore $\dim_{\mcl H}(A \cap T_h^\alpha) \geq d - \alpha^2/2$.
\end{proof}

\section{Proof of Theorem~\ref{thm-dim-relation}}
\label{sec-dimH}

\subsection{Upper bound}
\label{sec-dimH-upper}

In this subsection we will prove the upper bound in Theorem~\ref{thm-dim-relation}. 

\begin{proof}[Proof of Theorem~\ref{thm-dim-relation}, upper bound]
We start with some reductions. By the countably stability of Hausdorff dimension, we can assume without loss of generality that a.s.\ $X \subset  B_R(0) \setminus B_{1/R}(0)$ for some deterministic $R> 0$.  By Lemma~\ref{prop-annulus-abs-cont}, we can couple $h$ with a whole-plane GFF $\wt h$ normalized so that its circle average over $\bdy \BB D$ is 0 in such a way that the $\gamma$-quantum area measures $\mu_h$ and $\mu_{\wt h}$ are a.s.\ mutually absolutely continuous on $B_R(0) \setminus B_{1/R}(0)$, with Radon-Nikodym derivative bounded above and below by (random) positive constants. In such a coupling, it holds that for each interval $I\subset \BB R$ with $\eta'(I) \subset B_R(0) \setminus B_{1/R}(0)$ that $|I| \asymp \mu_{\wt h}( \eta'(I) )$ with random but $I$-independent implicit constants. In particular, $ \dim_{\mcl H} (\eta')^{-1}(X)$ is unchanged if we parameterize $\eta'$ by $\mu_{\wt h}$ instead of $\mu_h$.  Hence we can assume that $h$ is a whole-plane GFF normalized so that its circle average over $\bdy\BB D$ is 0, instead of the circle average embedding of a $\gamma$-quantum cone. We make this assumption throughout the remainder of the proof. 
  
Let $\alpha \in (0,2]$ and $r \in (0,1/2)$. By~\cite[Proposition~3.2]{ig1} the law of the restriction of $h$ to $B_R(0)$ is absolutely continuous with respect to the law of the corresponding restriction of a zero-boundary GFF $h^0$ on $B_{2R}(0)$, minus a random constant $ C$ equal to the circle average of $h^0$ over $\bdy\BB D$. Since the set $X$ is independent from $h+C$, it follows from Theorem~\ref{thm-thick-intersect} and Lemma~\ref{prop-thick-uniform} that we can find a random set $X^{\alpha,r} \subset X$ and a random $\ol\ep > 0$ such that a.s.\ 
\eqb \label{eqn-perfect-dim}
 \dim_{\mcl H} X^{\alpha , r}   \geq   \dim_{\mcl H} X - \frac{\alpha^2}{2} - r
\eqe  
and a.s.\ 
\eqbn 
  h_\ep(z) + C \geq (\alpha -r) \log\ep^{-1} , \quad \forall \ep \in (0,\ol\ep] \quad \forall z \in X^{\alpha,r} .
\eqen 
By decreasing $\ol\ep$ we can arrange that in fact
\eqb \label{eqn-thick-uniform}
h_\ep(z) \geq (\alpha -r) \log\ep^{-1} , \quad \forall \ep \in (0,\ol\ep] \quad \forall z \in X^{\alpha,r} .
\eqe 	 
 
Now let $\wh X \subset \BB R$ be as in the theorem statement for our given choice of $X$. We will prove an upper bound for $\dim_{\mcl H}\!\left( X^{\alpha , r} \right)$ in terms of $\dim_{\mcl H}(\wh X)$. Let
\[
\wh X^{\alpha,r}  := (\eta')^{-1} \!\left(X^{\alpha , r} \right) \cap \wh X .
\]
For $\ep > 0$, let $\mcl E_\ep =\mcl E_\ep(R,r)$ and $\mcl G_\ep = \mcl G_\ep( R,r)$ be defined as in Propositions~\ref{prop-sle-bubble} and~\ref{prop-thick-local}, respectively. Also let
\eqbn
\mcl S_\ep := \left\{\sup_{z\in B_R(0)} h_1(z) \leq \sqrt{\log \ep^{-1}} \right\} .
\eqen
Note that $\BB P\!\left(\mcl E_\ep \cap \mcl G_\ep \cap \mcl S_\ep\right) \rta 1$ as $\ep \rta 0$. 
 
By the definition of Hausdorff dimension, for each $\zeta > 0$, we can find a countable collection $\mcl I_\zeta$ of intervals of length at most $\zeta$, each of which contains a point of $\wh X^{\alpha,r}$,
such that 
\eqb \label{eqn-cover-properties}
\wh X^{\alpha,r} \subset \bigcup_{I\in\mcl I_\zeta}  I \quad \op{and}\quad \sum_{I\in\mcl I_\zeta} \op{diam}(I)^{\dim_{\mcl H}(\wh X) + r } \leq \zeta .
\eqe  

We claim that on $\mcl E_\ep \cap \mcl G_\ep \cap \mcl S_\ep$, there a.s.\ exists a random $\zeta_0 > 0$ such that for each $\zeta \in (0,\zeta_0]$ and each choice of $\wh X \subset \BB R$ and cover $\mcl I_\zeta$ as above that
\eqb \label{eqn-spatial-diam-upper}
\op{diam} \eta'( I) \leq    \op{diam}(I)^{ (2+\gamma^2/2- \gamma\alpha )^{-1}  + o_r(1) } \qquad \forall  I \in  \mcl I_\zeta ,
\eqe
with the $o_r(1)$ deterministic and independent of $I$.
Indeed, let $\ol\ep$ be as in~\eqref{eqn-thick-uniform} and set $\zeta_0 = \ol\ep \wedge \ep^{3}$. 
Suppose $\mcl E_\ep \cap \mcl G_\ep \cap \mcl S_\ep$ occurs, $\zeta \in (0,\zeta_0]$, $\wh X$ and $\mcl I_\zeta$ are as above, and $I \in \mcl I_\zeta$. Let $\delta_I :=   \op{diam}  \eta'(I) \wedge \ep^2$. By definition of $\mcl E_\ep$, we can find $z\in X^{\alpha , r}$ and $w\in  \eta'(I) \cap B_{\delta_I }(z)$ such that $B_{ \delta_I^{\frac{1}{1-r}}}(w) \subset \eta'(I)$. 
By definition of $\mcl G_\ep$, we have
\eqbn
\mu_h\!\left(\eta'(I)\right) \geq  \delta_I^{2+\frac{\gamma^2}{2} + o_r(1)} e^{\gamma h_{\delta_I }(z)} .
\eqen
Since $z \in  X^{\alpha , r}$,~\eqref{eqn-thick-uniform} and~\eqref{eqn-spatial-diam-upper} imply that if $\zeta \leq \zeta_0$, then $e^{\gamma h_{\delta_I }(z)} \geq \delta_I^{-\gamma \alpha  + \gamma r}$. Hence for such a $\zeta$, we have
\eqbn
\mu_h\!\left(\eta'(I)\right) \geq \delta_I^{2+\frac{\gamma^2}{2} - \gamma\alpha + o_r(1)} .  
\eqen
Since $\eta'$ is parameterized by quantum mass, we infer that
\eqbn
\op{diam}(I) \geq  \delta_I^{2+\frac{\gamma^2}{2} - \gamma\alpha + o_r(1)} . 
\eqen
Since $\op{diam}(I) \leq \zeta^{1/2}$, this implies~\eqref{eqn-spatial-diam-upper}. 

By~\eqref{eqn-cover-properties}, $\{\eta'(I) \,:\, I\in\mcl I_\zeta\}$ is a cover of $X^{\alpha , r}$. Moreover, if we are given $r' > 0$ and we choose $r >0$ sufficiently small, depending only on $r'$, then for $\zeta \in (0,\zeta_0]$ we have by~\eqref{eqn-cover-properties} and~\eqref{eqn-spatial-diam-upper} that
\eqbn
\sum_{I\in\mcl I_\zeta } (\op{diam} \eta'(I))^{ \left( \dim_{\mcl H}(\wh X) + r'\right) \left(2+\frac{\gamma^2}{2} - \gamma \alpha\right)   } \leq \zeta .
\eqen
Hence for such a choice of $r$, whenever $\mcl E_\ep \cap \mcl G_\ep \cap \mcl S_\ep$ occurs it holds for every choice of $\wh X$ as in the theorem statement that  
\[
\dim_{\mcl H}\!\left(X^{\alpha , r} \right) \leq \left( \dim_{\mcl H}(\wh X) + r'\right) \left(2+\frac{\gamma^2}{2} - \gamma\alpha\right)      .
\]
By letting $\ep\rta 0$ we infer that this relation holds a.s.\ for every choice of $\wh X$ simultaneously.  
By~\eqref{eqn-perfect-dim}, it is a.s.\ the case that for each choice of $\wh X$ as in the theorem statement, we have
\eqbn
\dim_{\mcl H}(X)   \leq   \left( \op{dim}_{\mcl H}(\wh X) + r'\right) \left(2+\frac{\gamma^2}{2} - \gamma\alpha\right) + \frac{\alpha^2}{2}   +r  .
\eqen
The right side is minimized when $r=r'=0$ by taking $\alpha= \gamma  \dim_{\mcl H}(\wh X)$. Since $r$ is arbitrary and $r'$ can be made as small as we like by shrinking $r$, this yields the upper bound in~\eqref{eqn-dim-relation}. 
\end{proof}

\subsection{Lower bound}
\label{sec-dimH-lower}

We will prove the lower bound in Theorem~\ref{thm-dim-relation} by covering $X$ with balls $B$ of radius $\ep_B$, such that $\sum_{B}\ep_B^d$ is small for some  $d>\dim_{\mcl H}(X)$. We obtain a cover of the time set $\widehat{X}$ by considering the pre-images of these balls under $\eta'$. The length of the intervals covering $\widehat{X}$ is estimated by considering the quantum mass of the balls in our cover via the circle average process, and by bounding the number of time intervals corresponding to each ball in the cover of $X$. The proof also relies on Proposition~\ref{prop-sle-bubble}, Lemma~\ref{prop-harmonic-var}, and \cite[Theorem~2.11]{rhodes-vargas-review}.

\begin{lem}
For each $R>1$, $r\in(0,1)$, $z\in B_R(0)$, $\tilde\ep>0$, and $\ep>0$, the ball $B_{\epsilon^{1-r}}(z)$ can be written as a union of sets of the form $\eta'(I)$ which intersect only along their boundaries, where $I$ is an interval that is not contained in any larger interval $I'$ satisfying $\eta'(I')\subset B_{\epsilon^{1-r}}(z)$. On the event $\mathcal{E}_{\tilde\ep^{1-r}}=\mathcal{E}_{\tilde\ep^{1-r}}(r,R)$ of Proposition~\ref{prop-sle-bubble}, the number of such sets that intersect $B_\epsilon(z)$, is bounded above by $\ep^{-3r}$ for all sufficiently small $\ep$, i.e., for $\ep<\ep(r,\tilde\ep)$.
\label{low_res2}
\end{lem}
\begin{proof}
On $\mathcal{E}_{\tilde\ep^{1-r}}$, every time interval $I$ of the form above with $\eta'(I)\cap B_\epsilon(z)\neq\emptyset$ and $\ep>0$ sufficiently small satisfies Area$(\eta'(I))\geq (\ep^{1-r}-\ep)^{\frac{2}{1-r}}=\epsilon^{2+o_r(1)}$. Since $\eta'(I)\subset B_{\epsilon^{1-r}}(z)$ and $B_{\epsilon^{1-r}}(z)$ has area $\ep^{2-2r}=\epsilon^{2+o_r(1)}$, the lemma follows. The exact exponent $-3r$ is obtained by dividing the area of $B_{\epsilon^{1-r}}(z)$, by the bound for the area of $\eta'(I)$.
\end{proof}

\begin{lem}
Let $h$ be a whole-plane GFF normalized such that $h_1(0)=0$, and let $R>0$ and $z\in B_R(0)$. For $0 \leq \beta< \frac{4}{\gamma^2}$ and $\ep \in (0,1)$, we have
\eqbn
\BB E[\mu_h(B_\ep(z))^\beta] \leq \ep^{f(\beta)+o_\ep(1)},
\eqen
where $f(\beta)=(2+\frac{\gamma^2}{2})\beta-\frac{\gamma^2}{2}\beta^2$ and the $o_\ep(1)$ depends on $\alpha$, $\beta$, and $R$, but not on $z$.
\label{low-res3}
\end{lem} 
\begin{proof}
Defining $\phi_z^\ep(w)=\ep w+z$ we have $h\circ \phi_z^\ep=\dot h+\mathfrak h$, for $\dot h$ a zero-boundary GFF in $B_{2}(0)$ and $\mathfrak h$ harmonic in $B_2(0)$ and independent of $\dot h$. As explained after \eqref{eqn-mu_h-ball-shift} the average of $\frk h$ around $\partial B_2(0)$ is equal to the circle average $h_{2\ep}(z)$. By the coordinate change formula for quantum surfaces we have
\eqbn
\begin{split}
\BB E\big(\mu_h(B_\epsilon(z))^\beta \big) 
&= \BB \ep^{(2+\gamma^2/2)\beta}
\BB E\big(\mu_{\dot h+\mathfrak h}(B_1(0))^\beta \big)\\
&\leq \ep^{(2+\gamma^2/2)\beta}\BB E\big(e^{\gamma\beta h_{2\ep}(z)}\times \mu_{\dot h}(B_1(0))^\beta \times \sup_{w\in B_1(0)} e^{\gamma \beta(\frk h(w)-h_{2\ep}(z))}\big).
\end{split}
\eqen 
Let $r>0$, and define the event $A_{r,\ep}$ by
\eqbn
A_{r,\ep}=\{\sup_{w\in B_1(0)} e^{\gamma (\frk h (w)-h_{2\ep}(z))}> \ep^{-r}\}
\eqen
By Lemma~\ref{prop-harmonic-var}, we have $\BB P(A_{r,\ep})=o^\infty_\ep(\ep)$. Since $h_{2\ep}(z)$ is Gaussian with variance at most $(1+o_\ep(1)) \log\ep^{-1}$, for each $\wt \beta > 0$ we have 
\eqb
\BB E\big(e^{\tilde\beta h_{2\ep}(z)}\big)=\ep^{-\frac 12 \tilde{\beta^2}+o_\ep(1)}.
\label{low2}
\eqe
By \cite[Theorem~2.11]{rhodes-vargas-review}, $\mu_{\dot h}(B_1(0))$ has finite moments of all orders $<4/\gamma^2$. Choose $p_j>1$, $j=1,2,3,4$, such that $\beta p_2<\frac{4}{\gamma^2}$ and $\sum_{j=1}^4 p_j^{-1}=1$. By H\"older's inequality,
\eqbn
\begin{split}
\BB E\big(e^{\gamma\beta h_{2\ep(z)}}&\times \mu_{\dot h}(B_1(0))^\beta \times \sup_{w\in B_1(0)} e^{\gamma \beta(\frk h(w)-h_{2\ep}(z))}\BB 1_{A_{r,\ep}}\big)\\
&\leq 
\BB E\big(e^{\gamma \beta p_1 h_{2\ep}(z)}\big)^{p_1^{-1}} \times
\BB E\big( \mu_{\dot h}(B_1(0))^{\beta p_2}\big)^{p_2^{-1}} \times
\BB E\big( \sup_{w\in B_1}e^{\gamma\beta p_3(\frk h (w)-h_{2\ep}(z))}\big)^{p_3^{-1}} \times
\BB P\big({A_{r,\ep}}\big)^{p_4^{-1}} \\
&\leq \ep^{-\frac 12 \gamma^2\beta^2 p_1+o_\ep(1)} \BB P\big({A_{r,\ep}}\big)^{p_4^{-1}}\\
&=o_\ep^\infty(\ep),
\end{split}
\eqen
which implies
\eqbn
\BB E\big(\mu_h(B_\epsilon(z))^\beta \big) \leq 
\ep^{(2+\gamma^2/2-r)\beta} \BB E\big(\ep^{\gamma\beta h_{2\ep}(z)} \mu_{\dot h}(B_1(0))^\beta \big) + o_\ep^\infty(\ep)= \ep^{(2+\gamma^2/2-r)\beta-\frac 12\gamma^2\beta^2+o_\ep(1)}
\eqen
by independence of $\dot h$ and $h_{2\ep}$.
\end{proof}

We are now ready to prove the lower bound of our main theorem:

\begin{proof}[Proof of Theorem~\ref{thm-dim-relation}, lower bound]
As in the proof of the upper bound in Section~\ref{sec-dimH-upper}, we can assume without loss of generality that $X\subset B_R(0) \setminus B_{1/R}(0)$ for some fixed $R >0$ and we can replace $h$ with a whole-plane GFF normalized so that its circle average over $\bdy\BB D$ is 0. 

 If dim$_{\mcl H}(X)=2$ the lower bound clearly holds, so we assume dim$_{\mcl H}(X)<2$. Let $\beta_0\in[0,\frac{4}{\gamma^2})$ be the (random) solution of $\dim_{\mcl H}(X)=f(\beta_0)$, and let $\beta\in(\beta_0,\frac{4}{\gamma^2})$. Then choose some $d\in(\dim_{\mcl H}(X),f(\beta))$, and some $r>0$ satisfying
\eqb
d<-3r+(1-r)f(\beta).
\label{low1}
\eqe
Then choose a sequence $\{\delta_n\}_{n\in\BB N}$, such that $\delta_n \in (0,1)$ for each $n$ and $\sum_{n=1}^\infty \delta_n^{r}<\infty$. Since $d>\dim_{\mcl H}(X)$, we can find, for each $n\in\BB N$, a random collection of balls $\mcl{B}_n$, measurable with respect to $\sigma(X)$ and covering $X$, such that 
\eqbn
\sum_{B\in\mcl{B}_n} \ep_B^{d}<\delta_n,
\eqen
where $\ep_B$ denotes the radius of $B$. For any $B\in\mcl{B}_n$, let $\mathcal I(B)$ denote the set of intervals $I$ as defined in Lemma~\ref{low_res2}, such that $\eta'(I)$ intersects $B$, and let $B'$ be the ball of radius $\ep_B^{1-r}$ centered at the same point as $B$. Let $\tilde{\ep}>0$, and assume without loss of generality that $\ep_B<\ep(r,\tilde{\ep})$ for all $B\in\mcl{B}_n$ and $n\in\BB N$ as defined in Lemma~\ref{low_res2}. Define the random variable $Z_n$ by 
\eqbn
\begin{split}
Z_n&= \BB 1_{\mathcal{E}_{\tilde{\ep}^{1-r}}} \sum_{B\in\mcl{B}_n}\sum_{I\in\mathcal{I}(B)} |I|^\beta,
\end{split}
\eqen
where $\mathcal{E}_{\tilde\ep^{1-r}}=\mathcal{E}_{\tilde\ep^{1-r}}(r,R)$ is the event of Proposition~\ref{prop-sle-bubble}.
By using $\sum_{I\in\mathcal I(B)}|I|\leq\mu(B')$ and $|\mathcal I(B)|<\ep_B^{-3r}$, we have
\eqbn
Z_n\leq\BB 1_{\mathcal{E}_{\tilde{\ep}^{1-r}}}
\sum_{B\in\mcl{B}_n}\ep_B^{-3r}\mu(B')^\beta.
\eqen
It follows from Lemma~\ref{low-res3} and~\eqref{low1} that for all sufficiently large $n$,
\eqbn
\BB E(Z_n \,|\, X) \leq \sum_{B\in\mcl{B}_n}\ep_B^{-3r+(1-r)f(\beta)+o_{n} (1)}\leq\sum_{B\in\mcl{B}_n}\ep_B^{d}<\delta_n.
\eqen
By Chebyshev's inequality, $\BB P(Z_n>\delta_n^{1-r}   )\leq\delta_n^r$, so by the Borel-Cantelli lemma the event $\{Z_n>\delta_n^{1-r}\}$ happens at most finitely often almost surely. It follows that $Z_n\rightarrow 0$ almost surely, and by letting $\tilde{\ep}\rightarrow 0$ and using Proposition~\ref{prop-sle-bubble}, a.s.\
\eqbn
\lim_{n\rta\infty} \sum_{B\in\mcl{B}_n}\sum_{I\in\mathcal{I}(B)} |I|^\beta =0.
\eqen
This completes the proof, since $\cup_{B\in\mcl{B}_n,I\in\mathcal I(B)}I$ is a cover of any set $\wh X$ as in the theorem statement.
\end{proof}
   
\section{Multiple points of space-filling SLE}
\label{sec-multiple-pt}

SLE$_{\kappa'}$ for $\kappa'\in(4,8)$ does not have triple points, and for $\kappa'\geq 8$ the set of triple points is countable, see \cite[Remark~5.3]{miller-wu-dim}. Space-filling SLE$_{\kappa'}$, $\kappa'\in(4,8)$, however, has uncountably many points with multiplicity at least $3$.  Moreover, as $\kappa' \downarrow 4$, the maximal multiplicity of the path a.s.\ tends to $\infty$. In this section we will calculate the a.s.\ Hausdorff dimension of $m$-tuple points for space-filling SLE$_{\kappa'}$ for $m\geq 3$. We remark that this dimension could also have been derived from \cite[Theorem~1.4]{miller-wu-dim} by arguing that the $m$-tuple points of space-filling SLE$_{\kappa'}$ have the same dimension as the $m$-tuple points of an SLE$_{\kappa'}(\kappa'-6)$.

As in the above sections, let $\eta'$ be a whole--plane space-filling SLE$_{\kappa'}$ parameterized by quantum area with respect to an independent $\gamma$-quantum cone $(\BB C,h,0,\infty)$. Let $Z=(L,R)$ denote the associated quantum boundary length processes, and recall that $Z$ is a two-dimensional Brownian motion with covariance given by~\eqref{eq-multiplept-10}.

Let $m\geq 3$. Except for the countable number of triple points corresponding to local minima of $L$ or $R$, there is a bijection between $m$-tuple points of SLE$_{\kappa'}$ and the $(m-2)$-tuple cone times of $Z$, which will be defined just below.  Recall the definition of an ordinary ($\pi/2$-)cone time in Definition~\ref{def-conetime}, in addition to the associated time $v(t)$ defined in~\eqref{eqn-cone-exit-time}. Let $t$ be a cone time of $Z$ (resp.\ cone time of the time-reversal of $Z$), and define the function $u$ (resp.\ $v^R$ and $u^R$) by
\eqbn
\begin{split}
v^R(t) &= \sup\{s<t\,:\,R_s<R_t\text{ or }L_s<L_t\},\\
u(t)&=\inf\left\{s>t\,:\,\inf_{s'\in [t,s]}R_{s'}<R_t\text{ and }\inf_{s'\in[s,t]}L_{s'}<L_t\right\},\\
u^R(t)&=\sup\left\{s<t\,:\,\inf_{s'\in[s,t]}R_{s'}<R_t\text{ and }\inf_{s'\in[s,t]}L_{s'}<L_t\right\}.
\end{split}
\eqen
An $m$-tuple cone time is a generalization of a cone time (see Figure~\ref{fig-multiplept1}).
\begin{defn}
Let $m\geq 3$ and $\BB t\in \BB R^{m-1}$, $\BB t=(t_0,\ldots,t_{m-2})$. Then $\BB t$ is an $(m-2)$-tuple cone vector of type (I) (resp.\ (II)) if there exists a $t_{m-1}>t_0$ such that the following properties are satisfied for any $j\in\{1,\ldots,m-2\}$.   For $j$ even (resp.\ odd) $t_j$ is a cone time of $Z$, and we have either $t_{j+1}=v(t_j)$ and $t_{j-1}=u(t_j)$, or $t_{j+1}=u(t_j)$ and $t_{j-1}=v(t_j)$. For $j$ odd (resp.\ even) $t_j$ is a cone time of the time-reversal of $Z$, and we have either $t_{j+1}=v^R(t_j)$ and $t_{j-1}=u^R(t_j)$, or $t_{j+1}=u^R(t_j)$ and $t_{j-1}=v^R(t_j)$. 

Let $\wt{\mcl T}(m)\subset\BB R^{m-1}$ denote the set of $(m-2)$-tuple cone vectors, and let ${\mcl T}(m)\subset \wt{\mcl T}(m)$ denote the set of cone vectors satisfying (I), and for which $t_0$ is running infimum of $L$.

We say that $t\in\BB R$ is an $(m-2)$-tuple cone time if $t=t_1$ for some $(m-2)$-tuple cone vector $\BB t$. Let $\wt{\mathfrak T}(m)\subset\BB R$ denote the set of $(m-2)$-tuple cone times, and let ${\mathfrak T}(m)\subset \wt{\frk T}(m)$ denote the set of cone times corresponding to elements of ${\mcl T}(m)$.
\label{def-multiplept-1}
\end{defn}

\begin{remark}
Note that $\eta'(t_j)=\eta'(t_0)$ for all $j\in\{1,\ldots,m-1\}$. Also note that, for $\BB t\in\mcl T(m)$, we have $t_{2j-2},t_{2j}<t_{2j-1}$, $L_{2j-1}=L_{2j-2}$ and $R_{2j-1}=R_{2j}$ for all relevant $j$, in particular $t_{m-1}$ is a running infimum of $R$ (resp.\ $L$) for odd (resp.\ even) $m$. Furthermore, note that we have chosen to let an $(m-2)$-tuple cone vector be represented by an $(m-1)$-dimensional vector, i.e., in the definition of cone vectors above we have chosen to include $t_0$ in the vector in addition to the $(m-2)$ cone times, while we have not included $t_{m-1}$. We chose not to include $t_{m-1}$ in order to simplify the calculation of some probabilities in Sections~\ref{sec-multiple-pt1} and~\ref{sec-multiple-pt2}, while we did include $t_0$ since it will be needed for one of our regularity conditions in Section~\ref{sec-multiple-pt2}. By symmetry $\wt{\mcl T}(m)$ (resp.\ $\wt{\frk T}(m)$) and $\mcl T(m)$ (resp.\ $\frk T(m)$) have the same Hausdorff dimension almost surely.
\end{remark}

\begin{figure}[ht!]
\begin{center}
\includegraphics[scale=0.92]{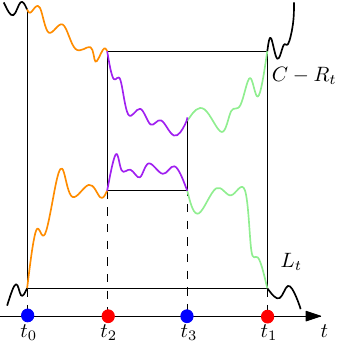}\,\,\,\,\,\,
\includegraphics[scale=0.92]{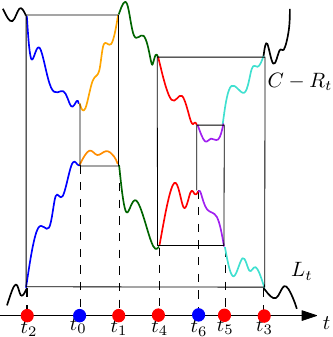}\,\,\,\,\,\,
\includegraphics[scale=0.92]{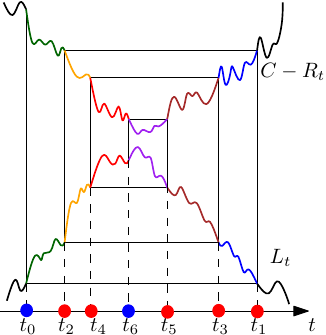}
\newline\newline
\includegraphics[scale=0.92]{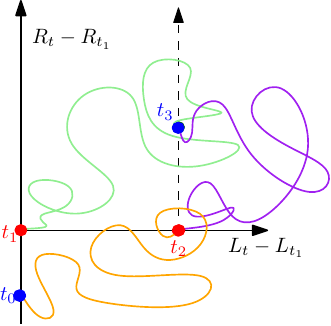}\,\,\,\,\,\,
\includegraphics[scale=0.92]{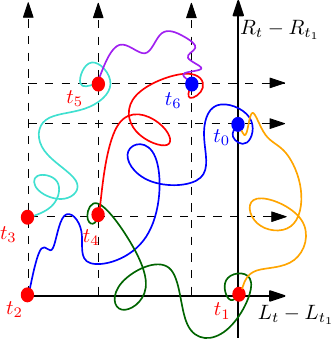}\,\,\,\,\,\,
\includegraphics[scale=0.92]{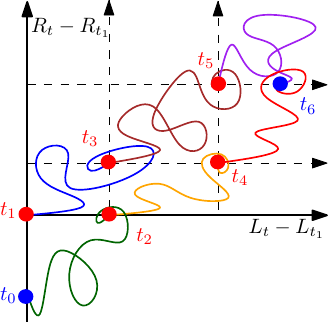}
\end{center}
\caption{Points of multiplicity $m$ can be illustrated as a rectangular path of $m$ vertical line segments and $m-1$ horizontal line segments in the peanosphere construction of \cite{wedges}, see the three top figures for some examples with $m=4$ and $m=7$. Each horizontal line segment intersects $L$ (or $C-R$) at its end points, and $L$ (or $C-R$) is above (or below) the line segment between the intersection times, implying that $L$ (or $R$) has a local running infimum where the line segment ends, and that its time-reversal has a local running infimum where the line segment starts. The local running infima correspond to a sequence of $\frac \pi 2$-cone times of the planar Brownian motion $(L,R)$, see the bottom figures. The $(m-2)$ times corresponding to $\frac{\pi}{2}$-cone times are marked with red dots, while the times $t_0$ and $t_{m-1}$, which are running infima of $L$ or $R$, and correspond to the same spatial point of $\eta'$, are marked with blue dots.} 
\label{fig-multiplept1}
\end{figure}

We now state the main results of this section.

\begin{thm} 
Let $\kappa'\in(4,8)$. Let $m\in[3,(2\kappa'-4)/(\kappa'-4)] \cap\BB N$. Then a.s.
\eqb
\dim_{\mcl H} \wt{\mathfrak T}(m)
=\dim_{\mcl H} \wt{\mcl T}(m) 
= \frac 12-(m-2)\left(\frac{\kappa'}{8}-\frac 12\right).
\label{eq-multiplept-19}
\eqe
If $m>(2\kappa'-4)/(\kappa'-4)$, $\wt{\mathfrak T}(m)$ and $\wt{\mcl T}(m)$ are empty a.s.
\label{prop-multiplept-1}
\end{thm}

Throughout this section, we let
\[ \theta = \frac{4\pi}{\kappa'} \]
be as in~\eqref{eq-multiplept-10}.

\begin{remark}
By a linear transformation it follows that~\eqref{eq-multiplept-19} also gives the dimension of the set of $(m-2)$-tuple $\theta$-cone times of standard planar Brownian motion.
\end{remark}

Theorem~\ref{prop-multiplept-1} and Theorem~\ref{thm-dim-relation} imply:
\begin{cor}
Let $\kappa'\in (4,8)$ and $m\in[3,(2\kappa'-4)/(\kappa'-4)] \cap\BB N$. The Hausdorff dimension of $m$-tuple points of space-filling SLE$_{\kappa'}$ is a.s.\ equal to
\eqbn
\frac{(4m - 4 - \kappa'(m-2))(12 +(\kappa'-4)m)}{ 8 \kappa'}.
\eqen
If $m>(2\kappa'-4)/(\kappa'-4)$, the set of $m$-tuple points is a.s.\ empty.
\label{prop-multiplept-2}
\end{cor}

\begin{proof}
The corollary follows from Theorem~\ref{thm-dim-relation} if we can show that, excluding the countable number of triple points corresponding to local minima of $L$ or $R$, there is a bijection between $m$-tuple points and $(m-2)$-tuple cone vectors. It follows directly from the peanosphere construction of \cite{wedges} that any $(m-2)$-tuple cone time gives us an $m$-tuple point. 

Conversely, note that an $m$-tuple point of $\eta'$ not corresponding to an element of $\wt{\mcl T}(m)$, would correspond to either (a) a vector $\BB t\in\BB R^{m-1}$ for which at least one of the elements $t\in\BB R$ is a local minimum of $L$ or $R$, or (b) a vector $\BB t\in\BB R^{m-1}$ for which at least one of the elements $t\in\BB R$ is a running infimum of $L$ and a running infimum of the time-reversal of $R$, or vice-versa. The set of vectors satisfying (a) is empty (except in the case $m=3$, when it is countable), since the set of local minima is countable, and the set of vectors satisfying (b) is empty by \cite[Theorem~1]{shimura88}.
\end{proof}

For $\delta\in(0,m^{-1})$, let
\eqbn
\mathcal T(\delta,m)=\left\{\BB t\in \mathcal T(m)\,:\,|t_i-t_j|>\delta,\,t_j\in(0,1),\,\forall i,j\in \{0,\ldots,m-2\}\right\}.
\eqen
By the self-similarity of Brownian motion and the stability of Hausdorff dimension under countable unions, it is sufficient to calculate the dimension of $\mathcal T(\delta,m)$ for fixed $\delta$. For $n\in\BB N$, define $\ep_n=(n!)^{-6}$. Let $k\in\BB N$ be the smallest integer such that $\ep_{k}<\delta$. For $m\geq 3$, $\delta\in(0,m^{-1})$, and $n\in\BB N$, define $D=D_{\delta,m}$ and $D_n=D_{\delta,m,n}$ by
\eqbn
\begin{split}
D&= \left\{\BB t = (t_0 , \dots ,t_{m-2}) \in (0,1)^{m-1 } \,:\,t_{2j-2},t_{2j}<t_{2j-1},\,|t_i-t_j|>\delta,\,\forall i,j \in \{0,\dots,m-2\}
 \right\},\\
D_{n}&=D\cap (\ep_n\BB Z) .
\end{split}
\eqen
We will consider $\delta$, $k$, $m$, $\kappa'$ as fixed throughout the rest of the section, and all implicit constants and decay rates might depend on these constants.

\subsection{Upper bound}
\label{sec-multiple-pt1}

The upper bound is based on an estimate for the probability that an $(m-1)$-dimensional vector is a so-called approximate cone vector, where approximate cone vectors are defined such that appropriately defined neighborhoods around each vector cover all actual cone vectors. The probability for a vector to be an approximate cone vector is calculated by considering two (approximately independent) types of events: one type of event concerning the local behavior of $(L,R)$ near the approximate cone times, and one type of event concerning the behavior of $(L,R)$ in the time interval between the cone times.

Let $C>1$ and $0<r\ll 1$. For $j\in \{0,\dots,m-2\}$, $n\in\BB N$ and $\BB t \in D_n$, let
\[
 A_{n,j,C,r}(\BB t) =
  \begin{cases}
   \left\{L_{t_j+s}\geq L_{t_j}-C\ep_n^{\frac 12-r},\,R_{t_j+s}\geq R_{t_j}-C\ep_n^{\frac 12-r},\, \forall s\in(0, \ep_{k+m+1})\right\} & \text{if } j>0\text{ even,} \\
   \left\{L_{t_j-s}\geq L_{t_j}-C\ep_n^{\frac 12-r},\,R_{t_j-s}\geq R_{t_j}-C\ep_n^{\frac 12-r},\, \forall s\in(0, \ep_{k+m+1})\right\} & \text{if } j\text{ odd,}\\
   \left\{L_{t_j+s}\geq L_{t_j}-C\ep_n^{\frac 12-r},\, \forall s\in(0, \ep_{k+m+1})\right\} & \text{if } j=0.\\
  \end{cases}
\]
For $j\in\{1,2,\ldots,\lfloor(m-1)/2\rfloor\}$, $n\in\BB N$ and $\BB t\in D_n$, let $B_{n,j,C,r}(\BB t)$ be the event that  
\begin{itemize}
\item $|X_{t_{2j-1}}-X_{t^*}|\leq C\ep_n^{\frac 12-r}$ and $\inf_{s\in[t^*,t_{2j-1}]}X_s\geq X_{t_{2j-1}}-C\ep_n^{\frac 12-r}$ for each $(X,t^*) \in \{(L,t_{2j-2}),(R,t_{2j}) \}$ if $2j+1<m$; or
\item $|L_{t_{2j-1}}-L_{t_{2j-2}}|\leq C\ep_n^{\frac 12-r}$ and $\inf_{s\in[t_{2j-2},t_{ 2j-1}]}L_s\geq L_{t_{2j-1}}-C\ep_n^{\frac 12-r}$ if $2j+1=m$.
\end{itemize}
Recall that $t_{2j-2},t_{2j}<t_{2j-1}$ for all $j\in\{1,2,\ldots,\lfloor(m-1)/2\rfloor\}$, hence the intervals considered above are well-defined.

Define
\eqbn
E_n(\BB t) = \, \left(\bigcap_{j=0}^{m-2} A_{n,j,C,r}(\BB t)\right)\cap \left(\bigcap_{j=1}^{\lfloor (m-1)/2 \rfloor} B_{n,j,C,r}(\BB t)\right).
\eqen
For $\BB t\in D_{n}$ we say that $\BB t$ is an \emph{$n$-approximate $(m-2)$-tuple cone vector} if the event $E_n(\BB t)$ occurs.

The event $A_{n,j,C,r}(\BB t)$ occurs when there is an approximate cone time for $Z$ or the time-reversal of $Z$ at time $t_j$. The events $B_{n,j,C,r}(\BB t)$ ensure that the $L$ or $R$ coordinate of two pairs of approximate cone times $t_{2j-1},t_{2j-2}$ and $t_{2j-1},t_{2j}$ are approximately identical.

We will need the following two lemmas both for the proof of the upper bound and for the proof of the lower bound of Theorem~\ref{prop-multiplept-1}. The first lemma is \cite[equation (4.3)]{shimura-cone}.
\begin{lem}
Let $L,R$ be correlated Brownian motions satisfying~\eqref{eq-multiplept-10}. For any $t\in\BB R$ and $\ep>0$,
\eqbn
\BB P\big(X_{s}\geq X_{t}-\ep^{\frac 12},\,\forall s\in [t,t+1],X=L,R\big) \asymp \ep^{\frac{\pi}{2\theta}},
\eqen
with the implicit constant depending only on $\theta$.
\label{prop-multiplept-3}
\end{lem}

Our second lemma will be used to estimate the conditional probabilities of the events $B_{n,j,C,r}(\BB t)$ given the other events we are interested in. We condition on $(L,R)$ restricted to intervals near each time $t_i$, and the event $\wt H_n(M)$ is introduced to ensure that the path of $(L,R)$ restricted to these intervals is not too irregular, and that $(L,R)$ does not violate the conditions of the event $B_{n,j,C,r}(\BB t)$ in these intervals.

\begin{lem}
Let $s\leq \ep_{k+1}$. For $j\in \{1,2, \ldots,\lfloor (m-1)/2\rfloor\}$ let $\mcl F_j$ be the $\sigma$-algebra generated by
\begin{enumerate}
\item $(L_t,R_t)$ for $ t\in(0, t_{2j-1}-2s)$,
\item $(L_t,R_t)-(L_{t_i},R_{t_{i}})$ for $i\neq 2j-1$ and $t\in (t_i-\ep_{k+1},t_i+\ep_{k+1})$, and
\item $(L_t,R_t)-(L_{t_{2j-1}},R_{t_{2j-1}})$ for $t\in (t_{2j-1}-s,t_{2j-1}+\ep_{k+1})$.
\end{enumerate}
Let $\{H_n\}_{n\in\BB N}$ be a sequence of events measurable with respect to $\mcl F_j$. For $M>0$ and $n\in\BB N$ let $\wt H_n(M)$ be the event that the following is true.
\begin{itemize} 
\item $Ms^{\frac 12}>X_{t_{2j-1}-s}-X_{t_{2j-1}}>M^{-1}s^{\frac 12}$ and $Ms^{\frac 12}> X_{t_{2j-1}-2s}-X_{t^*}>M^{-1}s^{\frac 12}$ for each $(X,t^*)\in \{(L,t_{2j-2}),(R,t_{2j}) \}$; and
\item $\inf_{t'\in[t_{2j-1}-s,t_{2j-1}]} X_{t'} \geq X_{t_{2j-1}}-C\ep_n^{\frac 12-r}$ and $\inf_{t'\in[t^*,  t_{2j-1}-2s]} L_{t'} \geq L_{t^*}-C\ep_n^{\frac 12-r}$ for each $(X,t^*)\in \{(L,t_{2j-2}),(R,t_{2j}) \}$ .
\end{itemize}
Assume there is a constant $M>0$ independent of $n$ and $\BB t$ such that the conditional probability of $\wt H_n(M)$ given $H_n$ is at least $M^{-1}$ for all $n\in\BB N$ and $\BB t\in D_n$. Then 
\eqb
\BB P\!\left(B_{n,j,C,r}(\BB t)\,\big|\,H_n\right) \asymp\left\{
  \begin{array}{l}
    \ep_n^{1-2r} \qquad\text{ for } 2j+1<m,\\
    \ep_n^{\frac 12-r} \qquad\text{ for } 2j+1=m,
  \end{array}
  \right.
\label{eq-multiplept-23}
\eqe
with the implicit constants depending only on $\delta,C,m,\theta,s$ and $M$. 

If $s=\ep_l$ for $l\in \{k+1,\dots,n-1\}$, and $M$ can be chosen independently of $l,n$ and $\BB t$, then 
\eqb
\BB P\!\left(B_{n,j,C,r}(\BB t)\,\big|\,H_n\right) \asymp\left\{
  \begin{array}{l}
    \ep_n^{1-2r}\ep_l^{-1+o_n(1)} \,\,\,\,\,\text{ for } 2j+1<m,\\
    \ep_n^{\frac 12-r}\ep_l^{-\frac 12+o_n(1)} \,\,\,\,\,\text{ for } 2j+1=m,
  \end{array}
  \right.
  \label{eq-multiplept-24}
\eqe
with the implicit constant depending only on $\delta,C,m,\theta$ and $M$. 
\label{prop-multiplept-9}
\end{lem}
\begin{proof}
First we will prove~\eqref{eq-multiplept-23}. Let $2j+1<m$. In order for $B_{n,j,C,r}(\BB t)$ to occur, we must have
\begin{equation}
\begin{split}
&|(X_{t_{2j-1}}-X_{t_{2j-1}-s})
+(X_{t_{2j-1}-s}-X_{t_{2j-1}-2s})
+(X_{t_{2j-1}-2s}-X_{t^*})|
\leq C\ep_n^{\frac 12-r}  \\
&\qquad \qquad \forall (X,t^*) \in \{(L,t_{2j-2}),(R,t_{2j}) \}
\end{split}
\label{eq-multiplept-17}
\end{equation}
and 
\eqb
\inf_{s\in[t^*,t_{2j-1}]} X_s\geq X_{t_{2j-1}}
-C\ep_n^{\frac 12-r},\qquad
\forall (X,t^*) \in \{(L,t_{2j-2}),(R,t_{2j}) \}. 
\label{eq-multiplept-21}
\eqe
The first and third terms in the left-hand side of~\eqref{eq-multiplept-17} is measurable with respect to $\mathcal F_j$, while the second term is independent of $\mathcal F_j$. The second term is a normally distributed random variable with variance of order~$s$. The probability of~\eqref{eq-multiplept-17} conditioned on $\mcl F_j$ is therefore equal to the probability that two jointly Gaussian random variables with variance of order $s$ take values in two given intervals of length $2C\ep_n^{\frac 12-r}$. If $2j+1=m$, the same result holds, only with one Gaussian random variable, instead of two Gaussian random variables. By the upper bounds in the first event defining $\wt H_n(M)$ the estimate~\eqref{eq-multiplept-23} follows with~\eqref{eq-multiplept-17} instead of $B_{n,j,C,r}(\BB t)$ on the left-hand side.

To complete the proof of~\eqref{eq-multiplept-23} we need to show that~\eqref{eq-multiplept-21} happens with uniformly positive probability conditioned on $H_n$ and~\eqref{eq-multiplept-17}. This follows by using that $L_{t_{2j-1}-s}-L_{t_{2j-1}}$, $L_{t_{2j-1}-2s}-L_{t_{2j-2}}$, and the corresponding quantities for $R$, have a macroscopic magnitude on the event $\wt H_n(M)$.

The estimate~\eqref{eq-multiplept-24} follows by small modifications of the argument above using Brownian scaling. We only consider the case $2j+1<m$, since the case $2j+1=m$ is similar. Again we need two jointly Gaussian random variables of variance $s$ to take values in bounded intervals of length $2C\ep_n^{1/2}$, and the upper bounds in the first event defining $\wt H_n(M)$ imply that this probability is of order $\ep_n^{1-2r}\ep_l^{-1}$. The event $\wt H_n(M)$ also ensures that~\eqref{eq-multiplept-21} occurs with uniformly positive probability conditioned on $H_n$ and~\eqref{eq-multiplept-17}. 
\end{proof}
The following lemma implies the upper bound of Theorem~\ref{prop-multiplept-1}:
\begin{lem}
For any $n\in\BB N$ and $\BB t\in D_{n}$,
\eqbn
\BB P\big( E_n(\BB t) \big)\asymp \ep^{\frac {1}{2}+(m-2)(\frac{\pi}{2\theta}+\frac 12)-cr}
\eqen
where $c>0$ is a constant depending only on $\delta,C,m$ and $\theta$.
\label{prop-multiplept-4}
\end{lem}

\begin{proof}
Since $|t_i-t_j|>\delta$ for any two $i,j\in \{0,\dots,m-2\}$, the Markov property of Brownian motion implies that the events $A_{n,j,C,r}(\BB t)$ for $j\in \{0,\dots,m-2\}$ are independent. By Lemma~\ref{prop-multiplept-3} we have $\BB P(A_{n,j,C,r}(\BB t))\asymp \ep_n^{\frac{\pi}{2\theta}(1-2r)}$ for $j=1, \ldots ,m-2$, and $\BB P(A_{n,0,C,r}(\BB t))\asymp \ep_n^{\frac 12-r}$. The events $B_{n,j,C,r}(\BB t)$ are not independent of each other and of the events $(A_{n,i,C,r}(\BB t))_{i=0,\ldots,m-2}$, but as we will see in the remainder of the proof we have sufficient independence to obtain a good estimate for conditional probabilities.

Let $j\in \{0,\dots,\lfloor (m-1)/2\rfloor \}$. Since $t_{2j-2},t_{2j}<t_{2j-1}$ by the definition of $D$, the event
\eqb
\left(\bigcap_{i=0}^{m-2}A_{n,i,C,r}(\BB t)\right) \cap \left(\bigcap_{i :t_{2i-1}<t_{2j-1}}B_{n,i,C,r}(\BB t)\right)
\label{eq-multiplept-11}
\eqe 
is measurable with respect to the $\sigma$-algebra $\mathcal F_j$ of Lemma~\ref{prop-multiplept-9} with $s=\ep_{k+m+1}$. There is a constant $M>0$ independent of $\BB t$ and $n$ such that the event $\wt H_n(M)$ of Lemma~\ref{prop-multiplept-9} is satisfied with probability at least $M^{-1}$ conditioned on the event~\eqref{eq-multiplept-11}. Therefore~\eqref{eq-multiplept-23} of Lemma~\ref{prop-multiplept-9} implies
\eqbn
\BB P\!\left(B_{n,j,C,r}(\BB t)\,\big|\,\left(\bigcap_{i=0}^{m-2} A_{n,i,C,r}(\BB t)\right)\cap \left(\bigcap_{i: t_{2i-1}<t_{2j-1}} B_{n,i,C,r}(\BB t)\right)\right) \asymp\left\{
  \begin{array}{l}
    \ep_n^{1-2r} \,\,\,\,\,\text{ for } 2j+1<m,\\
    \ep_n^{\frac 12-r} \,\,\,\,\,\text{ for } 2j+1=m.
  \end{array}
  \right.
\eqen
The estimate of the lemma follows by the above estimates and the observation that
\eqbn
\begin{split}
\BB P\big( E_n(\BB t) \big)&=\prod_{j=0}^{m-2}\BB P(A_{n,j,C,r}(\BB t))\times \prod_{j=1}^{\lfloor (m-1)/2 \rfloor} \BB P\!\left(B_{n,j,C,r}(\BB t)\,\big|\,\left(\bigcap_{i=0}^{m-2} A_{n,i,C,r}(\BB t)\right)\cap \left(\bigcap_{i: t_{2i-1}<t_{2j-1}} B_{n,i,C,r}(\BB t)\right)\right)\\
&\asymp \ep_n^{(m-2)\frac{\pi}{2\theta}}\times \ep_n^{\frac 12}\times \ep_n^{(m-2)\frac 12}\times \ep_n^{-cr}.
\end{split} 
\eqen
\end{proof}

\begin{proof}[Proof of upper bound in Theorem~\ref{prop-multiplept-1}]
For $n\in\BB N$ and $\BB t\in D_n$, let $S_n(\BB t)$ denote the $(m-1)$-dimensional cube of side length $2 \ep_n$ centered at $\BB t$, and note that $D\subset \cup_{\BB t\in D_{n}} S_n(\BB t)$. The set of all $\BB t\in D_{n}$ such that $E_n(\BB t)$ occurs is denoted by $D^*_n$. Let $\mcl A_{C,r}$ be the event that $(L,R)$ is $(\frac 12-r)$-H\"older continuous with H\"older norm at most $C/2$. Assume $\BB s\in \mcl T(\delta,m)$, and let $\BB t\in D_n$ be the element of $D_n$ that minimizes $\|\BB s-\BB t\|$, where $\|\cdot\|$ denotes e.g.\ the $L^\infty$ norm. If $\mcl A_{C,r}$ occurs, then $\BB t$ is an $n$-approximate $(m-2)$-tuple cone vector. It follows that
\eqb
\BB 1_{\mcl A_{C,r}}\mathcal T(\delta,m)\subset 
\bigcup_{\BB t\in D^*_{n}}S_n(\BB t).
\label{eq-multiplept-1}
\eqe

First assume $m\leq \frac{2\kappa'-4}{\kappa'-4}$, and let $d>\frac{1}{2}-(m-2)(\frac{\pi}{2\theta}-\frac 12)$. By Lemma~\ref{prop-multiplept-4},
for any $C>0$ and sufficiently small $r$,
\eqbn
\begin{split}
\BB E\bigg(\BB 1_{\mcl A_{C,r}}\sum_{\BB t\in D^*_{n}} \text{diam}(S_n(\BB t))^d\bigg) 
\preceq \sum_{t\in D_{n}} \ep_n^d\BB P(E_n(\BB t))
= \ep_n^{-(m-1)}\times \ep_n^{d} \times \ep_n^{\frac {1}{2}+(m-2)(\frac{\pi}{2\theta}+\frac 12)+cr}\rightarrow 0
\end{split}
\eqen 
as $n\rightarrow\infty$. Chebyshev's inequality and the Borel-Cantelli lemma together imply that a.s.
\eqb
\lim_{n\rta\infty}\BB 1_{\mcl A_{C,r}}\sum_{\BB t\in D^*_{n}} \text{diam}(S_n(\BB t))^d =0.
\label{eq-multiplept-9}
\eqe
 By~\eqref{eq-multiplept-1} the set $\{S_n(\BB t)\}_{\BB t\in D^*_{n,}}$ gives a cover for $\mcl T(\delta,m)$ on the event ${\mcl A_{C,r}}$, hence $\dim_{\mcl H} (\mcl T(\delta,m)) \leq d$ a.s.\ if ${\mcl A_{C,r}}$ occurs. Since $\BB P(\mcl A_{C,r})\rightarrow 1$ as $C \rightarrow \infty$, we a.s.\ have $\dim_{\mcl H} (\mcl T(\delta,m)) \leq d$.

By definition $\mathfrak T(m)=\text{Proj}_1(\mcl T(m))$, where Proj$_1:\BB R^{m-1}\rightarrow\BB R$ is the projection $\BB t\mapsto t_1$. Since coordinate projection is Lipschitz continuous, $\dim_{\mcl H}(\mathfrak T(m))\leq \dim_{\mcl H}(\mcl T(m))$. Since $\dim_{\mcl H}(\mcl T(m)) = \sup_{\delta\in\BB Q,\delta\in(0,m^{-1})} \dim_{\mcl H}(\mcl T(m,\delta))$, we obtain the desired upper bound by letting $d\rightarrow \frac 12 - (m-2)(\frac{\pi}{2\theta}-\frac 12)$, and using that $\dim_{\mcl H}(\wt{\mcl T}(m))=\dim_{\mcl H}({\mcl T}(m))$ and $\dim_{\mcl H}(\wt{\frk T}(m))=\dim_{\mcl H}({\frk T}(m))$.

Now we will consider the case $m>\frac{2\kappa'-4}{\kappa'-4}$. Note that~\eqref{eq-multiplept-9} still holds in this case, and that we can choose $d=0$. When $d=0$, the left hand side of~\eqref{eq-multiplept-9} counts the number of elements in $D^*_{n}$, and it follows that $D^*_{n}$ is empty for all sufficiently large $n$. By~\eqref{eq-multiplept-1} we can conclude that $\mcl T(\delta,m)$, hence $\wt{\mcl T}(m)$ and $\wt{\frk T}(m)$, are empty.
\end{proof}

\subsection{Lower bound}
\label{sec-multiple-pt2}
We will now prove the lower bound of Theorem~\ref{prop-multiplept-1}. The proof will be by standard methods, and relies on an estimate for the correlation of the two events that $\BB t$ and $\BB s$ are approximate cone vectors, see Proposition~\ref{prop-multiplept-5}. In order to obtain sufficient independence of these two events, we will work with a slightly modified definition of approximate cone vectors. Let $\BB t\in D_n$. For $j\in\{0,1,\ldots,m-2\}$ define the events
\eqbn
A_{n,j}(\BB t)=\,A_{n,j,1,0}(\BB t),\qquad B_{n,j}(\BB t)=B_{n,j,1,0}(\BB t),
\eqen
\[
F_{j}(\BB t) =
\begin{cases}
\Big\{
\underset{s\in[0,\ep_l]}{\inf} (X_{t_j}-X_{t_j-s})\in\ep_l^{\frac 12}(-3\log l,-l^{-3/2}),\,\forall l\geq k+1,\, 
\forall (X,t^*) \in \{(L,t_{2j-2}),(R,t_{2j}) \}\Big\},\,j>0\ \mathrm{even},\\
\Big\{
\underset{s\in[0,\ep_l]}{\inf}(X_{t_j+s}-X_{t_j})\in\ep_l^{\frac 12}(-3\log l,-l^{-3/2}),\,\forall l\geq k+1,\, 
\forall (X,t^*) \in \{(L,t_{2j-2}),(R,t_{2j}) \}\Big\},\ j\mathrm{\,odd},\\
\Big\{\underset{s\in[0,\ep_l]}{\inf}(L_{t_j}-L_{t_j-s})\in\ep_l^{\frac 12}(-3\log l,-l^{-3/2}),\,\forall l\geq k+1\Big\},\,j=0,
\end{cases}
\]
and for $j\in\{1,2,\ldots,\lfloor (m-1)/2\rfloor\}$, define the event
\eqbn
\begin{split}
B^*_{n,j}(\BB t)=&\,\left\{\inf_{s\in[t^*+\ep_n,t_{2j-1}-\ep_n]} X_s> X_{t_{2j-1}}+3\ep_n^{\frac 12},
\,\forall (X,t^*) \in \{(L,t_{2j-2}),(R,t_{2j}) \} \right\}.
\end{split}
\eqen
Also define the event $G(\BB t)$ by
\eqbn
G(\BB t) =\,\left\{R_{t_0+2\ep_{k+2}}<\inf_{t\in[t_0-\ep_{k+2},t_0+\ep_{k+2}]}R_t-\ep_{k+2}^{\frac 12}\right\},
\eqen 
and finally define the event $\wt E_n(\BB t)$ by
\eqbn
\wt E_n(\BB t) = \left(\bigcap_{j=0}^{m-2}A_{n,j}(\BB t)\right) \cap \left(\bigcap_{j=1}^{\lfloor(m-1)/2 \rfloor}B_{n,j}(\BB t)\right) \cap \left(\bigcap_{j=1}^{\lfloor(m-1)/2 \rfloor}B^*_{n,j}(\BB t)\right) \cap 
\left(\bigcap_{j=0}^{m-2} F_j(\BB t) \right)\cap G(\BB t).
\eqen
We say that $\BB t\in D_n$ is a \emph{perfect $n$-approximate $(m-2)$-tuple cone vector} if the event $\wt E_n(\BB t)$ occurs. Let $D^*_{n,\mcl P}$ denote the set of all $\BB t\in D_{n}$ such that $\wt E_n(\BB t)$ occurs.

The set $\mcl T_{\mcl P}(\delta,m)$ of perfect $(m-2)$-tuple cone vectors, is defined by 
\eqbn 
\mcl T_{\mcl P}(\delta,m):=
\bigcap_{k\geq 1} \overline{\left(\bigcup_{n\geq k}\bigcup_{\BB t\in D^*_{n,\mcl P}}S_n(\BB t) \right)}.
\eqen

As we will see below, the events $B^*_{n,j}(\BB t)$ imply that if $\BB s\neq\BB t$, both events $\wt E_n(\BB t)$ and $\wt E_n(\BB s)$ can only happen if we have $t_j<s_j$ for all even $j$ and $t_j>s_j$ for all odd $j$, or vice-versa. The events $F_j(\BB t)$ will imply that both events $\wt E_n(\BB t)$ and $\wt E_n(\BB s)$ can only happen if all the elements of the vector $\BB t-\BB s$ are of approximately the same magnitude, and the regularity condition $G(\BB t)$ will imply that both events $\wt E_n(\BB t)$ and $\wt E_n(\BB s)$ cannot happen if $|t_i-s_j|$ is very small for some $i\neq j$.

\begin{lem}
The set of perfect $(m-2)$-tuple cone vectors is contained in the set of $(m-2)$-tuple cone vectors, i.e., 
\eqbn
\mathcal T_{\mcl P}(\delta,m)\subset \mcl T(\delta,m).
\eqen
\label{prop-multiplept-7}
\end{lem}
\begin{proof}
Assume $\BB t\not\in\mcl T(\delta,m)$. Then at least one of the following conditions are satisfied: (i) there is a $j\in\{0,...,m-1 \}$ such that $t_j$ is not a cone time for $Z$ or for the time-reversal of $Z$, (ii) there is an even (resp.\ odd) $j\in\{1,\ldots,m-2\}$ such that $v(t_j)\neq t_{j\pm 1}$ or $u(t_j)\neq t_{j\pm 1}$ (resp.\ $v^R(t_j)\neq t_{j\pm 1}$ or $u^R(t_j)\neq t_{j\pm 1}$), or (iii) $t_0\not\in \{u(t_1),v(t_1)\}$. For each $n\in\BB N$ let $\BB s_n\in D_n$ be a vector such that $\|t-\BB s_n\|$ is minimized. In either case (i)-(iii) continuity of $L$ and $R$ imply that $\wt E_n(\BB s)$ cannot occur for sufficiently large $n$. Hence $\BB t\not\in \mcl T_{\mcl P}(\delta,m)$.  
\end{proof}

\begin{figure}[ht!]
\begin{center}
\includegraphics[scale=0.92]{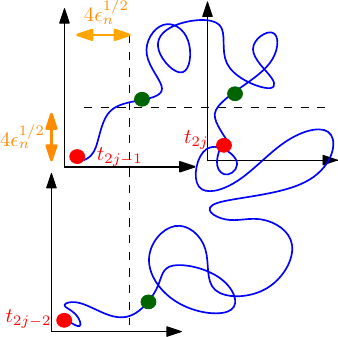}\,\,\,\,\,\,\,\,\,\,
\includegraphics[scale=0.92]{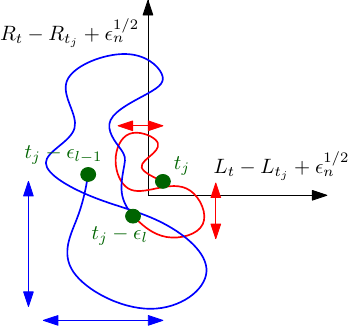}
\end{center}
\caption{The left figure illustrates regularity condition $B^*_{n,j}(\BB t)$. The three green dots correspond to times $t_{2j-2}+\ep_n$, $t_{2j}+\ep_n$ and $t_{2j-1}-\ep_n$, respectively. We want the curve to be bounded away from the boundary of the cone for most of the cone excursion, hence preventing approximate cone vectors $\BB s$ with cone excursions that are partially inside and partially outside the corresponding excursion of $\BB t$. For example, if $\wt E_n(\BB t) \cap \wt E_n(\BB s)$ occurs, $\BB t\neq \BB s$, and $t_{j}<s_{j}<s_{j+1}$, we want $s_{j+1}< t_{j+1}$. The right figure illustrates regularity condition $F_j(\BB t)$. The red curve is $Z|_{[t_j-\ep_l,t_j]}$, and the blue curve is $Z|_{[t_j-\ep_{l-1},t_j]}$. The vertical red arrow shows the absolute value of $\inf_{s\in [t_j-\ep_l,t_j]} (R_s-R_{t_j})$, and the vertical blue arrow shows the absolute value of $\inf_{s\in [t_j-\ep_{l-1},t_j]} (R_s-R_{t_j})$; the horizontal arrows show the same values for $L$. Regularity condition $F_j(\BB t)$ implies that the modulus of these infima decrease in a certain way as we increase $l$. Hence, if both events $\wt E_n(\BB t)$ and $\wt E_n(\BB s)$ occur we know the approximate value of $R_{t_j}-R_{s_j}$ and $L_{t_j}-L_{s_j}$ in terms of $t_j-s_j$. This will help us establish that all elements of the vector $\BB t-\BB s$ are of approximately the same order whenever both $\BB t$ and $\BB s$ are approximate perfect cone vectors.}
\label{fig-multiplept2}
\end{figure}

The following lemma combined with Proposition~\ref{prop-multiplept-5} will imply the lower bound of Theorem~\ref{prop-multiplept-1}.
\begin{lem}
For any $n\in\BB N$ and $\BB t\in D_{n}$,
\eqbn
\BB P(\wt E_n(\BB t)) \asymp \ep_n^{\frac 12 + (m-2)(\frac{\pi}{2\theta}+\frac 12)}
\eqen
with the implicit constants depending only on $\delta,m$ and $\theta$.
\label{prop-multiplept-6}
\end{lem}
\begin{proof}
Let $N_1$ and $N_2$ be normal random variables with correlation $\cos(\theta)$ and variance $a$, with $a$ as in~\eqref{eq-multiplept-10}. For any $l\geq k+1$ and $j\in\{2,...,2\lfloor (m-2)/2\rfloor \}$ even,
\eqbn
\begin{split}
\BB P\Big(\inf_{s\in[0,\ep_l]}&(R_{t_j}-R_{t_j-s})\not\in\ep_l^{\frac 12}(-3\log l,-l^{-\frac 32})\,\text{ or } \inf_{s\in[0,\ep_l]}(L_{t_j}-L_{t_j-s})\not\in\ep_l^{\frac 12}(-3\log l,-l^{-\frac 32})\Big) \\
&= \BB P\Big(|N_1|\not\in (l^{-\frac 32},3\log l) \text{ or } |N_2|\not\in (l^{-\frac 32},3\log l)\Big)\\
&\asymp l^{-\frac 32}.
\end{split}
\eqen 
By using this estimate, a similar estimate for odd $j$, and independence of the events $F_j(\BB t)$, the Borel-Cantelli lemma implies that $\bigcap_{j=0}^{m-2} F_j(\BB t)$ happens with positive probability.

The events $A_{n,j}(\BB t)$ are independent of each other and of the events $F_j(\BB t)$, and $\BB P(A_{n,j}(\BB t))\asymp\ep_n^{\frac{\pi}{2\theta}}$ by Lemma~\ref{prop-multiplept-3}. The event $G(\BB t)$ is independent of $F_j(\BB t)$ and $A_{n,j}(\BB t)$ for $j>0$. Conditioned on $F_0(\BB t)\cap A_{n,0}(\BB t)$ the event $G(\BB t)$ has uniformly positive probability, since the value of $R_{t_0+2\ep_{k+2}}-R_{t_0+\ep_{k+2}}$ is independent of $F_0(\BB t)\cap A_{n,0}(\BB t)$.

For $j=1,2,\ldots,\lfloor(m-1)/2\rfloor$, let $B_{n,j}^{*,\op{loc}}(\BB t)$ be the event that the following is true.
\begin{itemize}
\item $\inf_{s\in [t_{2j-1}-\ep_{k+m+1},t_{2j-1}-\ep_n]}X_s> X_{t_{2j-1}}+3\ep_n^{\frac 12}$ and $X_{t_{2j-1}-\ep_{k+m+1}}-X_{t_{2j-1}}>6\ep_{k+m+1}^{\frac 12}$ for each $(X,t^*)\in \{(L,t_{2j-2}),(R,t_{2j})\}$. 
\item $\inf_{s\in[t^*+\ep_n,t^*+\ep_{k+m+1}]}X_s> X_{t^*}+3\ep_n^{\frac 12}$ and $X_{t^*+\ep_{k+m+1}}-X_{t^*}>6\ep_{k+m+1}^{\frac 12}$ for each $(X,t^*)\in \{(L,t_{2j-2}),(R,t_{2j})\}$.  
\end{itemize}
Note that the occurrence of $B^{*,\op{loc}}_{n,j}(\BB t)$ implies that the conditions defining $B^*_{n,j}(\BB t)$ are satisfied near $t_{2j-2}$, $t_{2j-1}$ and $t_{2j}$. The event $B^{*,\op{loc}}_{n,j}(\BB t)$ is independent of $F_i(\BB t)$ for all $i$, since the event $F_i(\BB t)$ depends on the behavior of $Z$ right before (resp.\ after) $t_i$ for $i$ even (resp.\ odd), while the event $B^{*,\op{loc}}_{n,j}(\BB t)$ depends on the behavior of $Z$ right before (resp.\ after) $t_{2j-1}$ (resp.\ $t_{2j-2}$ and $t_{2j}$). The event $B^{*,\op{loc}}_{n,j}(\BB t)$ is also independent of the events $A_{n,i}(\BB t)$ for $i\neq 2j-2,2j-1,2j$. Furthermore $B^{*,\op{loc}}_{n,j}(\BB t)$ is independent of $G(\BB t)$ for $j>1$, and the probability of $B^{*,\op{loc}}_{n,1}(\BB t)$ changes only by a constant order factor when conditioning on $G(\BB t)$. By Brownian scaling,
\eqbn
\begin{split}
&\BB P\!\left(
\inf_{s\in[t_{2j}+\ep_n,t_{2j}+\ep_{k+m+1}]}
R_s> R_{t_{2j}}+3\ep_n^{\frac 12},\,R_{t_{2j}+\ep_{k+m+1}}-R_{t_{2j}}>6\ep_{k+m+1}^{\frac 12}\,|\,A_{n,2j}(\BB t)\right)\\
&\qquad\qquad\succeq\BB P\!\left(
\inf_{\begin{subarray}{l}
        s\in[t_{2j}+\ep_n,\\
        t_{2j}+\ep_{k+m+1}]
      \end{subarray}}
R_s> R_{t_{2j}}+3\ep_n^{\frac 12},\,R_{t_{2j}+\ep_{k+m+1}}-R_{t_{2j}}>6\ep_{k+m+1}^{\frac 12}\,|\,
\inf_{\begin{subarray}{l}
        s\in[t_{2j}+\ep_n,\\
        t_{2j}+\ep_{k+m+1}]
      \end{subarray}}
R_s\geq R_{t_{2j}}-\ep_n^{\frac 12}\right)\\
&\qquad\qquad\succeq 1,
\end{split}
\eqen
and similar estimates hold for the other three conditions in the definition of $B^{*,\op{loc}}_{n,j}(\BB t)$. It follows that
\eqbn
\BB P\!\left(
\left(\bigcap_{j=0}^{m-2}   A_{n,j}(\BB t) \cap   F_j(\BB t)  \right)\cap G(\BB t)\cap
\left(\bigcap_{j=1}^{\lfloor (m-1)/2 \rfloor} B_{n,j}^{*,\op{loc}}(\BB t)\right)\right) \asymp \ep_n^{\frac 12+(m-2)\frac{\pi}{2\theta}}.
\eqen
By~\eqref{eq-multiplept-23} of Lemma~\ref{prop-multiplept-9} with $s=\ep_{k+m+1}$ we get further
\eqb
\BB P\!\left(
\left(\bigcap_{j=0}^{m-2}  A_{n,j}(\BB t) \cap   F_j(\BB t)   \right)\cap G(\BB t) \cap
\left(\bigcap_{j=1}^{\lfloor(m-1)/2\rfloor} B_{n,j}^{*,\op{loc}}(\BB t) \cap B_{n,j}(\BB t)\right) 
\right) \asymp \ep_n^{\frac 12+(m-2)(\frac{\pi}{2\theta}+\frac 12)}.
\label{eq-multiplept-16}
\eqe
It remains to show that the same estimate holds if we replace $B_{n,j}^{*,\op{loc}}(\BB t)$ by $B_{n,j}^*(\BB t)$. It is straightforward to obtain the estimate with $\preceq$ and $B_{n,j}^*(\BB t)$ instead of $B_{n,j}^{*,\op{loc}}(\BB t)$, since $\wt E(\BB t)\subset E(\BB t)$ for $C=1,r=0$, and we have the estimate of Lemma~\ref{prop-multiplept-4}. To obtain~\eqref{eq-multiplept-16} with $\succeq$, and with $B_{n,j}^*(\BB t)$ instead of $B_{n,j}^{*,\op{loc}}(\BB t)$, we need to show that the inequalities of $B_{n,j}^*(\BB t)$ hold with uniformly positive probability also at a macroscopic distance from the cone times $t_{2j-2},t_{2j-1},t_{2j}$, conditioned on the event of~\eqref{eq-multiplept-16}. Note that all the events we condition on, except for $B_{n,j}$, only concern the behavior of the curve near the approximate cone times, while $B_{n,j}$ says that the $L$ and $R$ coordinate of pairs of cone times are close, and that the curve is above this $L$ or $R$ coordinate in the time interval between the cone times. The event $B_{n,j}^{*,\op{loc}}(\BB t)$ ensures that $R_{t_{2j}+\ep_{k+m+1}},R_{t_{2j-1}-\ep_{k+m+1}},L_{t_{2j-2}+\ep_{k+m+1}}$ and $L_{t_{2j-1}-\ep_{k+m+1}}$ have a macroscopic distance from $R_{t_{2j}},R_{t_{2j-1}},L_{t_{2j-2}}$ and $L_{t_{2j-1}}$, respectively. The occurrence of $B^*_{n,j}$, conditioned on the event of~\eqref{eq-multiplept-16}, therefore corresponds to the event that two approximate Brownian bridges of duration of order $1$ between two given pairs of points at (possibly different) height of order $1$, are always larger than $3\ep_n$. This happens with probability $\succeq 1$ uniformly in $n$ and $\BB t$. Hence the conditional probability of $B^*_{n,j}$ given the event of~\eqref{eq-multiplept-16} is positive uniformly in $n$ and $\BB t$.
\end{proof}

\begin{lem}
Let $\BB t,\BB s\in D_n$, $\BB t\neq\BB s$, and assume the event $\wt E_n(\BB t)\cap \wt E_n(\BB s)$ occurs for some $n\geq m+k$. Then at least one of the two following conditions are satisfied:
\begin{enumerate}
\item[(I)] All $i,j\in \{0,\dots,m-2\}$ satisfy $|s_i-t_j|\geq \ep_{k+m}$.
\item[(II)] There is an $l\in(k,n-1)$ such that $|s_j-t_j|\in [\ep_{l+m+1},\ep_{l+2}]$ for all $j\in\{0,\ldots,m-2\}$. We furthermore have either $s_j<t_j$ for all odd $j$ and $t_j<s_j$ for all even $j$, or $s_j<t_j$ for all even $j$ and $t_j<s_j$ for all odd $j$.
\end{enumerate}
\label{prop-multiplept-8}
\end{lem}

\begin{proof}
If (I) is not satisfied, at least one of the following conditions must be satisfied: (i) $0<\min_j|s_j-t_j|<\ep_{k+m}$, (ii) $s_j=t_j$ for some $j$, or (iii) $|s_j-t_i|<\ep_{k+m}$ for some $i\neq j$. We will show that (i) implies (II), and that (ii) and (iii) cannot occur, hence complete the proof of the lemma.

Case (i): Assume (i) holds, but (I) does not hold. Furthermore, assume the first condition of (II) does not hold, i.e., there is no $l>k$ such that $|s_j-t_j|\in [\ep_{l+m+1},\ep_{l+2}]$ for all $j$. Then we can find $i,j\in\{0,\ldots,m-2\}$ and $l'\geq k+2$ such that $|s_j-t_j|<\ep_{l'+1}$, $|s_{i}-t_{i}|>\ep_{l'}$, and $|i-j|=1$. Assume w.l.o.g.\ that $j$ is even, $i=j+1$, and $t_{j}< s_{j}<s_{j+1}$; all other cases can be treated similarly. 

We claim that $s_{j+1}<t_{j+1}$. Assume the opposite, i.e., $t_{2j'+1}< s_{2j'+1}$ for $j=2j'$. Note that this implies $t_{2j'+1}\in(s_{2j'},s_{2j'+1})$, since $t_{2j'}<t_{2j'+1}$, $|s_{2j'}-t_{2j'}|<\ep_{l'+1}$ and $|t_{2j'+1}-t_{2j'}|>\delta$. Regularity condition $B^*_{n,j'+1}(\BB t)$ gives a contradiction to $B_{n,j'+1}(\BB s)$, since $s_{2j'}\in(t_{2j'},t_{2j'+1})$, and
\eqbn
L_{s_{2j'}}\leq L_{s_{2j'+1}}+\ep_n^{\frac 12}<L_{t_{2j'+1}}-2\ep_n^{\frac 12}.
\eqen
This implies the claim.

Since $0<\min_{j'}|s_{j'}-t_{j'}|<\ep_{l'+1}$, we have $n>l'+1$, which implies $\ep_n< \ep_{l'+1}$.
Observe that the event $B_{n,j}(\BB t)$, hence $\wt E_n(\BB t)$, cannot occur if $\wt E_n(\BB s)$ occurs, since
\eqbn
\inf_{s\in[t_j,t_{j+1}]} L_s
\leq \inf_{s\in[s_{j+1},s_{j+1}+\ep_{l'}]} L_s
< \inf_{s\in[s_j-\ep_{l'+1},s_j]} L_s - 2\ep_{l'+1}^{\frac 12} + (L_{s_{j+1}}-L_{s_j})
< \inf_{s\in[s_j-\ep_{l'+1},s_j]} L_s - \ep_{n}^{\frac 12}
\leq L_{t_j}-\ep_n^{\frac 12},
\eqen
where the second inequality follows from $F_j(\BB s)$ and $F_{j+1}(\BB s)$. We have obtained a contradiction, hence there is an $l>k$ such that $|s_j-t_j|\in [\ep_{l+m+1},\ep_{l+2}]$ for all $j$. Note that $l+2\leq n$, since $0<\min_{j'}|s_{j'}-t_{j'}|\leq\ep_{l+2}$. By the same argument as when deriving $s_{j+1}<t_{j+1}$ above, regularity conditions $B^*_{n,j'+1}(\BB t)$ and $B^*_{n,j'+1}(\BB s)$ for $j'\in\{1,\ldots,\lfloor(m-1)/2\rfloor\}$ imply that $s_j<t_j$ for all odd $j$, and $t_j<s_j$ for all even $j$, or vice versa.

Case (ii): Since $\BB t\neq\BB s$, we can find $j$ such that $t_j=s_j$ and either $s_{j-1}\neq t_{j-1}$ or $s_{j+1}\neq t_{j+1}$. Assume w.l.o.g.\ that $j=2j'$ is even, and that $s_{j+1}> t_{j+1}$. By definition of $D$ we have $t_{j+1}\in(s_j,s_{j+1})$. We get a contradiction from the regularity condition $B^*_{n,j'+1}(\BB s)$, since
\eqb
L_{t_{2j'}}=L_{s_{2j'}}\leq L_{s_{2j'+1}}+\ep_n^{\frac 12}<L_{t_{2j'+1}}-2\ep_n^{\frac 12}\leq L_{t_{2j'}}-\ep_n^{\frac 12}.
\eqe

Case (iii): Assume w.l.o.g.\ that $j>i$. By $F_j(\BB t)$, $F_j(\BB s)$, $F_i(\BB t)$ and $F_i(\BB s)$, $i$ and $j$ have the same parity. By the same argument as in case (ii), we have $s_{i}\neq t_{j}$. By the same argument as for Case (i), and by induction on $l$, we have $|s_{i-l}-t_{j-l}|<\ep_{k+m-l}$ for all $l\leq i$, hence $|s_0-t_{j-i}|<\ep_{k+2}$. By the regularity condition $G(\BB s)$, the event $A_{n,j-i}(\BB t)$, hence $\wt E_n(\BB t)$, cannot occur, conditioned on $\wt E_n(\BB s)$. This gives us the desired contradiction, hence case (iii) cannot occur.
\end{proof}

\begin{prop}
\label{prop-multiplept-5} 
For any $\BB s,\BB t\in D_{n}$, we have
\eqb
\label{eq-multiplept-7}
\BB P(\wt E_n(\BB t)\cap \wt E_n(\BB s)) \leq 
\|\BB t-\BB s\|^{-\frac 12-(m-2)(\frac{\pi}{2\theta}+\frac 12)+o_{\|\BB t-\BB s\|}(1)}
\BB P(\wt E_n(\BB t))\BB P(\wt E_n(\BB s)).
\eqe
\end{prop}
\begin{proof}
We will prove the assertion separately for the two cases of Lemma~\ref{prop-multiplept-8}.

Case (I): The $2(m-1)$ events $A_{n,j}(\BB t)$ and $A_{n,j}(\BB s)$ are all independent, and by Lemma~\ref{prop-multiplept-3}, 
\eqbn
\BB P(A_{n,j}(\BB t))=\BB P(A_{n,j}(\BB s))\asymp \ep_n^{\frac{\pi}{2\theta}} \text{ for }j>0,\qquad
\BB P(A_{n,0}(\BB t))=\BB P(A_{n,0}(\BB s))\asymp \ep_n^{\frac{1}{2}}.
\eqen
Let $(B_i)_{i=1,\ldots,2\lfloor (m-1)/2\rfloor }$ denote an ordering of the events $B_{n,j}(\BB t)$, $B_{n,j}(\BB s)$, where $B_{j'}(\BB t')$ comes before $B_{j''}(\BB t'')$, $\BB t',\BB t''\in\{\BB t,\BB s \}$, if $t'_{2j'-1}<t''_{2j''-1}$. By~\eqref{eq-multiplept-23} of Lemma~\ref{prop-multiplept-9} with $s=\ep_{k+m+1}$,
\eqbn
\begin{split}
\BB P\!\left(B_j\,|\,\left(\bigcap_{i=0}^{m-2}A_{n,i}(\BB t)\right)\cap\left(\bigcap_{i=0}^{m-2}A_{n,i}(\BB s)\right)\cap\left(\bigcap_{i=1}^{j-1} B_i\right)\right)
\asymp\left\{
  \begin{array}{l}
    \ep_n \text{ if $B_j=B_{n,i}(\BB t)$ or $B_j=B_{n,i}(\BB s),\,2i+1<m$},\\
    \ep_n^{\frac 12} \text{ if $B_j=B_{n,i}(\BB t)$ or $B_j=B_{n,i}(\BB s),\,2i+1=m$}.
  \end{array}
  \right.
\end{split}
\eqen
It follows that
\eqbn
\begin{split}
\BB P(\wt E_n(\BB s)\cap \wt E_n(\BB t) )
\leq& \, \prod_{j=0}^{m-2} \BB P(A_{n,j} (\BB t) )
\times \prod_{j=0}^{m-2} \BB P(A_{n,j} (\BB s))\\
&\times \prod_{j=1}^{2\lfloor (m-1)/2\rfloor} \BB P\!\left(B_j \,|\,\left(\bigcap_{i=0}^{m-2}A_{n,i}(\BB t)\right)\cap\left(\bigcap_{i=0}^{m-2}A_{n,i}(\BB s)\right)\cap\left(\bigcap_{i=1}^{j-1} B_i\right)\right) \\
\asymp&\, \ep_n^{1+2(m-2)(\frac{\pi}{2\theta}+\frac 12)}\\
\asymp&\, \BB P(\wt E_n(\BB s)) \BB P( \wt E_n(\BB t) ).
\end{split}
\eqen

Case (II): Assume w.l.o.g.\ that $s_1<t_1$, which implies $s_j<t_j$ for all odd $j$ and $t_j<s_j$ for all even $j$. Since $\ep_{l+m+1}=\ep_l^{o_l(1)}$, and by Lemma~\ref{prop-multiplept-3}, Brownian scaling and the Markov property of Brownian motion, we have 
\eqbn
\BB P(A_{n,j}(\BB t)\,|\,A_{n,j}(\BB s)) =
\left\{
  \begin{array}{ll}
    \ep_l^{-\frac{\pi}{2\theta}+o_l(1)}\ep_n^{\frac{\pi}{2\theta}} \,\,\,\,\,&\text{ for } j>0,\\
    \ep_l^{-\frac 12+o_l(1)}\ep_n^{\frac 12} \,\,\,&\text{ for } j=0.
  \end{array}
  \right.
\eqen
By using this estimate, $\BB P(A_{n,j}(\BB s))\asymp\ep_n^{\frac{\pi}{2\theta}}$ for $j>0$, 
$\BB P(A_{n,0}(\BB s))\asymp\ep_n^{\frac{1}{2}}$, and 
that $A_{n,j}(\BB s)$ and $A_{n,j}(\BB t)$ are independent of $A_{n,i}(\BB s)$ and $A_{n,i}(\BB t)$ for $i\neq j$, we get further
\eqb
\BB P\!\left(\left(\bigcap_{i\text{ even}}A_{n,i}(\BB t)\right)\cap\left(\bigcap_{i=0}^{m-2}A_{n,i}(\BB s)\right)\right) = \ep_l^{-\frac 12-\frac{\pi}{2\theta}\lfloor(m-2)/2\rfloor+o_l(1)}
\ep_n^{1+\frac{\pi}{2\theta}(m-2+\lfloor(m-2)/2\rfloor)}.
\label{eq-multiplept-14}
\eqe
By~\eqref{eq-multiplept-23} of Lemma~\ref{prop-multiplept-9} with $s=\ep_{k+m+1}$, for $j \in \{1,\dots,\lfloor (m-1)/2\rfloor\}$ we have
\eqb
\begin{split}
&\BB P\!\left(B_{n,j}(\BB s)\,|\,
\left(\bigcap_{i=0}^{m-2}A_{n,i}(\BB s)\right)\cap \left(
\bigcap_{\begin{subarray}{l}
        i\text{ even or}\\
        i>2j-1
      \end{subarray}}
A_{n,i}(\BB t)
\right)\cap
\left(\bigcap_{i\,:\,t_{2i-1}>t_{2j-1}}B_{n,i}(\BB t)\cap B_{n,i}(\BB s)\right)
\right) \\
&\qquad\qquad\qquad\qquad\qquad\qquad\qquad\qquad\qquad\qquad\qquad\
\asymp\left\{
  \begin{array}{ll}
    \ep_n \,\,&\text{ for } 2j+1<m,\\
    \ep_n^{\frac 12} &\text{ for } 2j+1=m.
  \end{array}
  \right.
  \label{eq-multiplept-15}
  \end{split}
\eqe
Next we will show that
\eqb
\begin{split}
&\BB P\!\left(B_{n,j}(\BB t)\cap A_{n,2j-1}(\BB t)\, |\,
\left(\bigcap_{i=0}^{m-2}A_{n,i}(\BB s)\right)\cap \left(
\bigcap_{\begin{subarray}{l}
        i\text{ even or}\\
       	i>2j-1
      \end{subarray}}
A_{n,i}(\BB t)
\right)\cap
\left(\bigcap_{i\,:\,t_{2i-1}>t_{2j-1}}B_{n,i}(\BB t)\cap B_{n,i}(\BB s)\right)
\cap B_{n,j}(\BB s)
\right)\\
&\qquad\qquad\qquad\qquad\qquad\qquad\qquad\qquad\qquad\qquad\qquad\
\preceq
\left\{
  \begin{array}{ll}
    \ep_n^{\frac{\pi}{2\theta}+1}\ep_l^{-(\frac{\pi}{2\theta}+1)+o_l(1)} \,\,\,\,\,&\text{ for } 2j+1<m,\\
    \ep_n^{\frac{\pi}{2\theta}+\frac 12}\ep_l^{-(\frac{\pi}{2\theta}+\frac 12)+o_l(1)} \,\,\,&\text{ for } 2j+1=m.
  \end{array}
  \right.  
  \end{split}
  \label{eq-multiplept-13}
\eqe
Let
\eqbn
\wt{A}_{n,2j-1}(\BB t)= \left\{X_{t_{2j-1}-s}\geq X_{t_{2j-1}}-\ep_n^{\frac 12},\, \forall s\in(0,\ep_{l+m+1}/2),\,\forall X=\{L,R\} \right\},
\eqen
and note that $A_{n,2j-1}(\BB t)\subset \wt{A}_{n,2j-1}(\BB t)$. By Lemma~\ref{prop-multiplept-3} 
\eqb 
\BB P(\wt{A}_{n,2j-1})= \ep_l^{-\frac{\pi}{2\theta}+o_l(1)}\ep_n^{\frac{\pi}{2\theta}}.
\label{eq-multiplept-18}
\eqe
Let $s=\ep_{l+m+1}/2$, and note that $t_{2j-1}-s_{2j-1}\geq 2s$. Since both $\wt{A}_{n,2j-1}(\BB t)$ and the events we condition on in~\eqref{eq-multiplept-13} are measurable with respect to the $\sigma$-algebra generated by
\begin{enumerate}
\item $(L_t,R_t)$ for $ t\leq t_{2j-1}-2s$ and
\item $(L_t,R_t)-(L_{t_{2j-1}},R_{t_{2j-1}})$ for $t\in (t_{2j-1}-s,t_{2j-1})$,
\end{enumerate}
the estimate~\eqref{eq-multiplept-24} from Lemma~\ref{prop-multiplept-9} implies that
\eqbn
\begin{split}
&\BB P\!\left(B_{n,j}(\BB t)\,|\,
\left(\bigcap_{i=0}^{m-2}A_{n,i}(\BB t)\right)\cap \left(
\bigcap_{\begin{subarray}{l}
        i\text{ even or}\\
       	i>2j-1
      \end{subarray}}
A_{n,i}(\BB s)
\right)\cap
\left(\bigcap_{i>j}B_{n,i}(\BB t)\cap B_{n,i}(\BB s)\right)
\cap B_{n,j}(\BB s)\cap
\wt{A}_{n,2j-1}(\BB t)\right)\\
&\qquad\qquad\qquad\qquad\qquad\qquad\qquad\qquad\qquad\qquad\qquad
= \left\{
  \begin{array}{ll}
    \ep_n    \ep_l^{-1  +o_l(1)} \,\,\,\,\,&\text{ for } 2j+1<m,\\
    \ep_n^{\frac 12}\ep_l^{-\frac 12+o_l(1)} \,\,\,    &\text{ for } 2j+1=m.
  \end{array}
  \right.  
\end{split}
\eqen
By using this estimate and~\eqref{eq-multiplept-18}, we obtain~\eqref{eq-multiplept-13}. By multiplying equations~\eqref{eq-multiplept-15} and~\eqref{eq-multiplept-13}, taking the product over $j\in\{1,\ldots,\lfloor(m-1)/2\rfloor\}$, and then multiplying by~\eqref{eq-multiplept-14}, we get
\eqbn
\begin{split}
\BB P(\wt E_n(\BB t)\cap \wt E_n(\BB s)) &\leq 
\, \ep_l^{-\frac 12-(m-2)(\frac{\pi}{2\theta}+\frac 12)+o_l(1)}\times \ep_n^{1+2(m-2)(\frac{\pi}{2\theta}+\frac 12)}\\
 &=\, 
 \|\BB t-\BB s\|^{-\frac 12-(m-2)(\frac{\pi}{2\theta}+\frac 12)+o_{\|\BB t-\BB s\|}(1)}
 \BB P(\wt E_n(\BB t))\BB P(\wt E_n(\BB s)).
\end{split}
\eqen
\end{proof}

{\em Proof of lower bound in Theorem~\ref{prop-multiplept-1}:} 
By Lemma~\ref{prop-multiplept-6}, Proposition~\ref{prop-multiplept-5} and \cite[Proposition~4.8]{mww-nesting} we have
\eqb
\BB P(\text{dim}_{\mcl H}(\mcl T_{\mcl P}(\delta,m))\geq \wt d)>0, 
\label{eq-multiplept-20}
\eqe
for any $\wt d<\frac 12-(m-2)(\frac{\pi}{2\theta}-\frac 12)$. \cite[Proposition~4.8]{mww-nesting} is stated for $\BB D\subset\BB C$ instead of $D\subset\BB R^{m-1}$, but the proof carries through in exactly the same manner if considering the domain $D$ instead, only with 2 replaced by $(m-1)$ in the statement of the proposition. Also, \cite[Proposition~4.8]{mww-nesting} is stated for events that are defined for all points in the domain, while our events $\wt E_n(\BB t)$ are only defined for $\BB t\in D_n$. We obtain new events $E'_n(\BB t)$ defined for all $\BB t\in D$ as follows. Given $\BB t\in D$ we let $\BB s\in D_n$ be the vector such that $\|\BB t-\BB s\|$ is minimized, with some (unspecified) rule to break ties. We define $E'_n(\BB t)$ to be the event such that $E'_n(\BB t)$ occurs if and only if $\wt E_n(\BB s)$ occurs. 

By the scale invariance and the Markov property of Brownian motion, the event of~\eqref{eq-multiplept-20} almost surely occurs. By letting $\wt d\rightarrow\frac 12-(m-2)(\frac{\pi}{2\theta}-\frac 12)$ it follows that, almost surely,
\eqb
\text{dim}_{\mcl H}(\mcl T_{\mcl P}(\delta,m))\geq \frac 12-(m-2)\left(\frac{\pi}{2\theta}-\frac 12\right). 
\label{eq-multiplept-4}
\eqe
By Lemma~\ref{prop-multiplept-7} and $\mcl T(\delta,m)\subset \wt{\mcl T}(m)$ we get a lower bound for one of the two considered sets of Theorem~\ref{prop-multiplept-1}:
\eqbn
\text{dim}_{\mcl H}(\wt{\mcl T}(m))\geq \frac 12-(m-2)\left(\frac{\pi}{2\theta}-\frac 12\right).
\eqen

Let Proj$_1:\BB R^{m-1}\to\BB R$ be projection $\BB t\mapsto t_1$, and define ${\frk{T}}_{\mcl P}(\delta,m):=\text{Proj}_1(\mcl T_{\mcl P}(\delta,m))$. We will argue that a.s. 
\eqb
\dim_{\mcl H} \!\left(\mcl T_{\mcl P}(\delta,m)\right)\leq \dim_{\mcl H} \!\left(\mathfrak T_{\mcl P}(\delta,m)\right).
\label{eq-multiplept-3}
\eqe
Assume $d,s>0$, $s\ll 1$, satisfy $d-s>\text{dim}_{\mcl H}({\frk{T}}_{\mcl P}(\delta,m))$. Given any $\ep>0$, we can find a sequence of intervals $(I_i)_{i\in\BB N}$, such that ${\frk{T}}_{\mcl P}(\delta,m)\subset\cup_{i\in\BB N}I_i$ and $\sum_{i\in\BB N} |I_i|^{d-s}<\ep$. Let $(l_i)_{i\in\BB N}$ be such that $|I_i|\in(\ep_{l_i+1},\ep_{l_i})$. 

Let $\BB s,\BB t\in \mcl T_{\mcl P}(\delta,m)$, $\BB t\neq\BB s$. One of the two cases (I) or (II) of Lemma~\ref{prop-multiplept-8} must hold. Therefore we can find a sequence $(\wt I_i)_{i\in\BB N}$ of rectangles $\wt I_i\subset\BB R^{m-1}$, such that Proj$_1(\wt I_i)=I_i$, diam$(\wt I_i)\preceq \ep_{l_i-m+2}=|I_i|^{1+o_{|I_i|}(1)}$, and $\mcl T_{\mcl P}(\delta,m)\subset\cup_{i\in\BB N} \wt I_i$. We have $\sum_{i\in\BB N}\text{diam}(\wt I_i)^d\leq \sum_{i\in\BB N}|I_i|^{d-s}<\ep$ when $\max_i|I_i|$ is sufficiently small, hence for all sufficiently small $\ep$. It follows that $\text{dim}_{\mcl H}(\mcl T_{\mcl P}(\delta,m))\leq d$, and by letting $s\rightarrow 0$ and $d\rightarrow \text{dim}_{\mcl H}({\frk{T}}_{\mcl P}(\delta,m))$, we see that~\eqref{eq-multiplept-3} holds almost surely. The lower bound for the second of the considered sets of Theorem~\ref{prop-multiplept-1}, follows by
\eqbn
     \text{dim}_{\mcl H}\big(\wt{\frk{T}}_{}      (       m)\big)
\geq \text{dim}_{\mcl H}\big(   {\frk{T}}_{\mcl P}(\delta,m)\big)
\geq \text{dim}_{\mcl H}\big(   {\mcl{T}}_{\mcl P}(\delta,m)\big)
\geq \frac 12-(m-2)\left(\frac{\pi}{2\theta}-\frac 12\right).
\eqen

\section{Open questions}
\label{sec-problems}

In this section we will list some open problems relating to the results of this paper. 

\begin{enumerate}
\item Let $Z$ be a correlated Brownian motion as in~\eqref{eq-multiplept-10}, run for positive time, and let $\wh X$ be the set of times $t\geq 0$ which are not contained in any left cone interval for $Z$. The Hausdorff dimension of $\wh X$ is computed in a very indirect manner in Example~\ref{prop-dimcalc-gasket}. Can one obtain this dimension directly? If so, one would obtain a new proof of the dimension of the gasket of $\op{CLE}_{\kappa'}$ for $\kappa' \in (4,8)$.  
 
\item Let $\kappa'\in (4,8)$ and let $\Gamma$ be a $\op{CLE}_{\kappa'}$ in a simply connected domain $D\subset \BB C$. We define the \emph{thin gasket} of $\Gamma$ is to be the set $\mcl{T}$ of points in $D$ which are not disconnected from $\partial D$ by any loop in $\Gamma$.  Equivalently, $\mcl T$ is the closure of the union of the outermost outer boundaries of the CLE loops, where the outer boundary of a loop $\ell$ is the boundary of the set of points disconnected from $\partial D$ by $\ell$. The thin gasket differs from the ordinary gasket of~\cite{msw-gasket,ssw09} in that the ordinary gasket includes points which are disconnected from $\partial D$ by some loop in $\Gamma$ but which are not actually surrounded by this loop. What is the a.s.\ Hausdorff dimension of $\mcl T$? 

One can make a reasonable guess as to what this dimension should be, as follows. Suppose we take as our ansatz that the quantum scaling exponent $\Delta$ of $\mcl T$ in the KPZ formula is a linear function of~$\kappa'$. This assumption seems reasonable, as it is satisfied with $\mcl T$ replaced by the ordinary gasket, the $\op{SLE}_{\kappa'}$ or $\op{SLE}_{16/\kappa'}$ curve, and the double and cut points of SLE. By SLE duality, as $\kappa' \rta 8$, the outer boundaries of CLE loops start to look like $\op{SLE}_2$ curves. Therefore we should expect $\dim_{\mcl H}\mcl T \rta 5/4$ as $\kappa' \rta 8$. On the other hand, as $\kappa' \rta 4$, the thin gasket starts to look like the ordinary gasket, so by the results of \cite{nacu-werner11,ssw09}, we should expect $\dim_{\mcl H}\mcl T \rta 15/8$ as $\kappa'\rta 4$.  These guesses lead to the prediction that $\Delta = \tfrac{\kappa'}{16}$ which, in turn, by the KPZ formula this yields the prediction $\frac{5}{32}(16-\kappa')$ for the dimension.

In the peanosphere setting, consider the restriction of the whole-plane $\op{CLE}_{\kappa'}$ associated with the curve $\eta'$ to some bubble $U$ disconnected from $\infty$ by $\eta'$. This restriction has the law of a $\op{CLE}_{\kappa'}$ in $U$ and its thin gasket can be described by an explicit (but rather complicated) functional of the Brownian motion $Z = (L, R)$. Such a description is implicit in~\cite{gwynne-miller-cle} since this paper describes the set of points disconnected from $\partial U$ by each $\op{CLE}_{\kappa'}$ loop in $U$. One could obtain $\dim_{\mcl H} \mcl T$ by computing the Hausdorff dimension of this Brownian motion set and applying Theorem~\ref{thm-dim-relation}. Alternatively, one could attempt to make rigorous an argument of the sort given in~\cite[Appendix~B]{wedges}. As a third possibility, one could take a direct approach, possibly using imaginary geometry~\cite{ig1,ig2,ig3,ig4} to regularize the events in the two-point estimate as in~\cite{miller-wu-dim,gms-mf-spec}. 

\item Can one describe various multifractal quantities associated with the $\op{SLE}_\kappa$ curve in terms of the Brownian motion $Z = (L,R)$ in the peanosphere construction of~\cite{wedges}, such as the multifractal spectrum~\cite{gms-mf-spec,dup-mf-spec-bulk}, the winding spectrum~\cite{binder-dup-winding1,binder-dup-winding2}, various notions of higher multifractal spectra~\cite{dup-higher-mf}, the multifractal spectrum at the tip~\cite{lawler-viklund-tip}, the optimal H\"older exponent~\cite{lind-holder,lawler-viklund-holder}, or the integral means spectrum~\cite{bel-smirnov-hm-sle,gms-mf-spec}? If so, can these quantities be computed using Theorem~\ref{thm-dim-relation} or some variant thereof? 
\end{enumerate}

\bibliography{cibibshort,cibib,cibib1} 
\bibliographystyle{hmralphaabbrv}

\end{document}